\documentclass[12pt,reqno]{amsart}
\pdfoutput=1
\usepackage[headings]{fullpage}
\usepackage{amssymb,amsmath,bbm,tikz}
\usepackage{graphicx,color,enumerate, float}
\usepackage{calligra}
\usepackage[all,cmtip]{xy}
\usepackage{url}
\usepackage[bookmarks=true,
    colorlinks=true,%
    linkcolor=blue,%
    citecolor=blue,%
    filecolor=blue,%
    menucolor=blue,%
    urlcolor=blue,%
    breaklinks=true]{hyperref}

\usepackage{verbatim}
\newtheorem{theorem}{Theorem}[section]
\theoremstyle{definition}
\newtheorem{proposition}[theorem]{Proposition}
\newtheorem{lemma}[theorem]{Lemma}
\newtheorem{definition}[theorem]{Definition}
\newtheorem{remark}[theorem]{Remark}

\newtheorem{conjecture}[theorem]{Conjecture}

\newtheorem{example}[theorem]{Example}

\def\red#1{\textcolor[rgb]{1.00,0.00,0.00}{[#1]}}%

\def\BZ{\mathbbm Z}
\def\BQ{\mathbbm Q}
\def\BR{\mathbbm R}

\def\la{\langle}
\def\ra{\rangle}

\def\sgn{\mathrm{sgn}}
\def\be{\begin{equation}}
\def\ee{\end{equation}}
\newcommand{\Sk}{\mathcal{S}}
\def\js{\mathrm{js}}
\def\jx{\mathrm{jx}}
\newcommand*{\pdash}{\rule[0.5ex]{1.5em}{1.5pt}}
\newcommand{\jwproj}{\vcenter{\hbox{\includegraphics[scale=.1]{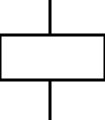}}}}
\def\TL{\mathrm{TL}}

\def\tw{\mathrm{tw}}
\def\bs{\mathrm{bs}}
\def\R{\mathbb{R}}
\def\rv{\mathcal{r}}
\def\TR{$\mathsf{TR}$}
\def\oo{{\mathrm{o}}}

\graphicspath{{figures/}}

\begin{document}

\title{The slope conjecture for Montesinos knots}

\author{Stavros Garoufalidis}
\address{Max Planck Institute for Mathematics \\
       Vivatsgasse 7, 53111 Bonn, GERMANY \newline
        {\tt \url{http://people.mpim-bonn.mpg.de/stavros}}}
\email{stavros@mpim-bonn.mpg.de}
\author{Christine Ruey Shan Lee}
\address{Department of Mathematics \\
         University of South Alabama \\
         Mobile, AL 36608, USA \newline 
         {\tt \url{https://sites.google.com/a/southalabama.edu/crslee/home}}}
\email{crslee@southalabama.edu}
\author{Roland van der Veen}
\address{Mathematisch Intstituut \\
         Leiden University, Leiden \\
         Niels Bohrweg 1\\
         The Netherlands \newline 
         {\tt \url{http://www.rolandvdv.nl}}}
\email{r.i.van.der.veen@math.leidenuniv.nl}
\thanks{
1991 {\em Mathematics Classification.} Primary 57N10. Secondary 57M25.
}

\date{\today} 

\begin{abstract}
The slope conjecture relates the degree of the colored Jones polynomial
of a knot to boundary slopes of essential surfaces. 
We develop a general approach that matches a state-sum formula for the colored Jones polynomial with
the parameters that describe surfaces in the complement.
We apply this to Montesinos knots proving the slope conjecture for Montesinos knots, with some restrictions.
\end{abstract}

\maketitle

{\footnotesize
\tableofcontents
}

\section{Introduction} 
\label{sec.intro}

\subsection{The slope conjecture and the case of Montesinos 
knots}

The slope conjecture relates one of the most important knot invariants,  
the colored Jones polynomial, to essential surfaces in the knot complement \cite{Ga:slope}. 
More precisely, the growth of the degree as a function of the color determines boundary slopes. 
Understanding the topological information that the polynomial detects in the knot is a central problem in quantum topology. The conjecture suggests the polynomial can be studied through surfaces, which are fundamental objects in 3-dimensional topology. 

Our philosophy is that the connection follows from a deeper correspondence between terms in an expansion of the polynomial and surfaces. This would potentially lead to a purely topological definition of quantum invariants. The coefficients of the polynomial should count isotopy classes of surfaces, much like in the case of the 3D-index \cite{DDG13}. As a first test of this principle, we focus on the slope conjecture for Montesinos knots. In this case Hatcher-Oertel~\cite{HO} provides a description of the set of essential surfaces of those knots. In particular they give an effective algorithm to compute the set of boundary slopes of incompressible and $\partial$-incompressible surfaces in the complement of such knots.

We provide a state-sum formula for the colored Jones polynomial that allows us to match the parameters of the terms of the sum that contribute to the degree of the polynomial with the parameters that describe the locally essential surfaces. The key innovation of our state sum is that we are able to identify those terms that actually contribute to the degree. The resulting degree function is piecewise-quadratic, allowing application of quadratic integer programming methods.

We interpret the curve systems formed by intersections with essential surfaces on a Conway sphere enclosing a rational tangle in terms of these degree-maximizing skein elements in the state sum. In this paper we carry out the matching for Montesinos knots but the state-sum \eqref{eq.sum} is valid in general. In fact using this framework, one could determine the degree of the colored Jones polynomial and find candidates for corresponding essential surfaces in many new cases beyond Montesinos knots.

While the local theory works in general, fitting together the surfaces in each tangle to obtain a (globally) essential surface has yet to be done. The behavior of the colored Jones polynomial under gluing of tangles has similar patterns, which may be explored in future work. 

The Montesinos knots, together with some well-understood 
algebraic knots, are knots that have small Seifert fibered 2-fold branched 
covers~\cite{Montesinos,Zieschang}. For our purposes, we will not use this
abstract definition, and instead construct Montesinos links by inserting rational tangles into pretzel knots. 
More precisely, a Montesinos link is the closure of a list of rational tangles arranged as in 
Figure~\ref{f.mtangle} and concretely as in Figure~\ref{f.montesinos}. See Definition \ref{d.Montesinos}.

\begin{figure}[!htpb]
\centering
\def \svgwidth{.4\columnwidth}
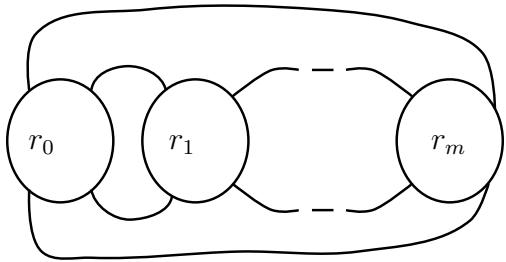
\caption{A Montesinos link.} 
\label{f.mtangle}
\end{figure}

\begin{figure}[!htpb]
\centering
\def \svgwidth{.4\columnwidth}
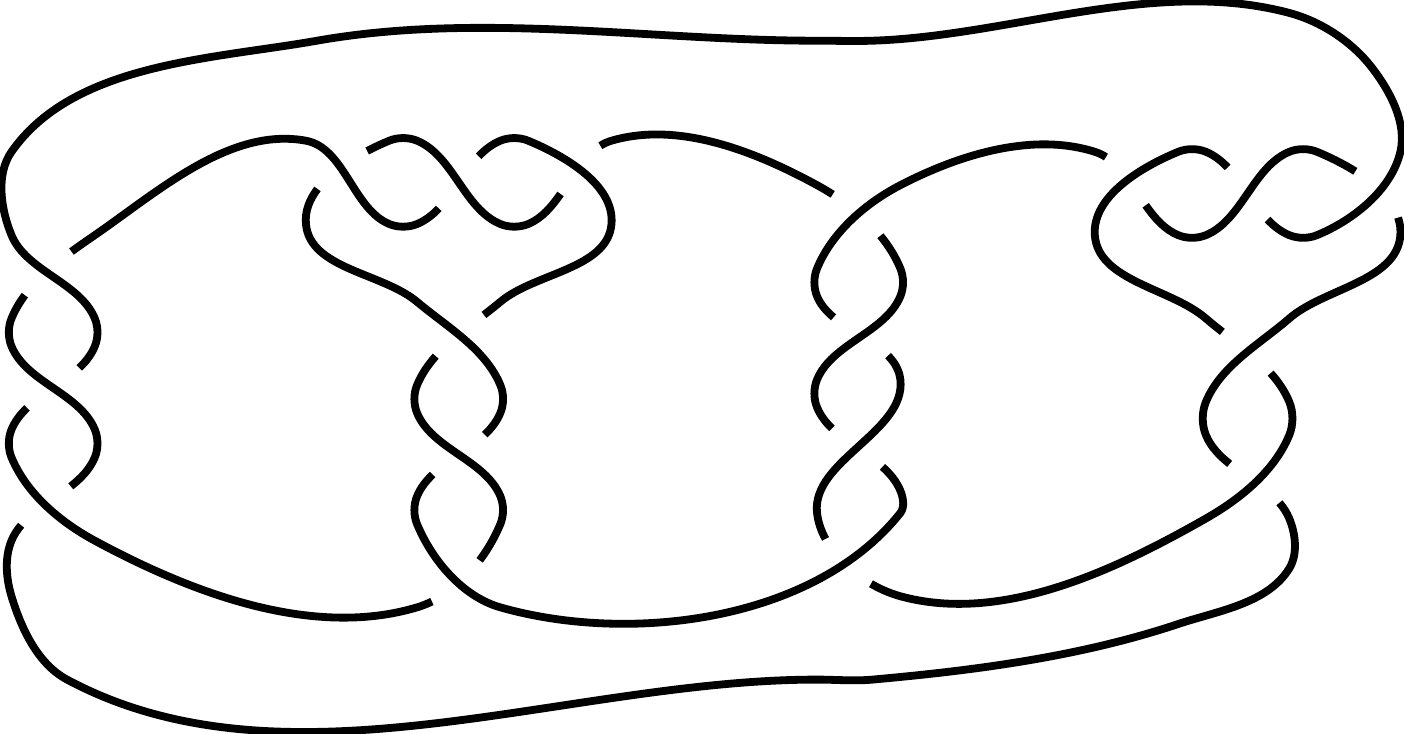
\caption{The Montesinos link
$K(-\frac{1}{3}, -\frac{3}{10}, \frac{1}{4}, \frac{3}{7})$.} 
\label{f.montesinos}
\end{figure}

Rational tangles are determined by rational 
numbers, see Section \ref{sec.RT}, thus a Montesinos link 
$K(r_0,r_1,\dots,r_m)$ is encoded by a list of rational numbers
$r_i \in \BQ$. Note that $K(r_0, r_1, \ldots, r_{m})$ is a knot if and only if either there is only one even denominator, or, there is no even denominator and the number of odd numerators is odd. When $r_i=1/q_i$ is the inverse of an integer, the Montesinos link
$K(1/q_0,\dots,1/q_m)$ is also known as the pretzel link $P(q_0,\dots,q_m)$.


\subsection{Our results}
\label{sub.results}

Recall the colored Jones polynomial $J_{K,n}(v) \in\BZ[v^{\pm 2}]$ of a 
knot $K$ colored by the $n$-dimensional irreducible 
representation of $\mathfrak{sl}_2$~\cite{Tu}. See Definition \ref{defn:cjp}. 
Our variable $v$ for the colored Jones polynomial is
related to the skein theory variable $A$ \cite{Pr91} and to the Jones variable 
$q$~\cite{Jo} by $v=A^{-1}=q^{-\frac{1}{4}}$. With our conventions,
if $3_1=P(1,1,1)$ denotes the left-hand trefoil, then 
$J_{3_1,2}(v)= v^{18}-v^{10}-v^{6}-v^2$.
For the $n$-colored unknot we get $J_{O,n} = \frac{v^{2n}-v^{-2n}}{v^{2}-v^{-2}}$.

Let $\delta_K(n)$ denote the maximum $v$-degree of the colored Jones 
polynomial $J_{K,n}(v)$. It follows that
$\delta_K(n)$ is a quadratic quasi-polynomial~\cite{Ga2}. In other words, for every knot $K$ there exists an $N_K\in\mathbb{N}$ such that for $n>N_K$:
\be
\delta_K(n) = \js_K(n) n^2 + \jx_K(n) n + c_K(n)  
\ee
where $\js_K$, $\jx_K$, and $c_K$ are periodic functions. 

\begin{conjecture} (The strong slope conjecture)\\
\label{con.ss}
For any knot $K$ and any $n>N_K$, there is an $n'$ and an essential surface $S \subset S^3\setminus K$ with $|\partial S|$ boundary components,  such that the boundary slope of $S$ equals  $\js_K(n)=p/q$ (reduced to lowest terms and with the assumption $q>0$), and $\frac{2\chi(S)}{q|\partial S|} = \jx_K(n')$. 
\end{conjecture}

The number $q|\partial S|$ is called the \emph{number of sheets} of $S$, denoted by $\#S$, and $\chi(S)$ is the Euler characteristic of $S$. See the discussion at the beginning of Section \ref{sec.incompressible} for the definition of an essential surface and boundary slope. 
We call a value of the function $\js_K$ a \emph{Jones slope} and a value of the function $\jx_K$ a \emph{normalized
Euler characteristic}. The original slope conjecture is the part of Conjecture \ref{con.ss} that concerns the interpretation of $\js_K$ as boundary slopes \cite{Ga:slope}, while the rest of the statement is a refinement by \cite{KT15}. The reader may consult these two sources  \cite{Ga:slope}, \cite{KT15} for additional background.
By considering the mirror image $\overline{K}$ of $K$ and the formula $J_{K,n}(v^{-1}) = J_{\overline{K},n}(v)$, the strong slope conjecture is equivalent to the statement in \cite{KT15} that includes the behavior of the minimal degree.

The slope conjecture and the strong slope conjecture were established for many knots including alternating knots, adequate 
knots, torus knots, knots with at most 9 crossings, 2-fusion knots (in this case only the slope conjecture is proven), graph 
knots, near-alternating knots, and most 3-tangle pretzel knots and 
3-tangle Montesinos knots~\cite{Ga:slope,FKP,GV, LV:3pretzel, MT, BMT18, L, LLY,Howie}. 
However the general case remains intractable and most proofs simply
compute the quantum side and the topology side separately, comparing only the end results.

Since the strong slope conjecture is known for adequate knots 
\cite{Ga:slope, FKP, FKP13}, we will ignore the Montesinos knots which are 
adequate. When $m\geq 2$, a non-adequate Montesinos knot $K(r_0, r_1, \ldots, r_m)$ has precisely one negative or positive tangle \cite[p.529]{LT}. Without loss of generality we need only to consider $\js_K(n)$ and $\jx_K(n)$ for a Montesinos knot with precisely one negative tangle. The positive tangle case follows from taking mirror image.

Before stating our main result on Montesinos knots we start with the case of pretzel knots as they are the basis for our argument. In fact Theorem~\ref{thm.0} is the bulk of our work. For $P(q_0, \ldots, q_m)$ to be a knot, at most one tangle has an even number of crossings, and if each tangle has an odd number of crossings, then the number of tangles has to be odd. In the theorem below, the condition on the parities of the $q_i$'s and the number of tangles may be dropped if one is willing to exclude an arithmetic sub-sequence of colors $n$.

\begin{theorem}
\label{thm.0} 
Fix an $(m+1)$-vector $q$ of odd integers $q= (q_0,\dots, q_m)$ with $m\geq 2$ even and $q_0 < -1< 1 <q_1,\dots, 
q_m$.
Let $P=P(q_0,\dots,q_m)$ denote the corresponding pretzel knot. Define rational functions $s(q), s_1(q) \in \BQ(q)$: 
\begin{equation}
\label{eq.sqq}
s(q) =  1 +q_0 +\frac{1}{\sum_{i=1}^m (q_i-1)^{-1}} , \qquad
s_1(q) = \frac{\sum_{i=1}^m(q_i+q_0-2)(q_i-1)^{-1}}{\sum_{i=1}^m (q_i-1)^{-1}} \,.
\end{equation}
For all $n> N_K$ we have:
\newline
\rm{(a)} 
If $s(q)<0$, then the strong slope conjecture holds with
\be
\label{eq.jsxP}
\js_P(n) = -2s(q), \qquad 
\jx_P(n) = -2s_1(q)+4s(q)-2(m-1). 
\ee 
In particular, $\js_P(n)$ and $\jx_P(n)$ are constant functions. 
\newline
\rm{(b)} If $s(q)=0$, then the strong slope conjecture holds with
\be 
\js_P(n) = 0, \qquad
\jx_P(n) =\begin{cases} -2(m-1) & \text{if} \,\, s_1(q) \geq 0 \\
-2s_1(q)-2(m-1) & \text{if} \,\, s_1(q) < 0
\end{cases} \,.
\ee 
In particular, $\js_P(n)$ and $\jx_P(n)$ are constant functions. 

\noindent \rm{(c)} If $s(q)>0$, then the strong slope conjecture holds with
\be 
\js_P(n) = 0, \qquad
\jx_P(n) =-2(m-1).  
\ee
In particular, $\js_P(n)$ and $\jx_P(n)$ are constant functions. 

\end{theorem}

Next, we consider the case of Montesinos knots. Recall that by applying Euclid's algorithm, every rational number $r$ has 
a unique positive continued fraction expansion $r=[b_0,\dots,b_{\ell'}]$, see \eqref{eq:poscfe}, with $\ell' < \infty$, $b_0 \in \mathbb{Z}$, $|b_j| \geq 1$ for $1\leq j \leq \ell'-1$, $|b_{\ell'}| \geq 2$, and $b_j$'s all of the same sign as $r$. From this we define an even length continued fraction expansion $[a_0, \ldots, a_{\ell_r}]$ of $r$ to be equal to $[b_0,\dots,b_{\ell'}]$ if $\ell'$ is even, and we define it to be equal to $[b_0,\dots,b_{\ell'}-1, 1]$ (resp. $[b_0,\dots,b_{\ell'}+1, -1]$) if $\ell'$ is odd and $r> 0$ (resp. $r<0)$ . Note  $[a_0, \ldots, a_{\ell_r}]$ is well-defined. We will call $[a_0, \ldots, a_{\ell_r}]$ the unique even length positive continued fraction expansion for $r$. Define $r[j]=a_j$ for 
$j=0,\dots,\ell_r$, and define 
$$
[ r ]_{e} = \sum_{3\leq j\leq \ell_r, \ j=\text{even}} r[j], \quad
[r ]_{o} = \sum_{3\leq j\leq \ell_r, \ j=\text{odd}} r[j], \quad
 [ r] =  [ r ]_{e} + [ r ]_{o} \,.
$$  
For example, the fraction $63/202 = [0, 3, 4, 1, 5, 2]$  has the unique even length positive continued fraction expansion $[0, 3, 4, 1, 5, 1, 1]$. Adding up all the partial quotients of the continued fraction expansion with even indices $\geq 3$,  we get $[63/202]_e = 5+1 = 6$. Similarly, adding up all the partial quotients with odd indices $\geq 3$, we get $[63/202]_o = 1+1 = 2$.

 Given a Montesinos knot $K(r_0, \ldots, r_m)$, define $D_K$ to be the diagram obtained by summing rational tangle diagrams corresponding to the unique even length positive continued fraction expansion for each $r_i$, and then taking the numerator closure. See Section \ref{sec.RT} for how a rational tangle diagram is assigned to a continued fraction expansion of a rational number and definitions for the tangle sum and numerator closure. 

By the classification of Montesinos knots by \cite{Bon79},  and the existence and use of reduced diagrams of Montesinos links \cite{LT} based on the classification, we will further restrict to Montesinos knots $K(r_0, \ldots, r_m)$ where $|r_i| < 1$ for all $0 \leq i \leq m$. See Section \ref{ss.Mclassify} for the discussion of why we may do so without loss of generality. 

Let $(r_0,\dots,r_m) \in \BQ^{m+1}$ denote a tuple
of rational numbers, and let $(q_0,\dots,q_m) \in \BZ^{m+1}$ denote the
associated tuple of integers where $q_i=r_i[1]+1$ for $1\leq i \leq m$ and \[q_0 = \begin{cases} &r_0[1] -1 \text{ if $\ell_{r_0} = 2$ and $r_0[2]=-1$, 
\text{and}} \\ 
 	& r_0[1] \text{ otherwise }  \end{cases} \]  from the unique even length positive continued fraction expansion of $r_i$'s. Again, for the following theorem the condition on the parities of the $q_i$'s and the number of tangles ($m\geq 2$ even) may be dropped if one is willing to exclude an arithmetic sub-sequence of colors $n$, thus proving a weaker version of the conjecture for all Montesinos knots.
\begin{theorem} 
\label{thm.1} 
Let $K=K(r_0, r_1, \ldots, r_{m})$ be a Montesinos knot such that 
$r_0 <0$, $r_i >0$ for all $1 \leq i\leq m$, and $|r_i| < 1$ for all $0\leq i \leq m$ with $m\geq 2$ even.
Suppose $q_0 <-1<1<q_1,\ldots, q_m$ are all odd, and $q'_0$ is an integer that is defined to be 0 if $r_0 = 1/q_0$, and defined to be $r_0[2]$ otherwise. Let $P=P(q_0, \dots, q_m)$ be the associated pretzel knot, and let $\omega(D_K)$, $\omega(D_P)$ denote the writhe of $D_K$, $D_P$ with orientations. Then the strong slope conjecture holds. For all $n>N_K$ we have:
\begin{align*}
\label{eq.jsxM}
\js_K(n) &= \js_P(n) -q'_0- [r_0] -\omega(D_P) + \omega(D_K)+ \sum_{i=1}^m (r_i[2]-1) +  \sum_{i=1}^m [ r_i ] , \\
\jx_K(n) &= \jx_P(n) -2\frac{q'_0}{r_0[2]}+2 [r_0]_o -2 \sum_{i=1}^m (r_i[2]-1) -2 \sum_{i=1}^m [r_i ]_e. 
\end{align*}
In particular, $\js_P(n)$ and $\jx_P(n)$ are constant functions. 

\end{theorem} 
\begin{example} 
Consider the Montesinos knot $K = K(-\frac{46}{327}, 
\frac{35}{151}, 
\frac{5}{31}, \frac{16}{35} ,  \frac{1}{5})$. Applying Theorem \ref{thm.0} and \ref{thm.1}, we compute the Jones slope $\js_K$ by using Euclid's algorithm to obtain the unique even length continued fraction expansion for each rational number in the definition of $K$.
We have for the first rational number $-46/327$, 
\[ -\frac{46}{327} = 0+ \frac{1}{-\frac{327}{46}} =0+ \frac{1}{-7+(-\frac{5}{46})} =0+\frac{1}{-7+\frac{1}{-\frac{46}{5}}} = 0+ \frac{1}{-7+\frac{1}{-9+(-\frac{1}{5})}} = [0, -7,-9, -5].  \] This is of odd length, so the unique even length continued fraction expansion for $ -\frac{46}{327}$  is 
\[ -\frac{46}{327} = [0, -7,-9, -4, -1].\] 
The rational numbers together with their unique even length continued fractions expansions are  
\[ -\frac{46}{327}=[0, -7, -9, -4, -1], \
\frac{35}{151}=[0, 4, 3, 5, 2], \
\frac{5}{31}=[0, 6, 5], \ \frac{16}{35} = [0, 2, 5, 2, 1], \ \frac{1}{5}=[0, 4, 1].\] 
The associated pretzel knot is 
$P(-7, 5, 7, 3, 5)$. 
Theorem \ref{thm.0} applied to the pretzel knot gives that 
\[s(q) = -\frac{36}{7} < 0 \text{ and } s_1(q) = -\frac{32}{7}.   \] 
So
\[\js_P(n) = (-2)(-\frac{36}{7}) =  \frac{72}{7} \text{ and } \jx_P(n) = -2(-\frac{32}{7}) + 4(-\frac{36}{7}) - 2(4-1) = -\frac{122}{7}.  \] 
Dunfield's program \cite{Dunfield:table}, which computes boundary slopes and other topological properties of essential surfaces for a Montesinos knot based on Hatcher and Oertel's algorithm, produces an essential surface $S$ whose boundary slope equals $\js_{P}(n) = -2s(q) = 72/7$,  and such that $2\chi(S)/(7|\partial S|) = \jx_P =  -122/7$. 
Now we compute $\js_K(n)$ and $\jx_K(n)$ using Theorem \ref{thm.1}. To aid in presentation, we replace each symbol in the equations in the theorem by the number computed from the example. 
We have 
\begin{align*}
\js_K(n) &= \underbrace{\js_P(n)}_{72/7} -\underbrace{r_0[2]}_{-9} -  \underbrace{[r_0]}_{-4+-1} - \underbrace{\omega(D_P)}_{-13} + \underbrace{\omega(D_K)}_{-43}+ \underbrace{\sum_{i=1}^m (r_i[2]-1)}_{(2)+(4)+(4)} + \underbrace{\sum_{i=1}^m [r_i]}_{(5+2)+(2+1)} = \frac{100}{7}. \\ 
\jx_K(n) &= \underbrace{\jx_P(n)}_{-122/7} -2+2 \underbrace{ [r_0]_o}_{-4} -2 \underbrace{\sum_{i=1}^m (r_i[2]-1)}_{(2)+(4)+(4)} -2 \underbrace{\sum_{i=1}^m [r_i]_e}_{(2)+(1)}= -\frac{374}{7}. 
\end{align*} 
For the Montesinos knot, Dunfield's program also produces an essential surface $S$ which realizes the strong slope conjecture, with boundary slope 100/7 and $2\chi(S) / 7|\partial S| = -374/7. $
\end{example}

\subsection{Plan of the proof}
\label{sub.plan}

We divide the proof of Theorem \ref{thm.0} and Theorem \ref{thm.1} into two parts, first concerning the claims regarding the degree of the colored Jones polynomial, and the second concerning the existence of essential surfaces realizing the strong slope conjecture. 

First we use a mix of skein theory and fusion, reviewed in Section \ref{ss.skeinandcjp}, to find a formula for the degree of the dominant
terms in the resulting state sum for the colored Jones polynomial in Section \ref{sec.cj}.
Using quadratic integer programming techniques we determine the maximal degree of these dominant terms 
in Section \ref{sec.QIP}, and this is applied to find the degree of the colored Jones polynomial for the pretzel knots we consider in Section \ref{sub.application}. In Section \ref{sec.cjmontsinos} we determine the degree of the colored Jones polynomial for the Montesinos knots we consider in Theorem \ref{thm.1} by reducing to the pretzel case.  Finally, we work out the relevant surfaces using the Hatcher-Oertel algorithm in Section \ref{sec.incompressible}, and we match the growth rate of the degree of the quantum invariant with the topology, using the analogy drawn between the parameters of the state sum and the parameters for the Hatcher-Oertel algorithm by Lemma \ref{l.HOm}. We explicitly describe the essential surfaces realizing the strong slope conjecture in Sections \ref{sec.jslopem} and \ref{sec.jslopemm}, and the proofs of Theorem \ref{thm.0} and Theorem \ref{thm.1} are completed in Section \ref{subsec.jslopemp} and Section \ref{subsec.jslopemmp}, respectively.


\section{Preliminaries}
\subsection{Rational tangles} \label{sec.RT}

Let us recall how to describe rational tangles by rational numbers
and their continued fraction expansions. Originally studied by Conway \cite{Cn}, this material is well-known and may be found for instance in~\cite{KL,BS}. An $(m, n)$-tangle is an embedding of a finite collection of arcs and circles into $B^3$, such that the endpoints of the arcs lie in the set of $m+n$ points on $\partial B^3=S^2$. We consider tangles up to isotopy of the ball $B^3$ fixing the boundary 2-sphere. The integer $m$ indicates the number of points on the upper hemisphere of $S^2$, and the integer $n$ indicates the number of points on the lower hemisphere. We may isotope a tangle so that its endpoints are arranged on a great circle of the boundary 2-sphere $S^2$, preserving the upper/lower information of endpoints  from the upper/lower hemisphere. A tangle diagram is then a regular projection of the tangle onto the plane of this great circle. We represent tangles by tangle diagrams, and we will refer to an $(m, m)$-tangle as an $m$-tangle. Our building blocks of rational tangles are the horizontal and the vertical 2-tangles shown below, called elementary tangles in \cite{KL}.
\begin{itemize}
\item 
A \emph{horizontal tangle} has $n$ horizontal half-twists (i.e., crossings) 
for $n\in \mathbb{Z}$. 
\begin{figure}[!htpb]
\centering \def \svgwidth{.5\columnwidth}
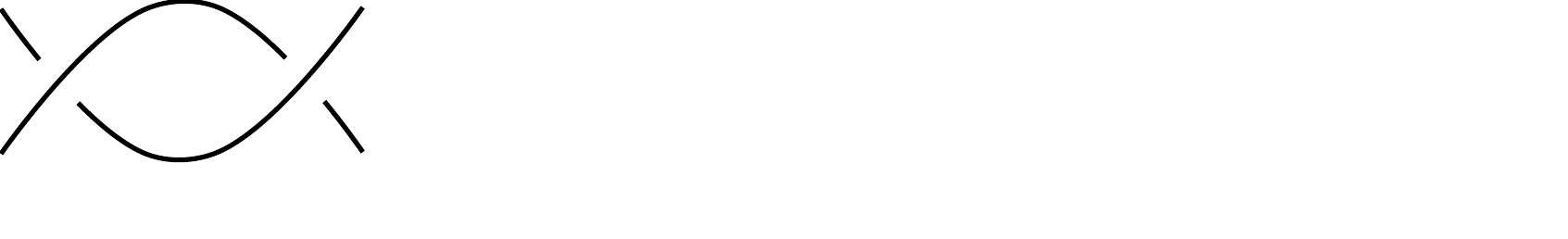
\end{figure}
\item A \emph{vertical tangle} has $n$ vertical half-twists (i.e., crossings)
for $n\in \mathbb{Z}$. 
\begin{figure}[H]
\centering \def \svgwidth{.5\columnwidth}
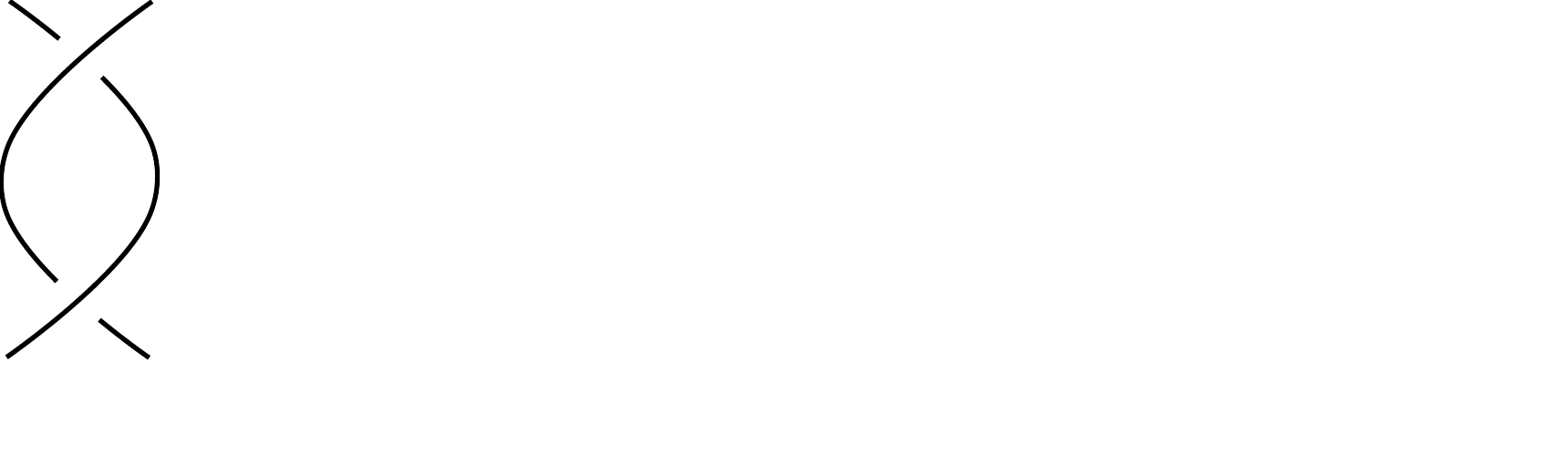
\end{figure}
\end{itemize}

The horizontal tangle with 0 half-twists will be called the 0 tangle, and the vertical tangle with 0 half-twists will be called the $\infty$ tangle. 

\begin{definition}A rational tangle is a 2-tangle that can be obtained by applying a finite number of consecutive twists of neighboring endpoints to the 0 tangle and the $\infty$ tangle. 
\end{definition}

For $2m$-tangles we define tangle addition, denoted by $\oplus$, and tangle multiplication, denoted by $*$, as follows in Figure \ref{f.tam}. We also define the numerator closure of a $2m$-tangle as a knot or link obtained by joining the two sets of $m$ endpoints in the upper hemisphere, and by joining the two sets of $m$ endpoints in the lower hemisphere.  
\begin{figure}[!htpb]
\centering 
\def \svgwidth{.9\columnwidth}
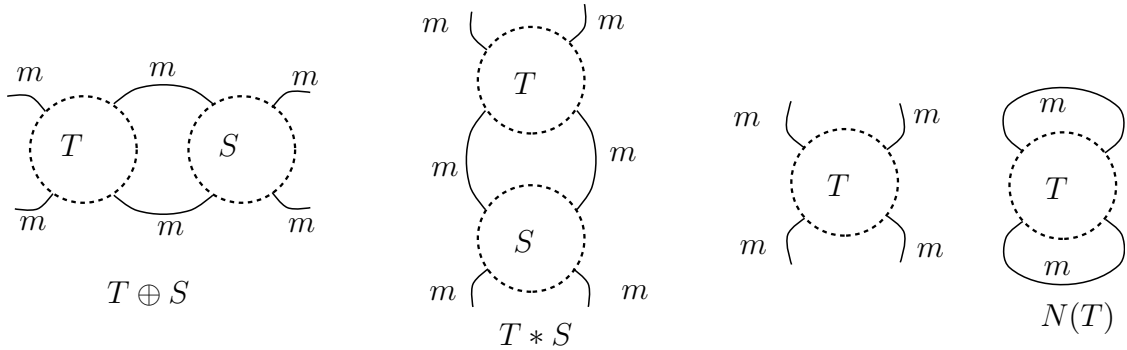
\caption{\label{f.tam} $2m$-tangle addition, multiplication, and numerator closure. }
\end{figure}

The following theorem is paraphrased from \cite{KL} with changes in notations for the elementary rational tangles. 
\begin{theorem}{\cite[Lemma 3]{KL}}
Every rational tangle can be isotoped to have a diagram in standard form, obtained by consecutive additions of horizontal tangles only on the right (or only on the left) and consecutive multiplications by vertical tangles only at the bottom (or only at the top), starting from the 0 tangle or the $\infty$ tangle. 

 More precisely, every rational tangle diagram may be isotoped to have the algebraic presentation
  \begin{equation} \label{e.algt1} (((a_\ell * \frac{1}{a_{\ell-1}})\oplus a_{\ell-2})* \cdots * \frac{1}{a_{1}}) \oplus a_0, \end{equation}
  if $\ell$ is even, or 
    \begin{equation} \label{e.algt2} (((\frac{1}{a_{\ell}}\oplus a_{\ell-1}) * a_{\ell-2})\oplus \cdots * \frac{1}{a_{1}}) \oplus a_0, \end{equation}
  if $\ell$ is odd, where $a_j \in \mathbb{Z}$ for $0 \leq j \leq \ell$, and $a_j \not=0$ for $1\leq j \leq \ell$.   
  \end{theorem} 
  
Recall the notation of the positive continued fraction 
expansion~\cite{KL,BS}:
\begin{equation} 
\label{eq:poscfe} 
[a_0, \ldots, a_\ell] = 
a_0+\cfrac{1}{a_1+\cfrac{1}{a_2+\cfrac{1}{a_3+\cdots +\cfrac{1}{a_\ell}}}} 
\end{equation} 
for integers $a_j \not= 0$ of the same sign for $1\leq j \leq \ell$ and $a_0 \in \mathbb{Z}$.
We define the rational number $r$ associated to a rational tangle in standard form with algebraic expression \eqref{e.algt1} or \eqref{e.algt2} to be 
  \[ r = [a_0, \ldots, a_\ell] . \] 
  
 Conversely, given a positive continued fraction expansion of a rational number $r =[a_0, \ldots, a_{\ell}]$ we may obtain a diagram of a rational tangle given by the corresponding algebraic expression \eqref{e.algt1} or \eqref{e.algt2}. See Figure \ref{f.tangleex} for an example. 

\begin{figure}[!htpb]
\centering
\includegraphics[height=0.20\textheight]{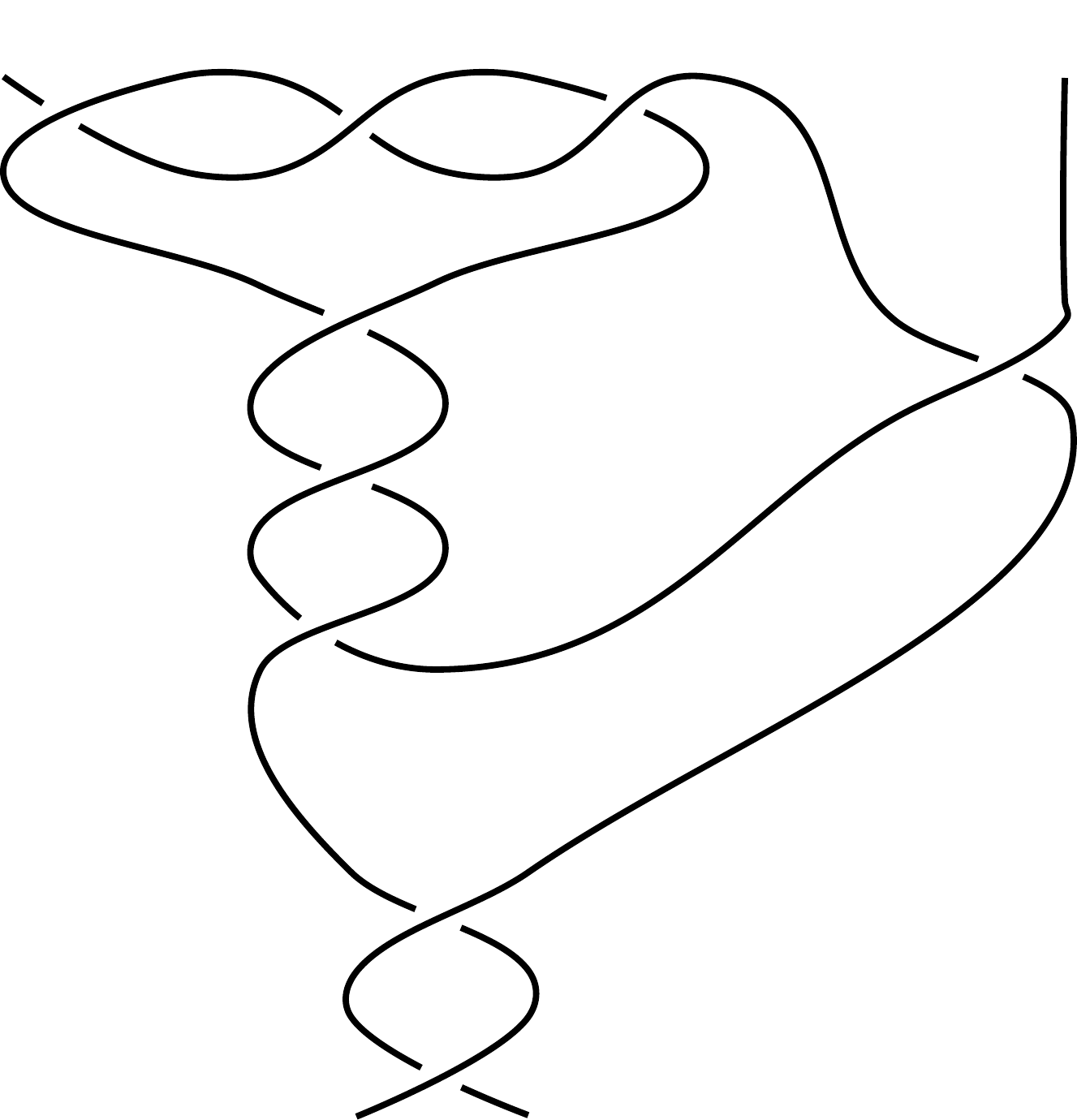}
\caption{\label{f.tangleex} A rational tangle diagram $T$ associated to the continued 
fraction expansion $[0, 2, 1, 3, 3]=13/36$. }
\end{figure}  
 
 A rational tangle is determined by their associated rational number to a standard diagram by the following theorem. 
 \begin{theorem}{\cite{Cn}} \label{t.rtcfe}
 Two rational tangles are isotopic if and only if they have the same associated rational number. 
 \end{theorem} 
See \cite[Theorem 3]{KL} for a proof of this statement. 

\begin{definition} \label{d.Montesinos}
A \emph{Montesinos link} $K(r_0, r_1, \ldots, r_m)$ is a link that admits a diagram $D$ obtained by summing rational tangle diagrams $T_{c_0}, T_{c_1}, \ldots, T_{c_m}$ then taking the numerator closure: 
\[D = N(((T_{c_0} \oplus T_{c_1})  \oplus T_{c_2})\oplus \cdots  \oplus T_{c_m}). \] Here $c_i$ for each $0\leq i \leq m$ is a choice of a positive continued fraction expansion of $r_i$, and $T_{c_i}$ is the rational tangle diagram constructed based on $c_i$ via \eqref{e.algt1} or \eqref{e.algt2}, depending on whether the length of $c_i$ is even or odd, respectively. 
\end{definition} 
Note that a different choice of positive continued fraction expansion for each $r_i$ in the sum of Definition \ref{d.Montesinos} produces a different diagram of the same knot by Theorem \ref{t.rtcfe}. To simplify our arguments, we will fix a diagram for the Montesinos knot $K(r_0, r_1, \ldots, r_m)$ by specifying the choice of a positive continued fraction expansion for each rational number $r_i$.

\subsection{Classification of Montesinos links} \label{ss.Mclassify}
The book \cite{BZ} has a complete account of the classification of Montesinos links, originally due to Bonahon \cite{Bon79}. The following version of the classification theorem comes from \cite{FKP13}. 

\begin{theorem}{\cite[Theorem 12.29]{BZ}} \label{t.Mclassify} Let $K(r_0, \ldots,  r_m)$ be a Montesinos link such that $m\geq 3$ and $r_0, \ldots, r_m \in \mathbb{Q}\setminus \mathbb{Z}$. Then $K$ is determined up to isomorphism by the rational number $\sum_{i=0}^m r_m$ and the vector $((r_0 \mod 1) , (r_1 \mod 1), \ldots, (r_m \mod 1))$, up to cyclic permutation and reversal of order.  
\end{theorem} 
 We will work with \emph{reduced}  diagrams for Montesinos knots as studied by Lickorish and Thistlethwaite \cite{LT}. Here we follow the exposition of \cite[Chapter 8]{FKP13}.  
\begin{definition}
Let $K$ be a Montesinos link.  A diagram is called a \emph{reduced Montesinos diagram} of $K$ if it is the numerator closure of the sum of rational angles $T_0, \ldots, T_m$ corresponding to rational numbers $r_0, \ldots, r_m$ with $m\geq 2$, and both of the following hold: 
\begin{enumerate}
\item Either all of the $r_i$'s have the same sign, or $0< |r_i| < 1$ for all $i$. 
\item For each $i$, the diagram of $T_i$ comes from a positive continued fraction expansion $[a_0, a_1, \ldots, a_{\ell_{i}}]$ of $r_i$ with the nonzero $a_j$'s all of the same sign as $r_i$. 
\end{enumerate} 
\end{definition} 
It follows as a consequence of the classification theorem that every Montesinos link $K(r_0, \ldots, r_m)$ with $m \geq 2$ has a reduced diagram. For example, if $r_i < 0$ while $r_{i'}  \geq 1$, we can subtract 1 from $r_{i'}$ and add 1 to $r_i$ until condition (1) is satisfied. This does not change the link type of the Montesinos link by Theorem \ref{t.Mclassify}. Since we are focused on Montesinos links with precisely one negative tangle we may assume that $0 < |r_i| <1$. Thus $r_i[0] = 0$ for all $0 \leq i \leq m$.

\subsection{Skein theory and the colored Jones polynomial} \label{ss.skeinandcjp}
We consider the skein module of properly embedded tangle diagrams on an oriented surface $F$
with a finite (possibly empty) collection of points specified on the 
boundary $\partial F$. This will be used to give a definition of the colored Jones polynomial from a diagram of a link.  For the original reference for skein modules 
see \cite{Pr91}. We will follow Lickorish's approach \cite[Section 13]{Lic97} 
except for the variable substitution (our $v$ is his $A^{-1}$ to avoid 
confusion with the $A$ for a Kauffman state). See \cite{Oh01} for how the skein theory gives the colored Jones polynomial, also known as the quantum $\mathfrak{sl}_2$ invariant. The word ``color" refers to the weight of the irreducible representation where one evaluates the invariant.

\begin{definition} 
\label{defn.skein}
Let $v$ be a fixed complex number. The linear skein module $\Sk(F)$ of 
$F$ is a vector space of formal linear sums over $\mathbb{C}$, of 
unoriented and properly-embedded tangle diagrams in $F$, considered up to isotopy of $F$ 
fixing $\partial F$, and quotiented by the \emph{skein relations} 
\begin{enumerate}[(i)]
\item 
$D \sqcup \vcenter{\hbox{\includegraphics[scale=.10]{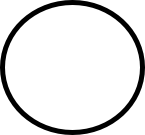}}} 
= (-v^{-2} - v^{2}) D$, and 
\item 
$ \vcenter{\hbox{\includegraphics[scale=.2]{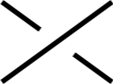}}} 
= v^{-1} \ \vcenter{\hbox{\includegraphics[scale=.2]{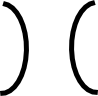}}} 
\ + v \ \vcenter{\hbox{\includegraphics[scale=.2]{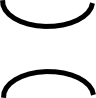}}} \ .
$
\end{enumerate}  
Here $ \vcenter{\hbox{\includegraphics[scale=.10]{figures/circ.png}}} $ denotes the unknot and $D \sqcup \vcenter{\hbox{\includegraphics[scale=.10]{figures/circ.png}}} $ is the disjoint union of the diagram $D$ with an unknot. Relation (ii) indicates how we can write a diagram with a crossing as a sum of two diagrams with coefficients in rational functions of $v$ by locally replacing the crossing by the two splicings on the right. 
\end{definition} 

We consider the linear skein module $\mathcal{S}(D^2, n, n')$ of 
the disk $D^2$ with $n+n'$-points specified on its boundary, 
where the boundary is viewed as a rectangle with $n$ marked points above 
and $n'$ marked points below. We will use this to decompose link diagrams into tangles. By the skein relations in Definition \ref{defn.skein}, every element in  $\mathcal{S}(D^2, n, n')$ is generated by crossingless matchings between the $n$ points on top and $n'$ points below. 
For crossingless matchings $D_1 \in \Sk(D^2, n, n')$ and $D_2 \in \Sk(D^2,n', n'')$, there is a natural multiplication 
operation $D_1\times D_2 \in \Sk(D^2, n, n'')$ defined by identifying the bottom boundary of $D_1$ 
with the top boundary of $D_2$ and matching the $n'$ common boundary points. Extending this by linearity to all elements in $\mathcal{S}(D^2, n, n)$ 
makes it  into an algebra $\TL^n_{n}$, called 
\emph{Temperley-Lieb algebra}. For the original references see \cite{TL71, Kau:TL}. 
We will simply write $\TL_n$ for $\TL_n^n$. There is a natural identification of $2n$-tangles with diagrams in $\TL_{2n}$. Pictorially, a non-negative integer such as $n$ next to a strand represents $n$ parallel strands. 


As an algebra, $\TL_n$  is generated by a 
basis $\{ |_n, e^{1}_n, \ldots, e^{n-1}_{n}\} $, where $|_n$ is the identity with 
respect to the multiplication, and $e^i_n$ is a crossingless tangle diagram 
as specified below in Figure \ref{fig:TLgen}. 

\begin{figure}[H]
\def \svgwidth{.5\columnwidth} 
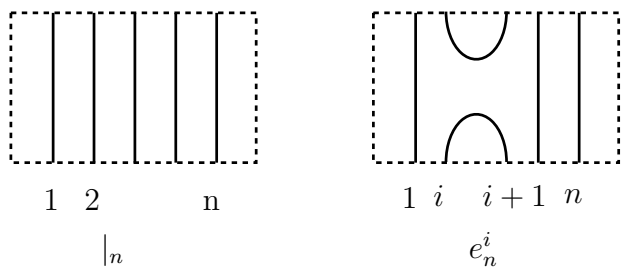
\caption{ \label{fig:TLgen} An example of the identity element $|_n$ (left) and 
a generator $e^i_n$ (right) of $\TL_n$ for $n=5$ and $i=2$.}
\end{figure}

Suppose that $v^{-4}$ is not a $k$th root of unity for $k\leq n$. There is 
an element, which we will denote by $\jwproj_n$, in $\TL_n$ called the 
$n$th \emph{Jones-Wenzl idempotent}.  For the original reference where the idempotent was 
defined and studied, see \cite{Wenzl}. Whenever $n$ is specified we will 
simply refer to this element as the Jones-Wenzl idempotent. 

The element $\jwproj_n$ is uniquely defined  by the 
following properties. (Note $\jwproj_1 = |_1$. )

\begin{enumerate}[(i)]
\item 
$\jwproj_n \times e^i_n = e^i_n \times \jwproj_n =0$ for $1\leq i \leq n-1$. 
\label{list:prop1}
\item 
$\jwproj_n -|_n $ belongs to the algebra generated by 
$\{e^1_n, e^2_n,\ldots, e^{n-1}_n\}$. 
\item 
$\jwproj_n \times \jwproj_n = \jwproj_n$. 
\item 
The image of \ $\jwproj_n$ in $\mathcal{S}(\mathbb{R}^2)$, obtained by 
embedding the disk $D^2$ in the plane and then joining the $n$ 
boundary points on the top with those on the bottom with $n$ disjoint 
planar parallel arcs outside of $D^2$, is equal to
\[ 
\frac{(-1)^{n}(v^{-2(n+1)}-v^{2(n+1)})}{v^{-2}-v^2} \cdot 
\text{the empty diagram in $\mathbb{R}^2$}.
\]
We will denote the rational function multiplying the empty diagram by $\triangle_n$. 
\label{list:prop4}
\end{enumerate}

\begin{definition}
\label{defn:cjp}
Let $D$ be a diagram of a link $K\subset S^3$ with $k$ components. 
For each component $D_i$ for $i \in \{1,\ldots, k\}$ take an annulus
$A_i = S^1\times I$ containing $D$ via the blackboard framing. Let $\mathcal{S}(S^1\times I)$ be the 
linear skein module of the annulus with no points marked on its boundary, and
let 
\[ 
f_D: \underbrace{\Sk(A_1) \times \cdots \times 
\Sk(A_k)}_{\text{Cartesian product}} \rightarrow \Sk(\mathbb{R}^2)     
\] 
be the map which sends a $k$-tuple of elements $(s_1 \in \Sk(A_1), \ldots, s_k \in \Sk(A_k))$ to 
$\Sk(\mathbb{R}^2)$ by immersing in the plane the  collection of skein elements in $\Sk(A_i)$ such that the over- and under-crossings of components of $D$ 
are the over- and under-crossings of the annuli.  For $n\geq 1$, the 
\emph{$n+1$th unreduced colored Jones polynomial} $J_{K, n+1}(v)$ may be defined as 
\[ 
J_{K, n+1}(v) := ((-1)^nv)^{\omega(D)(n^2+2n)} (-1)^n\left\langle 
f_D\underbrace{\left(\vcenter{\hbox{\def \svgwidth{.15\columnwidth} 
\begingroup%
  \makeatletter%
  \providecommand\color[2][]{%
    \errmessage{(Inkscape) Color is used for the text in Inkscape, but the package 'color.sty' is not loaded}%
    \renewcommand\color[2][]{}%
  }%
  \providecommand\transparent[1]{%
    \errmessage{(Inkscape) Transparency is used (non-zero) for the text in Inkscape, but the package 'transparent.sty' is not loaded}%
    \renewcommand\transparent[1]{}%
  }%
  \providecommand\rotatebox[2]{#2}%
  \ifx\svgwidth\undefined%
    \setlength{\unitlength}{130.91213437bp}%
    \ifx\svgscale\undefined%
      \relax%
    \else%
      \setlength{\unitlength}{\unitlength * \real{\svgscale}}%
    \fi%
  \else%
    \setlength{\unitlength}{\svgwidth}%
  \fi%
  \global\let\svgwidth\undefined%
  \global\let\svgscale\undefined%
  \makeatother%
  \begin{picture}(1,0.94154636)%
    \put(0,0){\includegraphics[width=\unitlength,page=1]{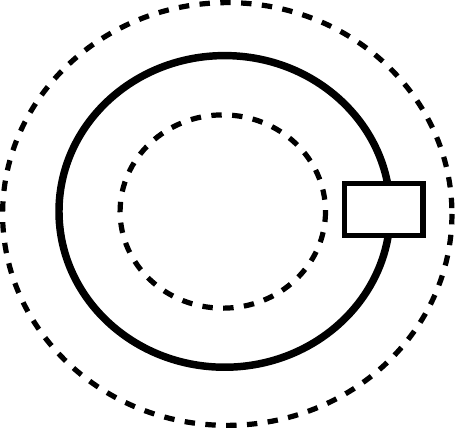}}%
    \put(0.36964248,0.04204078){\color[rgb]{0,0,0}\makebox(0,0)[lb]{\smash{$n$}}}%
  \end{picture}%
\endgroup%
}},  
\vcenter{\hbox{\def \svgwidth{.15\columnwidth} 
\begingroup%
  \makeatletter%
  \providecommand\color[2][]{%
    \errmessage{(Inkscape) Color is used for the text in Inkscape, but the package 'color.sty' is not loaded}%
    \renewcommand\color[2][]{}%
  }%
  \providecommand\transparent[1]{%
    \errmessage{(Inkscape) Transparency is used (non-zero) for the text in Inkscape, but the package 'transparent.sty' is not loaded}%
    \renewcommand\transparent[1]{}%
  }%
  \providecommand\rotatebox[2]{#2}%
  \ifx\svgwidth\undefined%
    \setlength{\unitlength}{130.91213437bp}%
    \ifx\svgscale\undefined%
      \relax%
    \else%
      \setlength{\unitlength}{\unitlength * \real{\svgscale}}%
    \fi%
  \else%
    \setlength{\unitlength}{\svgwidth}%
  \fi%
  \global\let\svgwidth\undefined%
  \global\let\svgscale\undefined%
  \makeatother%
  \begin{picture}(1,0.94154636)%
    \put(0,0){\includegraphics[width=\unitlength,page=1]{jwprojca.pdf}}%
    \put(0.36964248,0.04204078){\color[rgb]{0,0,0}\makebox(0,0)[lb]{\smash{$n$}}}%
  \end{picture}%
\endgroup%
}}, \cdots, 
\vcenter{\hbox{\def \svgwidth{.15\columnwidth} 
\begingroup%
  \makeatletter%
  \providecommand\color[2][]{%
    \errmessage{(Inkscape) Color is used for the text in Inkscape, but the package 'color.sty' is not loaded}%
    \renewcommand\color[2][]{}%
  }%
  \providecommand\transparent[1]{%
    \errmessage{(Inkscape) Transparency is used (non-zero) for the text in Inkscape, but the package 'transparent.sty' is not loaded}%
    \renewcommand\transparent[1]{}%
  }%
  \providecommand\rotatebox[2]{#2}%
  \ifx\svgwidth\undefined%
    \setlength{\unitlength}{130.91213437bp}%
    \ifx\svgscale\undefined%
      \relax%
    \else%
      \setlength{\unitlength}{\unitlength * \real{\svgscale}}%
    \fi%
  \else%
    \setlength{\unitlength}{\svgwidth}%
  \fi%
  \global\let\svgwidth\undefined%
  \global\let\svgscale\undefined%
  \makeatother%
  \begin{picture}(1,0.94154636)%
    \put(0,0){\includegraphics[width=\unitlength,page=1]{jwprojca.pdf}}%
    \put(0.36964248,0.04204078){\color[rgb]{0,0,0}\makebox(0,0)[lb]{\smash{$n$}}}%
  \end{picture}%
\endgroup%
}} \hspace{.3cm} 
\right)}_{k \text{ times }} \right\rangle, 
\] 
\end{definition} 
\noindent where $\langle \Sk \rangle$ for a linear skein element in $\mathcal{S}(\mathbb{R}^2)$ 
is the polynomial in $v$ multiplying the empty diagram after resolving 
crossings and removing disjoint circles of $\Sk$ using the skein relations. This is called the \emph{Kauffman bracket} 
of $\Sk$. To simplify notation, we will write 
\[ 
 D^{n} =f_D\left(\vcenter{\hbox{\def \svgwidth{.15\columnwidth} 
\begingroup%
  \makeatletter%
  \providecommand\color[2][]{%
    \errmessage{(Inkscape) Color is used for the text in Inkscape, but the package 'color.sty' is not loaded}%
    \renewcommand\color[2][]{}%
  }%
  \providecommand\transparent[1]{%
    \errmessage{(Inkscape) Transparency is used (non-zero) for the text in Inkscape, but the package 'transparent.sty' is not loaded}%
    \renewcommand\transparent[1]{}%
  }%
  \providecommand\rotatebox[2]{#2}%
  \ifx\svgwidth\undefined%
    \setlength{\unitlength}{130.91213437bp}%
    \ifx\svgscale\undefined%
      \relax%
    \else%
      \setlength{\unitlength}{\unitlength * \real{\svgscale}}%
    \fi%
  \else%
    \setlength{\unitlength}{\svgwidth}%
  \fi%
  \global\let\svgwidth\undefined%
  \global\let\svgscale\undefined%
  \makeatother%
  \begin{picture}(1,0.94154636)%
    \put(0,0){\includegraphics[width=\unitlength,page=1]{jwprojca.pdf}}%
    \put(0.36964248,0.04204078){\color[rgb]{0,0,0}\makebox(0,0)[lb]{\smash{$n$}}}%
  \end{picture}%
\endgroup%
}},  
\vcenter{\hbox{\def \svgwidth{.15\columnwidth} 
\begingroup%
  \makeatletter%
  \providecommand\color[2][]{%
    \errmessage{(Inkscape) Color is used for the text in Inkscape, but the package 'color.sty' is not loaded}%
    \renewcommand\color[2][]{}%
  }%
  \providecommand\transparent[1]{%
    \errmessage{(Inkscape) Transparency is used (non-zero) for the text in Inkscape, but the package 'transparent.sty' is not loaded}%
    \renewcommand\transparent[1]{}%
  }%
  \providecommand\rotatebox[2]{#2}%
  \ifx\svgwidth\undefined%
    \setlength{\unitlength}{130.91213437bp}%
    \ifx\svgscale\undefined%
      \relax%
    \else%
      \setlength{\unitlength}{\unitlength * \real{\svgscale}}%
    \fi%
  \else%
    \setlength{\unitlength}{\svgwidth}%
  \fi%
  \global\let\svgwidth\undefined%
  \global\let\svgscale\undefined%
  \makeatother%
  \begin{picture}(1,0.94154636)%
    \put(0,0){\includegraphics[width=\unitlength,page=1]{jwprojca.pdf}}%
    \put(0.36964248,0.04204078){\color[rgb]{0,0,0}\makebox(0,0)[lb]{\smash{$n$}}}%
  \end{picture}%
\endgroup%
}}, \cdots, 
\vcenter{\hbox{\def \svgwidth{.15\columnwidth} 
\begingroup%
  \makeatletter%
  \providecommand\color[2][]{%
    \errmessage{(Inkscape) Color is used for the text in Inkscape, but the package 'color.sty' is not loaded}%
    \renewcommand\color[2][]{}%
  }%
  \providecommand\transparent[1]{%
    \errmessage{(Inkscape) Transparency is used (non-zero) for the text in Inkscape, but the package 'transparent.sty' is not loaded}%
    \renewcommand\transparent[1]{}%
  }%
  \providecommand\rotatebox[2]{#2}%
  \ifx\svgwidth\undefined%
    \setlength{\unitlength}{130.91213437bp}%
    \ifx\svgscale\undefined%
      \relax%
    \else%
      \setlength{\unitlength}{\unitlength * \real{\svgscale}}%
    \fi%
  \else%
    \setlength{\unitlength}{\svgwidth}%
  \fi%
  \global\let\svgwidth\undefined%
  \global\let\svgscale\undefined%
  \makeatother%
  \begin{picture}(1,0.94154636)%
    \put(0,0){\includegraphics[width=\unitlength,page=1]{jwprojca.pdf}}%
    \put(0.36964248,0.04204078){\color[rgb]{0,0,0}\makebox(0,0)[lb]{\smash{$n$}}}%
  \end{picture}%
\endgroup%
}} \hspace{.3cm} \right)
.\] 

A Kauffman state \cite{Kau87}, which we will denote by $\sigma$, is a choice of the $A$- or $B$-resolution 
at a crossing of a link diagram. 

\begin{figure}[!htpb]
\centering
\def \svgwidth{.5\columnwidth}
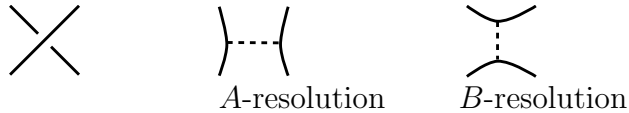
\caption{$A$- and $B$-resolutions of a crossing. The dashed segment records the location where the crossing was.}
\label{fig:abres}
\end{figure} 

\begin{definition}
Let $\sigma$ be a Kauffman state on a skein element with crossings, define 
\[\sgn(\sigma) =  (\# \text{ of $B$-resolutions of $\sigma$})-(\# \text{ of $A$-resolutions of $\sigma$}).   \] 
This quantity keeps track of the number of $A$- and $B$-resolutions chosen by $\sigma$. 
\end{definition}

\begin{definition}
Given a skein element $\Sk$ with crossings in $\Sk(\mathbb{R}^2)$, the \emph{$\sigma$-state} denoted by $\Sk_{\sigma}$ is the set of disjoint arcs and circles, possibly connecting Jones-Wenzl idempotents, resulting from applying a Kauffman state $\sigma$ to $\Sk$. The \emph{$\sigma$-state graph}  $\Sk^G_{\sigma}$  is the set of disjoint arcs and circles, possibly connecting Jones-Wenzl idempotents, resulting from applying a Kauffman state $\sigma$ to $\Sk$ along with (dashed) segments recording the original locations of the crossings as shown in Figure \ref{fig:abres}. 
\end{definition}

We summarize standard techniques and formulas for computing the colored Jones polynomial using Definition \ref{defn:cjp} that are used in this paper. 
Given a diagram $D^n$ decorated with a single Jones-Wenzl idempotent from a link, a \emph{state sum} for the Kauffman bracket $\langle D^n \rangle$ of $D^n$ is an expansion of $\langle D^n \rangle$ into a sum over skein elements $(D^n)_{\sigma}$ resulting from applying a Kauffman state $\sigma$ on a subset of crossings in $D^n$. As an example, one can compute the second colored Jones polynomial of the trefoil knot $3_1$ by writing down the following state sum in Figure \ref{f.trefoilss}. 
\begin{figure}[H]
\begin{center}
\def \svgwidth{\columnwidth}
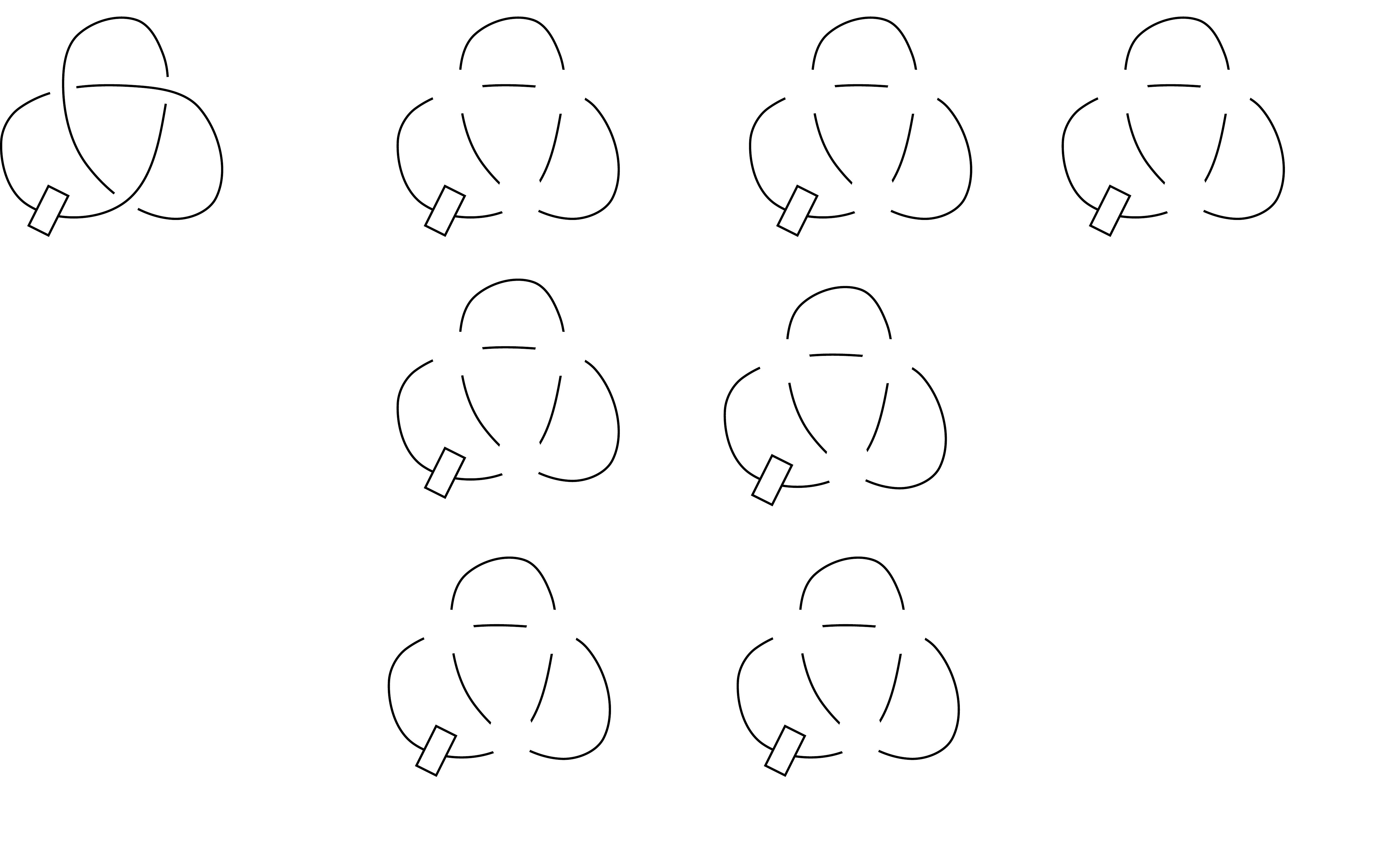
\end{center}
\caption{\label{f.trefoilss} A state sum for the 2nd colored Jones polynomial of the left-hand trefoil 
$3_1$. In this example, Kauffman states are taken over the set of all crossings of the diagram.}
\end{figure} 
 We are left with disjoint arcs and circles connecting the Jones-Wenzl idempotent. These may be removed by applying skein relations and by applying properties of the idempotent to obtain the polynomial. Note that we can also write down a state sum for a skein element with crossings which may be decorated by Jones-Wenzl idempotents. 

Since we are interested in bounding degrees of the Kauffman brackets of skein elements in the state sum, we will define a few more relevant combinatorial quantities and gather some useful results. 

The \emph{degree} of a rational function $L(v)$, denoted by $\deg_v(L(v))$,  is the maximum power of $v$ in the formal Laurent series expansion of $L(v)$ with finitely many positive degree terms. 

Let $\Sk_{\sigma}$ be a skein element coming from applying a Kauffman state $\sigma$ to a skein element $\Sk$ with crossings and decorated by Jones-Wenzl idempotents in $\Sk(\mathbb{R}^2)$. Then
$\overline{\Sk_{\sigma}}$ is the set of disjoint circles obtained from $\Sk_\sigma$ by replacing
all idempotents with the identity. 

\begin{definition}
A \emph{sequence} $s$ of states starting at $\sigma_1$ and ending at 
$\sigma_f$ on a set of crossings in a skein element $\Sk$ is a finite 
sequence of Kauffman states $\sigma_1, \ldots, \sigma_f$, where $\sigma_{i}$ 
and $\sigma_{i+1}$ differ on the choice of the $A$- or $B$-resolution at only 
one crossing $x$, so that $\sigma_{i+1}$ chooses the $A$-resolution at $x$ 
and $\sigma_i$ chooses the $B$-resolution.
\end{definition}

Let $s=\{\sigma_1, \ldots, \sigma_f \}$ be a sequence of states starting 
at $\sigma_1$ and ending at $\sigma_f$. In each step from $\sigma_i$ to 
$\sigma_{i+1}$ either two circles of $\overline{\Sk_{\sigma_i}}$ merge into 
one or a circle of $\overline{\Sk_{\sigma_i}}$ splits into two. When two 
circles merge into one as the result of changing the $B$-resolution to 
the $A$-resolution, the number of circles of the skein element decreases by 1 
while the sign of the state decreases by 2. More precisely, let 
$\Sk_{\sigma}$ be the skein element resulting from applying the Kauffman state 
$\sigma$, we have 
\begin{equation*} 
\sgn(\sigma_{i+1}) + \deg_v \langle \overline{\Sk_{\sigma_{i+1}}}\rangle = 
\sgn(\sigma_{i}) + \deg_v \langle \overline{\Sk_{\sigma_{i}}}\rangle -4 \,, 
\end{equation*}
when a pair of circles merges from $\overline{\Sk_{\sigma_{i}}}$ to 
$\overline{\Sk_{\sigma_{i+1}}}$. This immediately gives the following corollary.

\begin{lemma} 
\label{l.split}
Let $s=\{\sigma_1, \ldots, \sigma_f\}$ be a sequence of states on 
a skein element $\Sk$ with crossings, then 
\[
\sgn(\sigma_{1}) + \deg_v \langle \overline{\Sk_{\sigma_{1}}}\rangle 
= \sgn(\sigma_{f}) + \deg_v \langle \overline{\Sk_{\sigma_{f}}}\rangle 
\] 
if and only if a circle is split from $\overline{\Sk_{\sigma_i}}$ to 
$\overline{\Sk_{\sigma_{i+1}}}$ for every $1\leq i \leq f-1$. Otherwise 
\[
\sgn(\sigma_{1}) + \deg_v \langle \overline{\Sk_{\sigma_{1}}}\rangle 
> \sgn(\sigma_{f}) + \deg_v \langle \overline{\Sk_{\sigma_{f}}}\rangle .
\] 
\end{lemma} 
 
 We will also use standard fusion and untwisting formulas involving skein elements decorated by Jones-Wenzl idempotents for which 
one can consult \cite{Lic97} and the original reference \cite{MV}.
\begin{figure}[H]

\def \svgwidth{.8\columnwidth}
\centering
\begin{equation} \label{eq.fusion}
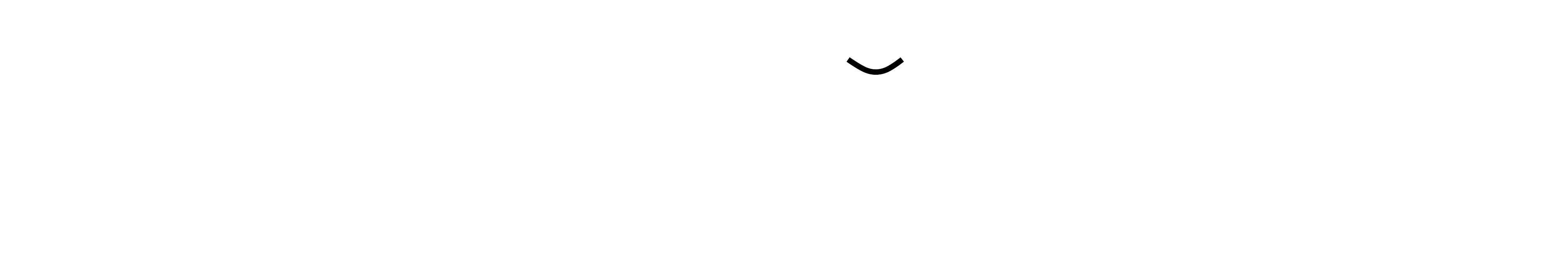
\end{equation} 
\caption{\label{f.fusion} Fusion formula: the skein element which locally looks like the left-hand side is equal to the sum of skein elements on the right-hand side with corresponding local replacements.  }
\end{figure}

\begin{figure}[H]
\centering
\def \svgwidth{.8\columnwidth}
\begin{equation} \label{eq.untwisting}
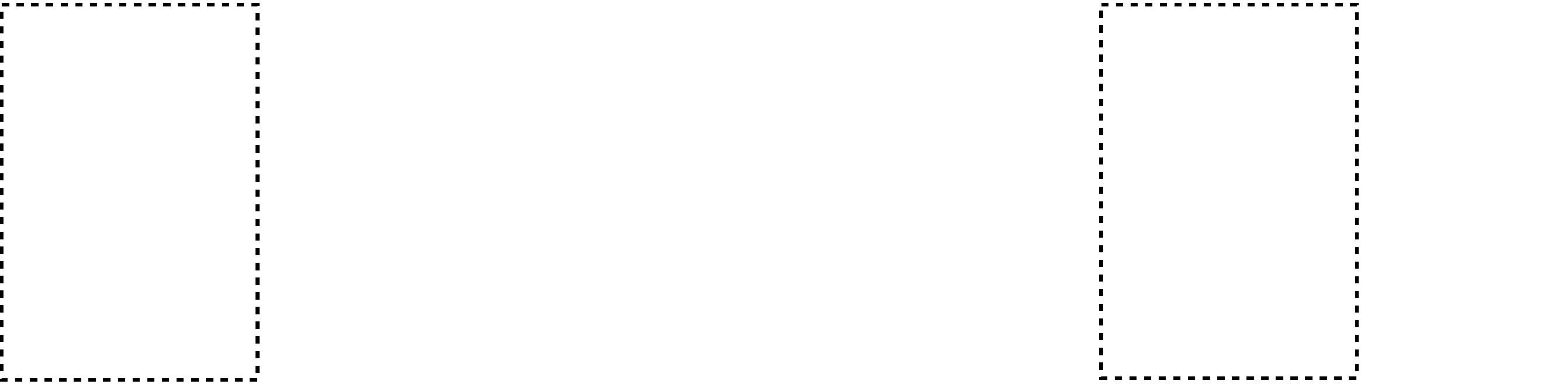
\end{equation} 
\caption{Untwisting formula: the skein element which locally looks like the left-hand side is equal to the skein element on the right-hand side with the local replacement.  }
\end{figure} 

We say that a triple $(a, b, c)$ of non-negative integers is admissible if $a+b+c$ is even and $a\leq b+c$, $b \leq c+a$, and $c \leq a+b$. For $k$ a non-negative integer, let $\triangle_{k}! := \triangle_k \triangle_{k-1} \cdots \triangle_1$, with the convention that $\triangle_0 = \triangle_{-1}  =1$. In Figure \ref{f.fusion} above, the function $\theta(a, b, c)$ is defined by 
\[\theta(a, b, c) = \frac{\triangle_{x+y+z}!\triangle_{x-1}!\triangle_{y-1}!\triangle_{z-1}!}{\triangle_{y+z-1}!\triangle_{z+x-1}! \triangle_{x+y-1}!},   \]
where $x, y$, and $z$ are determined by $a = y+z$, $b = z+x$, and $c = x+y$.



\section{The colored Jones polynomial of pretzel knots}
\label{sec.cj}

From this point on we will always consider the \emph{standard diagram} $K$ when referring to the pretzel knot
$K=P(q_0, \ldots, q_{m})$, with $|q_i|>1$. 
 Throughout the section the integer $n\geq 2$ is fixed, and we will illustrate graphically using the example $P(-5, 3, 3, 3, 5)$. 

The colored Jones polynomial for a fixed $n$ of a knot is by Definition \ref{defn:cjp} the Kauffman bracket of the $n$-blackboard cable ($n$-cable for short) of a diagram of $K$ decorated by a Jones-Wenzl idempotent, multiplied by a monomial in $v$ raised to the power of the writhe of the diagram with orientation. We write the colored Jones polynomial as $$J_{K,n+1}(v) = ((-1)^nv)^{\omega(K)n(n+2)}(-1)^n\langle K^n \rangle.$$ 

The Jones-Wenzl idempotent is a sum of tangle diagrams with coefficients  rational functions of $v$ in the algebra $\TL_n$.  A skein element in $\TL^n_{n'}$ decorated by Jones-Wenzl idempotents is thus also a sum of tangle diagrams with coefficients rational functions of $v$ by locally replacing idempotent with its sum. We extend the tangle sum operation $\oplus$ to skein elements $\Sk$ in  $\TL_{2n}$ decorated by Jones-Wenzl idempotents, written
\[\Sk = \sum_{T \in \TL_{2n}} s(v)T,  \]
as 
\[\Sk \oplus  \Sk' = \sum_{T, T' \in \TL_{2n}}  s(v)s'(v)T \oplus T'. \]  
Graphically, this will be the same as joining the top right and bottom right $2n$-strands of $\Sk$ to the top left and bottom left $2n$-strands to $\Sk'$ as in Figure \ref{f.tam}, except with the presence of the idempotent and possibly crossings indicating that this is actually a sum of such diagrams in $\TL_{2n}$. Similarly, we extend the numerator closure to skein elements in $\TL_{2n}$. 

We will represent the diagram $K^n = N(K_-^n\oplus K_+^n)$ as the numerator closure of the sum of two $2n$-tangles decorated by Jones-Wenzl idempotents, with the label $n$ indicating the number of parallel strands. This decomposition of $K^n$ reflects the original splitting of $K = N(K_- \oplus K_+)$ into two 2-tangles $K_-$ and $K_+$. A twist region is a vertical 2-tangle with a nonzero number of crossings all of the same sign. Let $K_-$ be the negative twist region consisting of $-q_0$ crossings, and $K_+$ the rest of the diagram $K$. For a fixed $n$ double the idempotents in $K^n$ so that four are framing the $n$-cable of the negative twist region consisting of $-q_0$ crossings, and four are framing the $n$-cable of the rest of the knot diagram.  The $2n$-tangle $K_-^n$ is the $n$-cable of  $K_-$ along with the four idempotents, and
$K_+^n$ is the rest of $K^n$, which is the $n$-cable of $K_+$, also decorated with four idempotents. See the middle figure in Figure \ref{f.statedecomp}. 

\begin{figure}[H]
\centering
\def \svgwidth{\columnwidth}
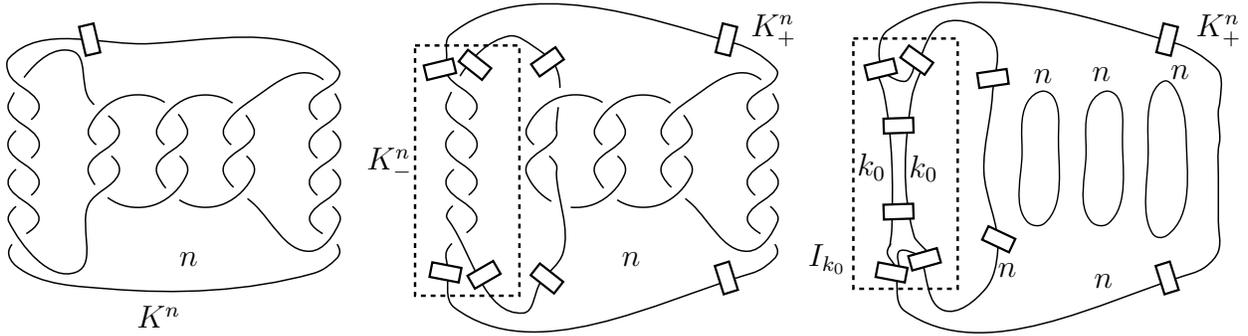
\caption{From left to right: $K^n$, doubling the idempotents and the splitting $K^n =  N(K^n_- \oplus   K^n_+)$, and $N(I_{k_0}\oplus (K_+^n)_{\sigma_A})$, where $\sigma_A$ is the Kauffman state that chooses the $A$-resolution on all the crossings in $K_+^n$. The dotted boxes enclose the skein elements in $\Sk(D^2, 2n, 2n)$, which are sums of $2n$-tangles. }
\label{f.statedecomp}
\end{figure} 

It is convenient to compute the bracket of these $2n$-tangles first. 
For any tangle $T$ write $\langle T^n \rangle$ to mean
cabling each component by a Jones-Wenzl idempotent of order $n$ and evaluating in the Temperley-Lieb algebra $\TL_ {2n}$ using the Kauffman bracket.

We write $\langle{K_-^n}\rangle = \sum_{k_0}G_{k_0}(v)I_{k_0}$ for 2n-tangles $I_{k_0}$ with four Jones-Wenzl idempotents of size $n$
connected in the middle by two Jones-Wenzl idempotents of size $2k_0$ arranged in an $I$-shape using the fusion and untwisting formulas. Apply the fusion formula \eqref{eq.fusion} to two strands of $K^n_-$ going into (or coming out of) the $n$-cabled negative twist region. Then, apply the untwisting formula \eqref{eq.untwisting} to get rid of all the negative crossings. The function $G_{k_0}(v) = \frac{\triangle_{2k_0}}{\theta(n ,n, 2k_0)} ((-1)^{n-k_0}v^{2n-2k_0 + n^2-2k_0^2})^{q_0}$ is a rational function that is the product of two coefficient functions in $v$ multiplying the replacement skein elements. 
The other tangle $K_+^n$ is expanded into a state sum by taking Kauffman states over all the crossings in $K_+^n$, leaving the four Jones-Wenzl idempotents of size $n$. Let $(K_+^n)_\sigma$ denote the skein element resulting from applying a Kauffman state $\sigma$ to all the crossings of $K_+^n$. Then, 
$\langle{K_+^n}\rangle = \sum_\sigma v^{\sgn(\sigma)} \langle (K_+^n)_\sigma \rangle$ as discussed in Section \ref{ss.skeinandcjp}. The state sum we consider is indexed by pairs $(k_0,\sigma)$ and we write
\be \label{eq.sum}
\langle K^n \rangle = \sum_{(k_0,\sigma)}G_{k_0}(v)v^{\sgn(\sigma)}\langle N(I_{k_0} \oplus (K^n_+)_\sigma) \rangle.  \ee
See the rightmost figure of Figure \ref{f.statedecomp} for an example of $N(I_{k_0}\oplus (K^n_+)_\sigma)$.
Using the notion of through strands, we collect like terms together in our state sum.

\begin{definition} 
\label{def.TL}
Consider the Temperley-Lieb algebra $\TL^n_{n'}$ with $n$ inputs and
$n'$ outputs.
Let $T$ be an element of $\TL^n_{n'}$ with no crossings. Viewing $\partial D^2$ as a square, an arc in $T$ with one endpoint 
on the top boundary of the disk $D^2$ defining $\TL^n_{n'}$ and another endpoint on the bottom boundary  is called a \emph{through strand} of $T$. 
\end{definition}

We can organize states $(k_0,\sigma)$ according to the number of through strands at various levels. The \emph{global} number of through strands of $\sigma$, denoted by $c=c(\sigma)$, is the number of through strands of $(K^n_+)_\sigma$ in $\TL_{2n}$ inside the box framed by four idempotents in $K_+^n$, see Figure \ref{fig.TLts} for examples of Kauffman states $(K^n_+)_{\sigma}$ and their through strands.
\begin{figure}[ht]
\def \svgwidth{.35\columnwidth}
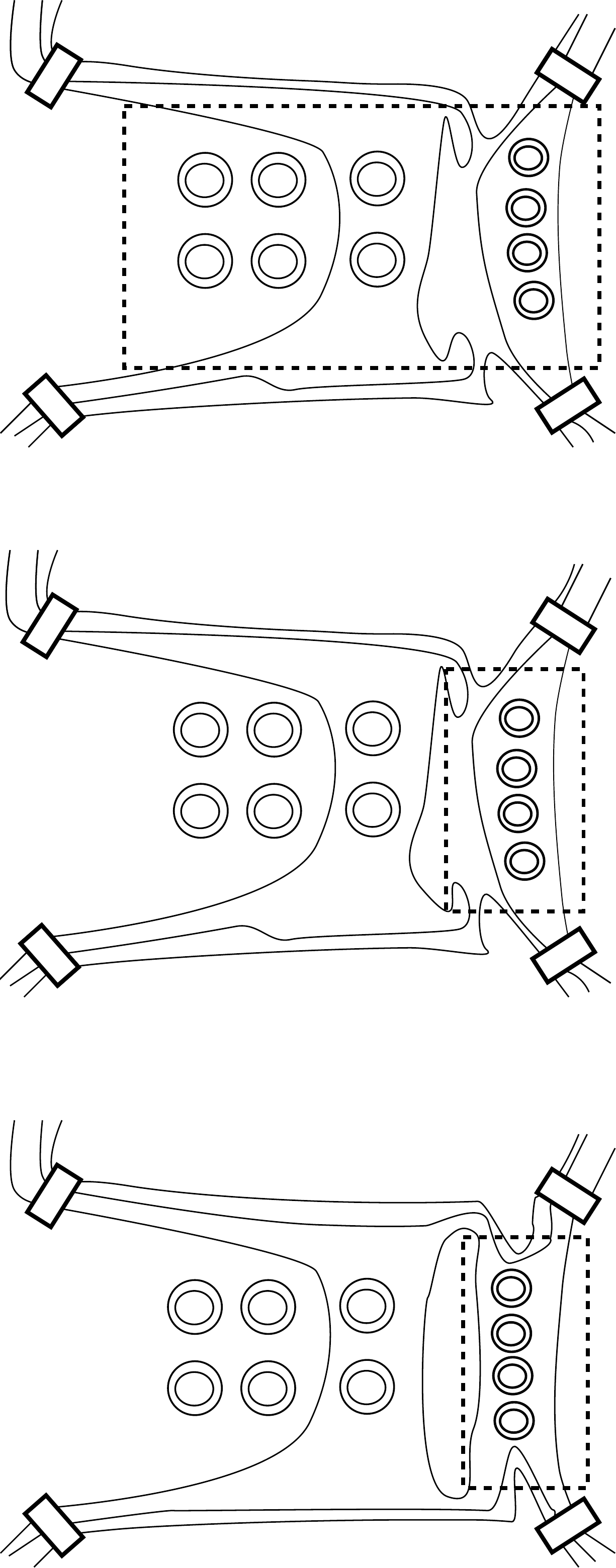
\caption{ \label{fig.TLts} Top: an example of $(K^n_+)_\sigma$ with $n=3$ and $c(\sigma) = 4$. Middle: when restricting $\sigma$ to $i=4$th twist region, we have $c_4(\sigma)  = k_4(\sigma)= 2$. Bottom: we show an example of a state $\sigma$ where $c_4(\sigma) = 1$ and therefore $k_4(\sigma) = \lceil \frac{1}{2} \rceil = 1$.} 
\end{figure}

For $1 \leq i \leq m$, we will also define $c_i(\sigma)$ to be the number of $i$th \emph{local} through strands when restricting $\sigma$ to the $i$th twist region, that are also global through strands. The parameter $k_i$ corresponding to a Kauffman state $\sigma$ for each twist region $q_i$  will be defined as
$k_i(\sigma) = \lceil \frac{c_i(\sigma)}{2} \rceil$. The intuition for these parameters is that they will be used to bound the degree of each term in the state sum relative to each other, which is crucial to determining the degree of the $n$th colored Jones polynomial $J_{K, n+1}$.

With the notation $k = (k_0,\dots,k_m)$ we set 
\be
\mathcal{G}_{c, k} = \sum_{k_0}\sum_{\sigma: k_i(\sigma)=k_i, c(\sigma) = c}G_{k_0}(v) v^{\sgn(\sigma)}\langle N(I_{k_0}\oplus (K^n_+)_\sigma) \rangle. \label{eq.ssum}
\ee

Note $0\leq k_i \leq n$ and define the parameters $c, k$ to be \emph{tight} if $k_0 = k_1+\dots + k_m= \frac{c}{2}$. We prove the following theorem. 
\begin{theorem}
\label{thm.Delta}
Assume $|q_i|>1$ and write $\langle K^n \rangle = \sum_{c, k} \mathcal{G}_{c, k}$ using \eqref{eq.sum} and \eqref{eq.ssum}.
For tight $c, k$ we have
$\mathcal{G}_{c, k} = (-1)^{q_0(n-k_0)+n+k_0
+\sum_{i=1}^{m}(n-k_i)(q_i-1)} v^{\delta(n, k)}+l.o.t.$\footnote{The abbreviation $l.o.t.$ means lower order terms in $v$.} and $\delta(n,k) = $ 
\be
\label{eq.Deltank}
-2\left((q_0+1)k_0^2+\sum_{i=1}^m (q_i-1) k_i^2 
+\sum_{i=1}^m(-2+q_0+q_i)k_i-\frac{n(n+2)}{2}\sum_{i=0}^m q_i+ (m-1)n\right).
\ee
If $c,k$ are not tight then there exists a tight pair $c',k'$ (coming from some Kauffman state) such that $\deg_v \mathcal{G}_{c,  k} < \deg_v \mathcal{G}_{c',  k'}$.
\end{theorem}

This theorem will be used in the next section to find the actual degree of $J_{K, n+1}$
using quadratic integer programming.

\subsection{Outline of the proof of Theorem \ref{thm.Delta}}
Let $c, k$ be tight and let $st(c, k)$ be the set of states $(k_0,\sigma)$ with $c(\sigma) = c$ and $k_i(\sigma)=k_i$ for all $1\leq i\leq m$. A state in $st(c,  k)$ is said to be \emph{taut} if its term $G_{k_0}(v)v^{\sgn(\sigma)}\langle N(I_{k_0} \oplus (K^n_+)_\sigma) \rangle$ in \eqref{eq.ssum} maximizes the $v$-degree within $st(c, k)$. For any fixed tight $c, k$ we plan to construct all taut states. The first examples  we construct will be \emph{minimal states}, from which we will derive all taut states. A state in $st(c, k)$ is minimal if it chooses the least number of $A$-resolutions.

We will first show that minimal states are characterized by having a certain configuration, or position, on the set of crossings where they choose the $A$-resolution, called \emph{pyramidal}. This will also be used to show that  $c, k$ not tight implies $\deg_v \mathcal{G}_{c,  k} < \deg_v \mathcal{G}_{c',  k'}$ for some tight pair $c', k'$.

Then, with the construction of all taut states from minimal states, we show that  $\delta(n, k)$ is the degree of a taut state with parameters $k$, and
\[\mathcal{G}^{taut}_{c, k \text{ tight }} =(-1)^{q_0(n-k_0)+n+k_0
+\sum_{i=1}^{m}(n-k_i)(q_i-1)} v^{\delta(n, k)} + l.o.t.,  \] where  $\mathcal{G}_{c, k \text{ tight }}^{taut}$ is the double sum of $\mathcal{G}_{c, k}$ only over taut states with tight $c, k$.
 This will lead to
\[\mathcal{G}_{c, k \text{ tight}}  = (-1)^{q_0(n-k_0)+n+k_0
+\sum_{i=1}^{m}(n-k_i)(q_i-1)} v^{\delta(n, k)} + l.o.t. \]
and conclude Theorem \ref{thm.Delta}.

\subsection*{Conventions for representing a Kauffman state}
Throughout the rest of Section \ref{sec.cj}, we will indicate schematically a crossingless skein element $\Sk_{\sigma}$, resulting from applying a Kauffman state to a skein element $\Sk$ with crossings,  
by the following convention. Let $\Sk^G_{B}$ be the all-$B$ state graph of $\Sk$. For a Kauffman state $\sigma$ let $A_{\sigma}$ 
be the set of crossings of $\Sk$ on which $\sigma$ chooses the 
$A$-resolution, and define $|A_{\sigma}|$ to be the number of crossings in $A_{\sigma}$. The skein element $\Sk_{\sigma}$ is represented by $\Sk_B^G$ with 
colored edges, such that the edge in $\Sk_B^G$ corresponding to a crossing 
in $A_{\sigma}$ is colored red, and all other edges remain black. The skein element $\Sk_{\sigma}$ may then be 
recovered from $\Sk^G_B$ by a local replacement of two arcs with a dashed segment. See Figure 
\ref{fig:localreplacement} below.

\begin{figure}[!htpb]
\centering
\def \svgwidth{.25\columnwidth}
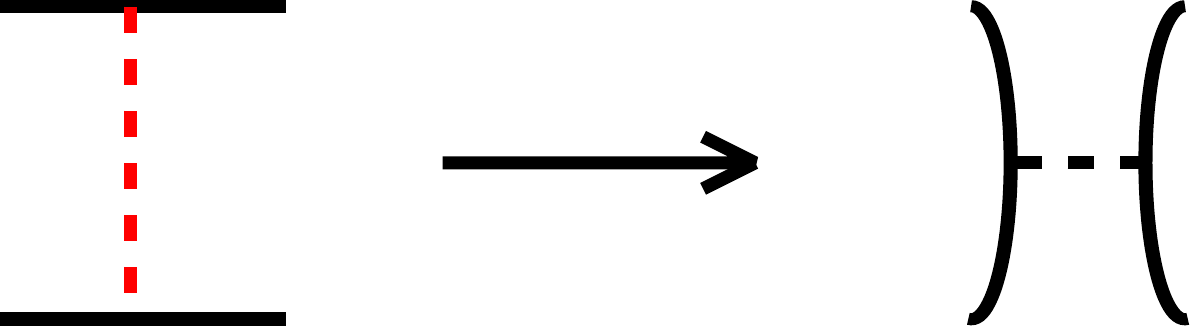
\caption{A red edge in $\Sk^G_B$ indicates the choice of the $A$-resolution for a Kauffman state $\sigma$ on $\Sk$.}
\label{fig:localreplacement} 
\end{figure}

\subsection{Simplifying the state sum and pyramidal position for crossings}

We will denote by $\Sk(k_0, \sigma)$ the skein element $N(I_{k_0} \oplus (K^n_+)_{\sigma})$ as in \eqref{eq.ssum}. 
\begin{lemma} \label{l.zero} Fix $(k_0, \sigma)$ determining a skein element $\Sk(k_0, \sigma)$ with $k_i = k_i(\sigma)$ and $c=c(\sigma)$.  
 If $k_0 > \sum_{i=1}^m k_i$, then $\Sk(k_0, \sigma)$ = 0. 
\end{lemma} 
\begin{proof}
Note that $\sum_{i=1}^m k_i \geq \frac{c}{2}$. Thus if $k_0 > \sum_{i=1}^m k_i$, then $k_0 > \frac{c}{2}$, and the lemma follows from \cite[Lemma 3.2]{L}.  
\end{proof} 

With the information of through strands $c(\sigma)$ and $\{k_i(\sigma)\}$, we describe the structure of $A_{\sigma}$ for a Kauffman state $\sigma$. It is necessary to introduce a labeling of the crossings with respect to their positions in the all-$B$ Kauffman state graph $\Sk^G(k_0, B)=N(I_{k_0}\oplus (K^n_+)^G_{B})$.

We first further decompose $K^n_+ = \Sk^t \times \Sk^w \times \Sk^b$ 
where $\times$ is the multiplication by stacking in $\TL$, and let the crossings 
contained in those skein elements be denoted by $C^{t}$, $C^{w}$, and $C^{b}$, 
respectively. See Figure \ref{fig:decomp} for an example.

\begin{figure}[ht]
\centering
\def\svgwidth{.5\columnwidth}
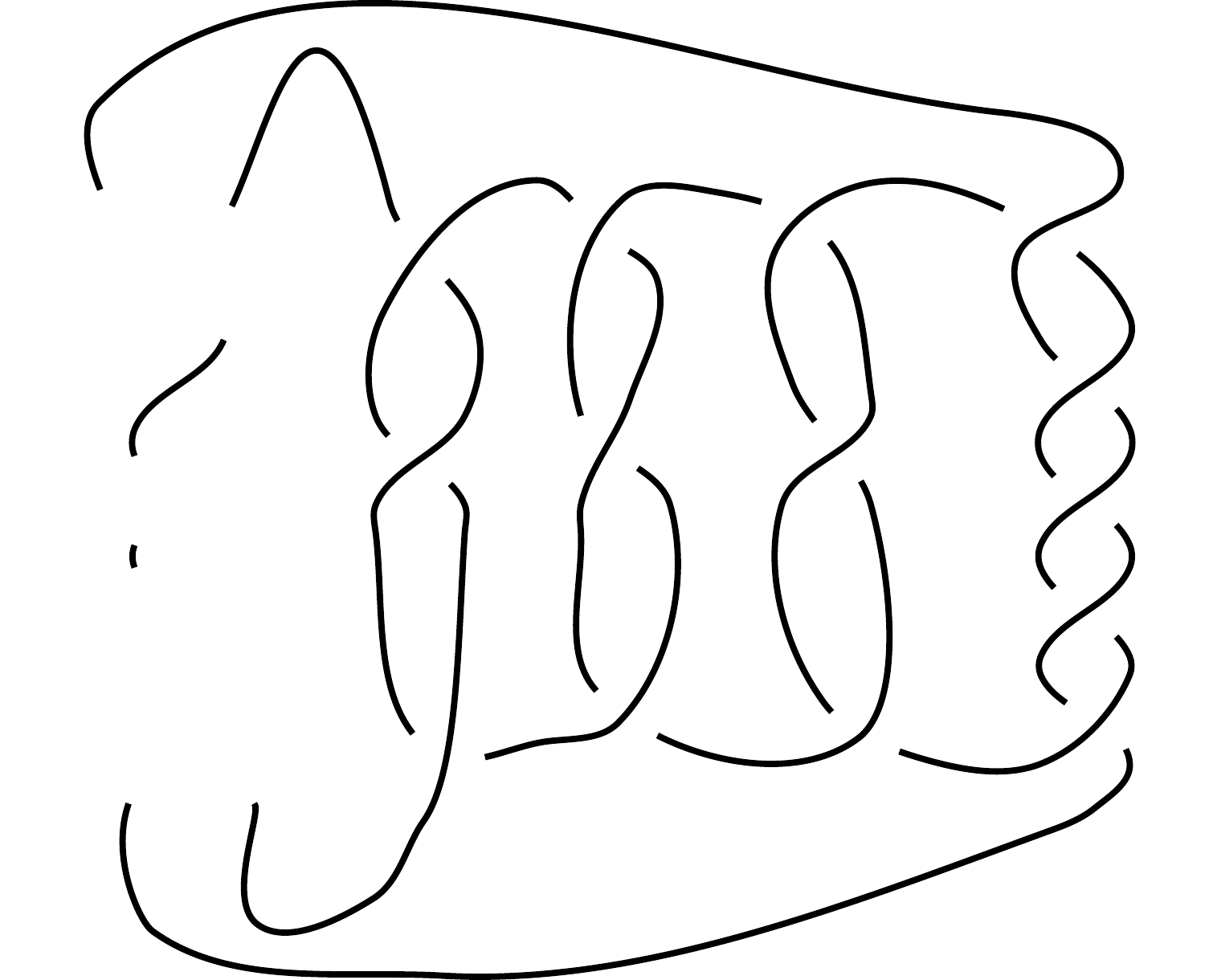
\caption{Skein element $\Sk = N(I_{k_0} \oplus (\Sk^t \times \Sk^w \times \Sk^b))$
of the pretzel knot $P(-5, 3, 3, 3, 5)$. We have $\Sk^t \in \TL^{2n}_{2mn}$, $\Sk^w \in \TL_{2mn}$, and $\Sk^b \in \TL^{2mn}_{2n}$ where $m=4$. } 
\label{fig:decomp}
\end{figure}
See Figure \ref{f.sbsm} for a guide to the labeling. The skein element $(K_+^n)_B$ consists of $n$ arcs on top in the region defining $\Sk^t$,  $n$ arcs on the bottom in the region defining $\Sk^b$, and $q_i-1$ sets of $n$ circles for the $i$th twist region in the region defining $\Sk^w$. The $n$ upper arcs are labeled by $S^u_1, \ldots, S^u_n$, and the $n$ lower arcs are labeled by $S^{\ell}_1, \ldots, S^{\ell}_n$, respectively. $C^{u}_j$ is the set of crossings whose corresponding segments in $(K_+^n)^G_B$ lie between the arcs $S_j^{u}$ and $S_{j+1}^u$. Similarly we define $C^{\ell}_j$ by reflection. 

For the crossings in the region defining $\Sk^w$, we divide each set of $n$ state circles into upper and lower half arcs as also shown in Figure \ref{f.sbsm}, and use an additional label $s$ for $1\leq s \leq q_i$. Thus the notation $C^{\ell, s}_{i, j}$, where $1\leq s \leq q_i$ for each twist region with $q_i$ crossings  and $1\leq j \leq n$ indicating a circle in the $n$-cable, means the crossings between the state circles $S^{\ell, s}_{i, j}$ and  $S^{\ell, s}_{i, j+1}$, see Figure \ref{f.sbsm}.

\begin{figure}[ht]
\centering
\footnotesize{
\def\svgwidth{.9\columnwidth}
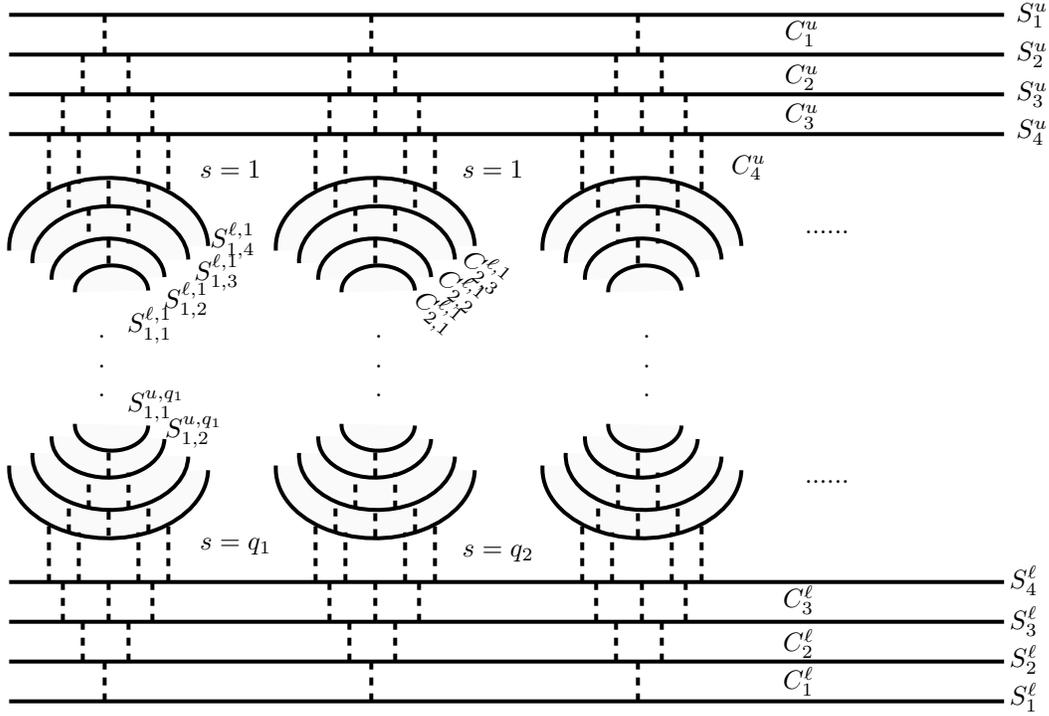}
\caption{Labeling of crossings, arcs, and circles from applying the all-$B$ state to $K^n_+$. In this example $n=4$. }
\label{f.sbsm}
\end{figure}  

It is helpful to see a local picture at each $n$-cabled crossing in $K_+^n$. 
\begin{figure}[H]
\centering
\footnotesize{
\def\svgwidth{.35\columnwidth}
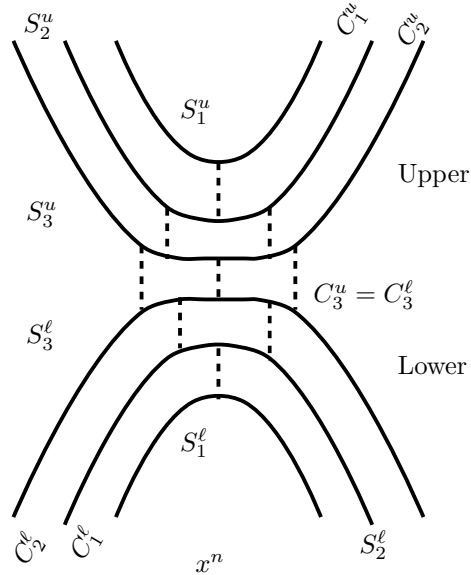}
\caption{Local labeling of $n^2$ crossings on the all-$B$ state of an $n$-cabled crossing. In this example $n=3$.}
\label{f.markings}
\end{figure}

 The goal of this subsection is to prove the following theorem. 
\begin{theorem} \label{t.pyramid}
Suppose a skein element $\Sk(k_0, \sigma)$ has parameters $k_i = k_i(\sigma)$ and $c=c(\sigma)$. Then there is a subset  $A'_{\sigma} \subseteq A_{\sigma}$ of crossings on which the Kauffman state $\sigma$ chooses the $A$-resolution,  such that
we have $A'_{\sigma} = A_{\sigma}^t \cup A_{\sigma}^w \cup A_{\sigma}^b$ denoting the crossings in the regions determining $\Sk^t$, $\Sk^{w}$, and $\Sk^{b}$, respectively, and the following conditions are satisfied.
\begin{itemize}

\item[(i)] $|A_{\sigma}^{w}| = \sum_{i=1}^m (q_i -2 )k_i^2$. The set $A^{w}_{\sigma} = \cup_{i=1}^m \cup_{s=1}^{q_i} \cup_{j=n-k_i+1}^{n} (u^s_{i, j} \cup \ell^s_{i, j})$ is a union of crossings with $u^s_{i,j} \subset C^{u,s}_{i, j}$ and $\ell^s_{i, j} \subset C^{\ell,s}_{i, j}$, such that 
\begin{itemize}
\item 
For each $n-k_i+1\leq j \leq n$, $u_{i, j}^s$, $\ell_{i, j}^s$ each has $j-n+k_i$ crossings. 
\item 
For each $n-k_i+2\leq j \leq n$ and a pair of crossings $x, x'$ in $u^s_{i, j}$ 
(resp. $\ell^s_{i, j}$) whose corresponding segments $e, e'$ in $(K_+^n)^G_B$ are adjacent (i.e., there is no other edge in $u^s_{i, j}$ 
between $e$ and $e'$), there is a crossing $x''$ in $u^s_{i, j-1}$ (resp. 
$\ell^s_{i, j-1}$), where the end of the corresponding segment $e''$ on 
$S^{u, s}_{i, j}$ (resp. $S^{\ell, s}_{i, j}$) lies between the ends of $e$ and $e'$.
\end{itemize} 
\item[(ii)] $|A_{\sigma}^t| = |A_{\sigma}^b| = \frac{c^2/4-c/2+ \sum_{i=1}^m  (k^2_i+ k_i)}{2}$.  The set $A_{\sigma}^t = \cup_{j=n-c/2+1}^n u_j$ is a union of crossings $u_j \subset C_j^u$, and the set $A_{\sigma}^b = \cup_{j=n-c/2+1}^n \ell_j$ is a union of crossings $\ell_j \subset C_j^\ell$ satisfying: 
\begin{itemize}
\item 
For $n-\frac{c}{2}+1 \leq j \leq n$, $u_j$ (resp. $\ell_j$) has $j-n+\frac{c}{2}$ crossings.
\item 
For each $n-\frac{c}{2}+2\leq j \leq n$ and a pair of crossings $x, x'$ in $u_j$ (resp. $\ell_j$) 
whose corresponding segments $e, e'$ in $(K_+^n)^G_B$ are 
adjacent (i.e., there is no other crossing in $u_j$ whose corresponding segment is between $e$ and $e'$), there is a crossing $x''$ in $u_{j-1}$ (resp. $\ell_{j-1}$), where the end of the 
corresponding segment $e''$ on $S^u_{j}$ (resp. $S^{\ell}_j$) lies between the ends of $e$ and 
$e'$.
\end{itemize}
\end{itemize} 
\end{theorem} 
It follows that $|A'_{\sigma}| = |A_{\sigma}^t|+|A_{\sigma}^w|+|A_{\sigma}^b| = \frac{c^2}{4}-\frac{c}{2}+ \sum_{i=1}^m (k^2_i+ k_i) + \sum_{i=1}^m (q_i-2)k_i^2$. The set of crossings $A'_{\sigma}$ is said to be in \emph{pyramidal position}.
\begin{proof} 
Statement (i) is a direct application to every set of $n$-cabled crossings in each twist region of $\Sk^w$ of the following result from \cite{L}. 

\begin{lemma}\cite[Lem. 3.7]{L} 
\label{l.pyramid}
Let $\Sk$ be a skein element in $\TL_{2n}$ consisting of 
a single $n$-cabled positive crossing $x^n$ with labels as shown in 
Figure \ref{f.markings}. 

If $(x^n)_{\sigma}$ for a Kauffman state $\sigma$ on $x^n$ has  $2k$ through 
strands, then $\sigma$ chooses the $A$-resolution on a set of $k^2$ crossings 
$C_{\sigma}$ of $x^n$, where $C_{\sigma} = \cup_{j=n-k+1}^n (u_j \cup \ell_j)$ is 
a union of crossings $u_j \subseteq C^u_j$ and $\ell_j \subseteq C^{\ell}_j$, 
such that

\begin{itemize}
\item 
For each $n-k+1\leq j \leq n$, $u_j$, $\ell_j$ each has $j-n+k$ crossings. 
\item 
For each  $n-k+2\leq j \leq n$, and a pair of crossings $x, x'$ in $u_j$ 
(resp. $\ell_j$) whose corresponding segments $c, c'$ in the all-$B$ 
state of $x^n$ are adjacent (i.e., there is no other edge in $C_{\sigma}$ 
between $c$ and $c'$), there is a crossing $x''$ in $u_{j-1}$ (resp. 
$\ell_{j-1}$), where the end of the corresponding segment $c''$ on 
$S^u_{j}$ (resp. $S^{\ell}_j$) lies between the ends of $c$ and $c'$.
\end{itemize} 
\end{lemma}

The same proof applies to the crossings in the strip $\Sk^{t}$ to show the existence of a set of crossings $A^t_{\sigma}$ satisfying (ii), see Figure \ref{f.sbs}. Reflection with respect to the horizontal axis will show (ii) for $\Sk^b$.

\begin{figure}[!htpb]
\centering
\def\svgwidth{.3\columnwidth}
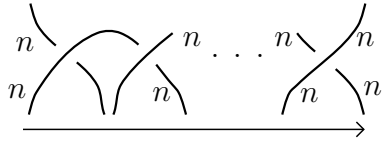
\caption{The arrow indicates the direction from left to right of the crossings in $\Sk^t$.}
\label{f.sbs}
\end{figure} 
\end{proof}

We will now apply what we know about the crossings on which a state $\sigma$ chooses the $A$-resolution from Theorem \ref{t.pyramid} to construct degree-maximizing states for given global through strands $c(\sigma)$ and parameters $\{k_i(\sigma)\}$. See Figure \ref{f.skeinex} for an example of a pyramidal position of crossings.
\begin{figure}[ht]
\def \svgwidth{.4\columnwidth}
\begingroup%
  \makeatletter%
  \providecommand\color[2][]{%
    \errmessage{(Inkscape) Color is used for the text in Inkscape, but the package 'color.sty' is not loaded}%
    \renewcommand\color[2][]{}%
  }%
  \providecommand\transparent[1]{%
    \errmessage{(Inkscape) Transparency is used (non-zero) for the text in Inkscape, but the package 'transparent.sty' is not loaded}%
    \renewcommand\transparent[1]{}%
  }%
  \providecommand\rotatebox[2]{#2}%
  \newcommand*\fsize{\dimexpr\f@size pt\relax}%
  \newcommand*\lineheight[1]{\fontsize{\fsize}{#1\fsize}\selectfont}%
  \ifx\svgwidth\undefined%
    \setlength{\unitlength}{404.54679234bp}%
    \ifx\svgscale\undefined%
      \relax%
    \else%
      \setlength{\unitlength}{\unitlength * \real{\svgscale}}%
    \fi%
  \else%
    \setlength{\unitlength}{\svgwidth}%
  \fi%
  \global\let\svgwidth\undefined%
  \global\let\svgscale\undefined%
  \makeatother%
  \begin{picture}(1,2.03514025)%
    \lineheight{1}%
    \setlength\tabcolsep{0pt}%
    \put(0,0){\includegraphics[width=\unitlength,page=1]{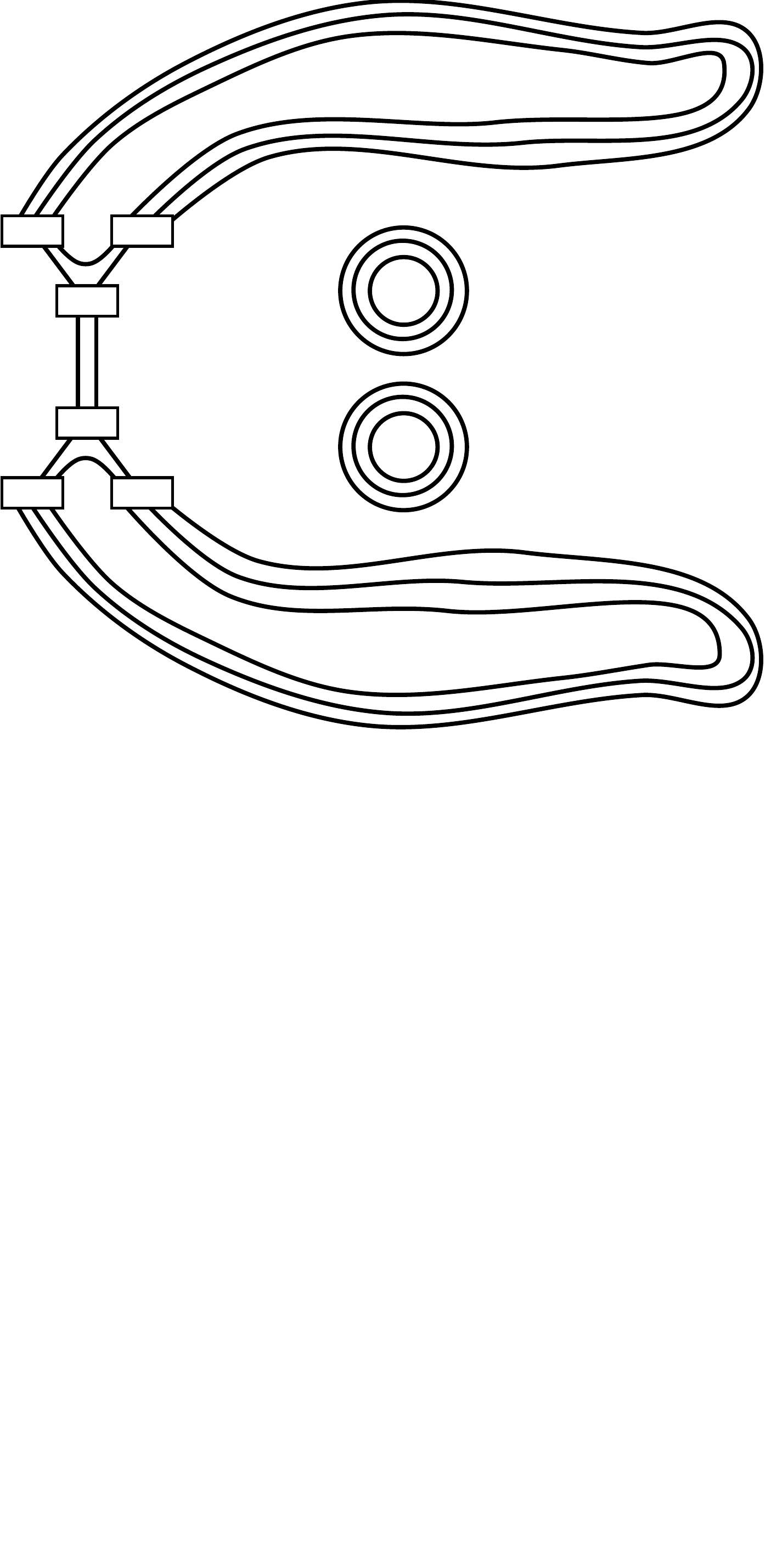}}%
    \put(0.02349275,1.56018241){\color[rgb]{0,0,0}\makebox(0,0)[lt]{\lineheight{0}\smash{\begin{tabular}[t]{l}$k_0$\end{tabular}}}}%
    \put(0,0){\includegraphics[width=\unitlength,page=2]{Skeinexample.pdf}}%
    \put(0.03545798,0.47134628){\color[rgb]{0,0,0}\makebox(0,0)[lt]{\lineheight{0}\smash{\begin{tabular}[t]{l}$k_0$\end{tabular}}}}%
    \put(0,0){\includegraphics[width=\unitlength,page=3]{Skeinexample.pdf}}%
    \put(0.14669357,0.47134628){\color[rgb]{0,0,0}\makebox(0,0)[lt]{\lineheight{0}\smash{\begin{tabular}[t]{l}$k_0$\end{tabular}}}}%
    \put(0.13472834,1.56018241){\color[rgb]{0,0,0}\makebox(0,0)[lt]{\lineheight{0}\smash{\begin{tabular}[t]{l}$k_0$\end{tabular}}}}%
    \put(0,0){\includegraphics[width=\unitlength,page=4]{Skeinexample.pdf}}%
  \end{picture}%
\endgroup%

\caption{\label{f.skeinex} A minimal state $\tau$ is shown with $n=3$ and $c(\tau)=6$ global through strands. In the top picture one can see the pyramidal position of the crossings $A_{\tau}$ as described by Theorem \ref{t.pyramid}. The skein element $\Sk(k_0,\tau)$ with $k=(k_0, 0, 0, 2, 1)$ resulting from applying $\tau$ is shown below.} 
\end{figure}

\subsection{Minimal states are taut and their degrees are $\delta(n,k)$}
The contribution of the state $(k_0,\sigma)$ to the state sum is $G_{k_0}(v)v^{\sgn(\sigma)} \langle N(I_{k_0} \oplus (K^n_+)_{\sigma}) \rangle$ as in \eqref{eq.ssum}.  We denote its $v$-degree by $d(k_0,\sigma)$.

Recall the skein element $\Sk(k_0,\sigma) =  N(I_{k_0} \oplus (K^n_+)_{\sigma})$. Also recall $A_\sigma$ denotes the set of crossings on which $\sigma$ chooses the $A$-resolution, and $|A_\sigma|$ is the number of crossings in $A_\sigma$. A \emph{minimal state} with tight parameters $c, k$ (i.e., $k_0=k_1+\cdots+k_m=\frac{c}{2}$) has the least $|A_\sigma|$ in $st(c, k)$. Let $\oo(A_\sigma)$
denote the number of circles of $\overline{\Sk(k_0,\sigma)}$, which is the skein element obtained by replacing all the Jones-Wenzl idempotents in $\Sk(k_0,\sigma)$ by the identity, respectively.

\begin{lemma} \label{l.mindegree}
A minimal state $(k_0,\tau)$ with $c(\tau)$ through strands and tight $c, k$ has $A_{\tau}$ in pyramidal position as specified in Theorem  \ref{t.pyramid} and distance $|A_\tau|$ from the all-$B$ state given by 
\[
|A_{\tau}| = 2\left( (\sum_{i=1}^m k_i)\frac{(\sum_{i=1}^m k_i-1)}{2} 
+ \sum_{i=1}^{m} \frac{k_i(k_i+1)}{2} 
\right) + \sum_{i=1}^m (q_i-2) k_i^2 \,.
\] Moreover, 
\be
\label{eq.tauDnk}
G_{k_0}(v)v^{\sgn(\sigma)} \langle N(I_{k_0} \oplus (K^n_+)_{\sigma}) \rangle = (-1)^{q_0(n-k_0)+n+k_0
+\sum_{i=1}^{m}(n-k_i)(q_i-1)}v^{\delta(n, k)}+l.o.t. 
\ee
\end{lemma}

\begin{proof}
Observe that minimal states $\tau$ have corresponding crossings
$A_\tau$ in pyramidal position. Moreover, if $A_\tau$ is pyramidal, then $|A_\tau|$ determines the number of circles $\oo(A_\tau)$. The skein element $\Sk(k_0,\tau)$ is adequate as long as 
$k_0\leq \sum_{i=1}^m k_i$. This means that no circles of $\overline{\Sk(k_0, \tau)}$ goes through a location where there was an idempotent twice. 
 Thus by \cite[Lem. 4]{Arm13}, we have 
\[
\deg_v v^{\sgn(\tau)}\langle \Sk(k_0,\tau) \rangle 
= \deg_v  v^{\sgn(\tau)} \langle \overline{\Sk(k_0,\tau)} \rangle\,, 
\] 
and we simply need to determine the number of circles in 
$\overline{\Sk(k_0,\tau)}$ and $\sgn(\tau)$  in order to compute the degree 
of the Kauffman bracket. This is completely specified by the pyramidal position of $A_{\tau}$ by just applying the Kauffman state. With the assumption that $k_0=\sum_{i=1}^m k_i = \frac{c}{2}$ since $c, k$ is tight, the degree is then
\begin{align*} 
d(k_0,\tau) 
&= \underbrace{\sum_{i=1}^m q_in^2 - 2(2\left( 
\frac{\left(\sum_{i=1}^m k_i \right)(\left(\sum_{i=1}^m k_i \right)-1)}{2} 
+ \sum_{i=1}^{m} \frac{k_i(k_i+1)}{2} \right) 
+ \sum_{i=1}^m (q_i-2) k_i^2)}_{\sgn(\tau)}  \\ 
&+ \underbrace{2\left(2n-(\left(\sum_{i=1}^m k_i \right)-k_0)
+\sum_{i=1}^{m}(n-k_i)(q_i-1)\right)}_{2o(A_{\tau})} \\ 
&+ \underbrace{q_0(2n-2k_0+\frac{2n^2-4k_0^2}{2})
+2k_0-2n}_{\text{fusion and untwisting}}.
\end{align*} 
The sign of the leading term is given by 
\begin{align*}
(-1)^{\underbrace{q_0(n-k_0)+n+k_0}_{\text{fusion and untwisting}} +\underbrace{o(A_\tau)}_{\text{number of circles}}}  &= (-1)^{q_0(n-k_0)+n+k_0
+\sum_{i=1}^{m}(n-k_i)(q_i-1)} . \end{align*}
\end{proof}

\begin{lemma}
\label{lem.minimal} 
 Minimal states are taut. In other words, given $c, k$ tight, we have 
 $$\max_{\sigma: c(\sigma)= c, k_i(\sigma) = k_i}  d(k_0, \sigma) = d(k_0, \tau), $$ where $\tau$ is a minimal state with $c(\tau) = c$ and $k_i(\tau) = k_i$. 
\end{lemma}
\begin{proof}
Note that for \emph{any} state $\sigma$ with corresponding skein element $\Sk(k_0,\sigma)$, we have 
\[
A_{\tau} \subseteq A_{\sigma} 
\] 
for a minimal state $\tau$ with the same parameter set $(c, k)$ by Theorem \ref{t.pyramid}, and $d(k_0,\tau) = d(k_0,\tau')$ for two minimal states $\tau, \tau'$ with the same parameters $c(\tau) = c(\tau')$ and $k_i(\tau) = k_i(\tau')$ by Lemma \ref{l.mindegree}. This implies $d(k_0,\sigma) \leq d(k_0, \tau)$ by Lemma \ref{l.split}.
\end{proof}

\subsubsection{Constructing minimal states}

\begin{lemma} \label{l.constructm}
A minimal state exists for any tight $c, k$, where $c$ is an even integer between $0$ and $2n$ and $k_0 = \sum_{i=1}^m k_i = \frac{c}{2}$. 
\end{lemma} 

\begin{proof} 
It is not hard to see that at an $n$-cabled crossing $x^n$ in a twist region with $q_i$  crossings in $\Sk^{w}$, for any $0 \leq k_i\leq n$  there is always a minimal state 
giving $2k_i$ through strands. For an $n$-cabled crossing $x^n$ in 
$\Sk^t$ or $\Sk^b$, it is also not hard to see that we may take 
the pyramidal position for the minimal state for the upper 
half (or bottom half, for $\Sk^b$) of each crossing in $x^n$ in $C^u_n$ (or $C^{\ell}_n$) and in $C^{\ell, s}_{i, j}$ (or $C^{u, s}_{i, j}$) for each twist region.  

What remains to be shown is that a minimal state \emph{overall} always exists, given the set of parameters $\{k_i\}$ and $c$ total through strands for crossings in the top and bottom strips delimited by
$\{S^u_{j}\}_{j=1}^n$ and $\{ S^{\ell}_j \}_{j=1}^n$. To see this, we take 
the leftmost position for the crossings $x^n$ in $(\cup_{i=1}^m \cup^n_{j=1} C^{\ell, 1}_{i, j} )\cup C^u_{n}$ with $\{2k_i\}$ through strands, which we already know to exist. Given two crossings $x$ and $x'$ in $C^u_n$ whose corresponding segments in $\Sk(k_0, B)$ have ends on $S^u_n$ we can always find another crossing $x''$ in $C^u_{n-1}$, the end of whose corresponding segment on $S^u_n$ lies between those of $x$ and $x'$, because the previously chosen crossings in $C^u_n$ are leftmost. Pick the leftmost possible and repeat to choose crossings in $C^u_{j}$ for $n-k+1\leq j \leq n-2$. We pick crossings in the bottom strip by reflection. For the remaining $n$-cabled crossings $x^n$ in $\Sk^w$ in a twist region corresponding to $q_i$, any subset of crossings in pyramidal position with $2k_i$ through strands will complete the description of a minimal state satisfying the conditions in the lemma.  
\end{proof} 

\begin{lemma} \label{l.excludeodd} Let $\sigma$ be a state with $c = c(\sigma)$ and $k_i = k_i(\sigma)$ which is not tight, that is, $\sum_{i=1}^m k_i > \frac{c}{2}$ or $k_0 < \frac{c}{2}$, then 
$d(k_0, \sigma) < d(k_0, \tau)$ , where $\tau$ is a minimal state with $c(\tau) = c$ through strands.
\end{lemma} 
\begin{proof}
For the case $\sum_{i=1}^m k_i > \frac{c}{2}$, we can apply Theorem \ref{t.pyramid} to conclude that there is a minimal state $\tau$ (there may be multiple such states) such that 
\[A_{\tau} \subset A_{\sigma},  \] with $k_i(\tau) \leq k_i(\sigma)$ for each $i$. There must be some $i$ for which $k_i(\tau) < k_i(\sigma)$. Applying the $B$-resolution to the additional crossings to obtain a sequence of states from $\tau$ to $\sigma$, we see that it must contain two consecutive terms that merge a pair of circles. 

If $k_0 < \frac{c}{2}$, since $d(k_0, \sigma)$ increases monotonically in $k_0$ in $G_{k_0}(v)$ from the fusion and untwisting formulas, we can see that $d(k_0, \sigma) < d(c/2, \tau)$. 
\end{proof}

\subsection{Enumerating all taut states}
By Lemma \ref{lem.minimal}, we have shown that every taut state contains a minimal state. Next we show that every taut state is obtained from a unique such minimal state $\tau$ by changing the resolution from $B$-to $A$-on a set of crossings $F_{\tau}$. We show that any taut state $\sigma$ with $c(\sigma) = c(\tau)$ and $k_i(\sigma) = k_i(\tau)$ containing $\tau$ as the \emph{leftmost} minimal state, to be defined below, satisfies $A_{\sigma} = A_{\tau} \cup p$, where $p$ is any subset of $F_{\tau}$.

All the circles here in the definitions and theorems are understood with possible extra labels $u, \ell, s, i, j$ indicating where they are in the regions defining $\Sk^t, \Sk^w$, and $\Sk^b$. To simplify notation we do not show these extra labels. 
\begin{definition} 
\label{d.Ftau} 
For each $x\in A_{\tau}$ between $S_{j-1}$ and $S_{j}$, let $R_x$ be the set 
of crossings to the right of $x$ between $S_{j-1}$ and $S_{j}$, but to 
the left of any $x'\in A_{\tau}$ between $S_{j-2}$ and $S_{j-1}$, and any 
$x''\in A_{\tau}$ between $S_{j}$ and $S_{j+1}$. We define
the following (possibly empty) subset $F_{\tau}$ of crossings of $K^n$.
\[
F_{\tau}:= \cup_{x\in A_{\tau}} R_x \,. 
\] See Figure \ref{f:addededges} and \ref{f:exedges} for examples.
\end{definition} 

\begin{figure}[!htpb]
\centering
\def \svgwidth{.5\columnwidth}
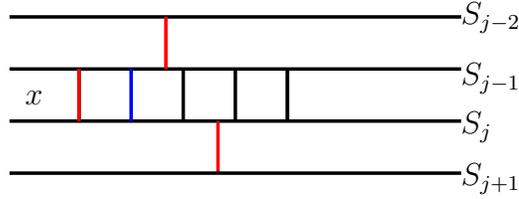
\caption{Only the blue edge  is in $R_x$ because of the presence of the 
top and bottom red edges.}
\label{f:addededges}
\end{figure}

\begin{figure}[!htpb]
\centering
\def \svgwidth{.5\columnwidth}
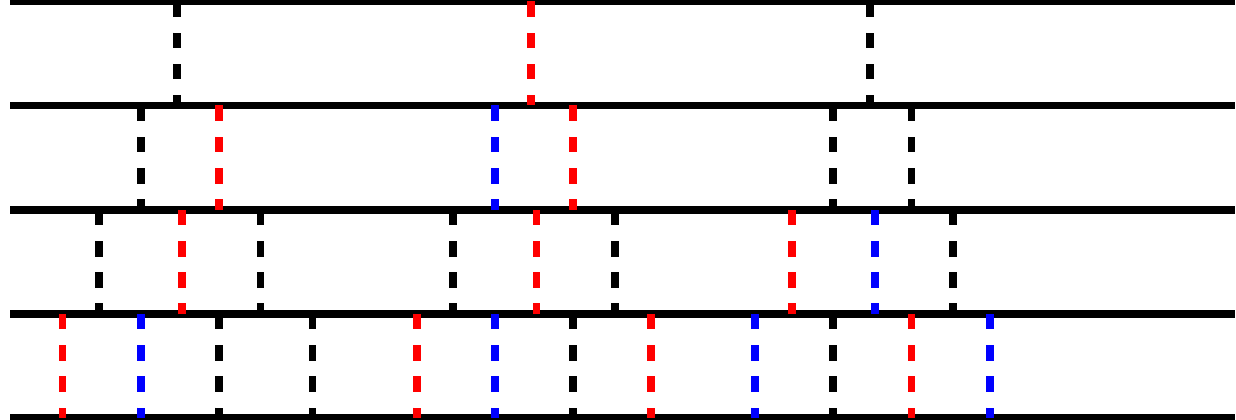
\caption{An example of $F_{\tau}$ with edges shown in blue with the 
minimal state $\tau$ shown as red edges.}
\label{f:exedges}
\end{figure}

\begin{definition} 
\label{def.distance}
Given a set of crossings $C$ of $K^n$, a  crossing $x \in C$, and $1\leq j \leq n$, 
define the \emph{distance $|x|_C$ of a 
crossing $x\in C$ from the left} to be 
\[
|x|_C:= \text{For $x\in C_j$, the \# of edges in 
$\Sk^G(k_0, B)$ to the left of $x$ between $S_j$ and $S_{j+1}$}. 
\]
The \emph{distance} of the set $C$ from the left is defined as
\[
\sum_{ x \in C} |x|_C \,.  
\] 
\end{definition}
Given any state $\sigma$ with tight parameters $c, k$, 
we extract the \emph{leftmost} minimal state $\tau_{\sigma}$ where 
$A_{\tau_{\sigma}} \subseteq A_{\sigma}$, i.e., there is no other minimal state $\tau'$ such that $A_{\tau'}\subset A_{\sigma}$, and the distance of 
$A_{\tau'}$ from the left is less than the distance of $A_{\tau_{\sigma}}$ 
from the left. 
 
\begin{figure}[ht]
\centering
\def \svgwidth{.5\columnwidth}
\begingroup%
  \makeatletter%
  \providecommand\color[2][]{%
    \errmessage{(Inkscape) Color is used for the text in Inkscape, but the package 'color.sty' is not loaded}%
    \renewcommand\color[2][]{}%
  }%
  \providecommand\transparent[1]{%
    \errmessage{(Inkscape) Transparency is used (non-zero) for the text in Inkscape, but the package 'transparent.sty' is not loaded}%
    \renewcommand\transparent[1]{}%
  }%
  \providecommand\rotatebox[2]{#2}%
  \newcommand*\fsize{\dimexpr\f@size pt\relax}%
  \newcommand*\lineheight[1]{\fontsize{\fsize}{#1\fsize}\selectfont}%
  \ifx\svgwidth\undefined%
    \setlength{\unitlength}{375.50505066bp}%
    \ifx\svgscale\undefined%
      \relax%
    \else%
      \setlength{\unitlength}{\unitlength * \real{\svgscale}}%
    \fi%
  \else%
    \setlength{\unitlength}{\svgwidth}%
  \fi%
  \global\let\svgwidth\undefined%
  \global\let\svgscale\undefined%
  \makeatother%
  \begin{picture}(1,2.03476677)%
    \lineheight{1}%
    \setlength\tabcolsep{0pt}%
    \put(0,0){\includegraphics[width=\unitlength,page=1]{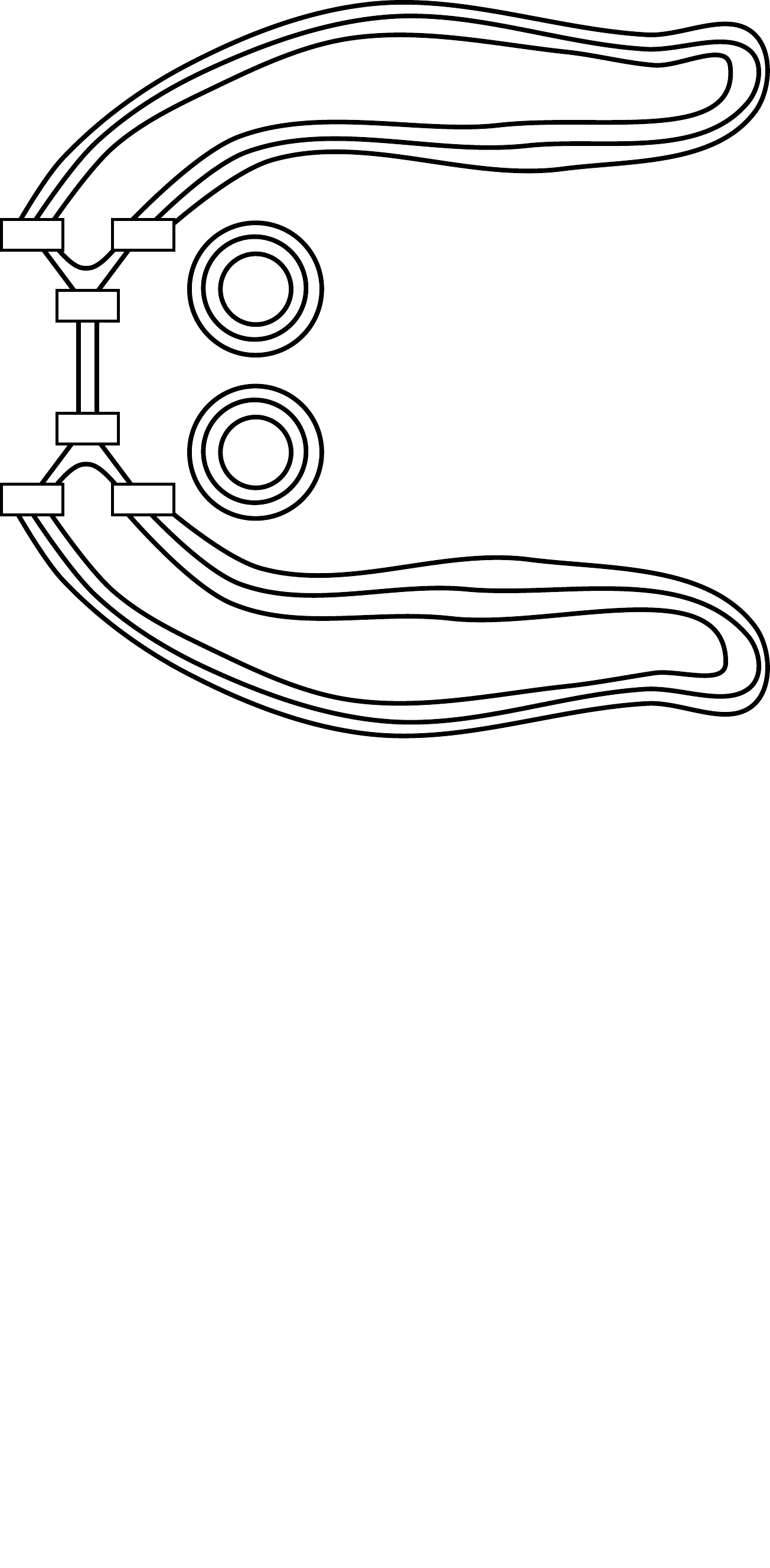}}%
    \put(0.04907679,1.55305157){\color[rgb]{0,0,0}\makebox(0,0)[lt]{\lineheight{0}\smash{\begin{tabular}[t]{l}$k_0$\end{tabular}}}}%
    \put(0,0){\includegraphics[width=\unitlength,page=2]{ntautstateex.pdf}}%
    \put(0.13296381,1.55305157){\color[rgb]{0,0,0}\makebox(0,0)[lt]{\lineheight{0}\smash{\begin{tabular}[t]{l}$k_0$\end{tabular}}}}%
    \put(0,0){\includegraphics[width=\unitlength,page=3]{ntautstateex.pdf}}%
    \put(0.04907679,0.47666848){\color[rgb]{0,0,0}\makebox(0,0)[lt]{\lineheight{0}\smash{\begin{tabular}[t]{l}$k_0$\end{tabular}}}}%
    \put(0,0){\includegraphics[width=\unitlength,page=4]{ntautstateex.pdf}}%
    \put(0.13695843,0.47666848){\color[rgb]{0,0,0}\makebox(0,0)[lt]{\lineheight{0}\smash{\begin{tabular}[t]{l}$k_0$\end{tabular}}}}%
  \end{picture}%
\endgroup%

\caption{On top, a taut state having the same degree as a minimal state but is not 
equal to it. The bottom picture shows the resulting skein element from applying the state. We have $c=6$, $k_1 = 0$, $k_2 = 0$, $k_3 = 2$, and $k_4=1$ as the minimal state in Figure \ref{f.skeinex}, and the 
thickened red edges indicate the difference from a minimal state 
with the same parameters. Choosing the $A$-resolution at each of the 
thickened red edges splits off a circle.}
\label{f:ntautstateex}
\end{figure} 

\begin{lemma} 
\label{l.taut} 
A Kauffman state $\sigma$ with tight parameters $c(\sigma), \{k_i(\sigma)\}$ is taut if and only if $A_{\sigma}$ may be written as 
\[ 
A_{\sigma} = A_{\tau_{\sigma}} \cup p \,  
\] 
where $\tau_{\sigma}$ is the leftmost minimal state from $\sigma$ such that 
$A_{\tau_{\sigma}} \subseteq A_{\sigma}$, and $p$ is a subset of $F_{\tau_{\sigma}}$. See Figure \ref{f:ntautstateex} for an example of a taut state that is not a minimal state, and how it is obtained from the leftmost minimal state that it contains.
\end{lemma} 

\begin{proof} 
By construction, if a state $\sigma$ is such that 
\[ 
A_{\sigma} = A_{\tau_{\sigma}} \cup p 
\] 
where $p$ is a subset of $F_{\tau_{\sigma}}$, then $\sigma$ is a taut state. 

Conversely, suppose by way of contradfiction that $\sigma$ is taut, 
which means that it has the same parameters $(c, k)$ as its leftmost 
minimal state $\tau_{\sigma}$ with the same degree, but that there is a crossing 
$x\in A_{\sigma}$ and $x\notin F_{\tau_{\sigma}}$. Then there are two cases:
\begin{enumerate}
\item 
$x$ is to the left or to the right of all the edges in $A_{\tau_{\sigma}}$. 
\item 
$x \in C_j$ is between $x', x'' \in C_j$ in $A_{\tau_{\sigma}}$
for some $j$. 
\end{enumerate} 

In both cases we consider the state $\sigma'$ where 
\[
A_{\sigma'} = A_{\tau_{\sigma}}\cup \{x\},
\] 
and we assume that taking the $A$-resolution on $x$ splits off a circle 
from the skein element $\overline{\Sk(k_0,\sigma)}$. Otherwise, by 
Lemma \ref{l.split} applied to a sequence from $\tau_{\sigma}$ to $\sigma$ starting with changing the resolution from $B$ to $A$ on $x$, 
\[
\deg_v v^{\sgn(\sigma)}\langle \overline{\Sk(k_0,\sigma)} 
\rangle < \deg_v v^{\sgn(\tau_{\sigma})}\langle 
\overline{\Sk(k_0,\tau_\sigma)}\rangle \,, 
\] 
a contradiction to $\sigma$ being taut. 

In case (1), the state $\sigma'$ has parameters $(c, k')$ such that 
$\sum_{i=1}^m k_i' <  \sum_i^m k_i$. If each step of a sequence from 
$\sigma'$ to $\sigma$ splits a circle in order to maintain the degree, 
then the parameters for $\sigma$,  and hence the number of global through 
strands of $\Sk(k_0, \sigma)$ will differ from $\Sk(k_0, \tau_{\sigma})$, a contradiction.  

In case (2), we have that $x \notin F_{\tau_{\sigma}}$ must be an edge of the following 
form between a pair of edges $x', x''$ as indicated in the generic local 
picture shown in Figure \ref{f:addedgesv}, since $\tau_{\sigma}$ is assumed to be leftmost.

\begin{figure}[!htpb]
\centering
\def \svgwidth{.3\columnwidth}
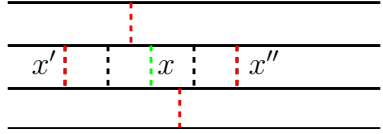
\caption{The crossing $x$ corresponds to the green edge.}
\label{f:addedgesv}
\end{figure}
Choosing the $A$-resolution at $x$ merges a pair of circles in $\overline{\Sk(k_0, \tau_{\sigma})}$ which means that $d(k_0, \sigma) < d(k_0, \tau_{\sigma})$, a contradiction.
\end{proof}

\subsection{Adding up all taut states in $st(c,k)$}

Note that in general there may be many taut states $\sigma$ with
fixed parameters $(c = c(\sigma), k = k(\sigma))$.

\begin{theorem} 
\label{thm:cancelling}
Let $c, k = \{k_i\}_{i=1}^m$ be tight. The sum 
\begin{equation} 
\label{eq:sum} 
\sum_{\sigma \text{ taut} : c(\sigma)=c, k_i(\sigma) = k_i} 
v^{\sgn(\sigma)} \langle \Sk(k_0,\sigma) \rangle  = (-1)^{q_0(n-k_0)+n+k_0
+\sum_{i=1}^{m}(n-k_i)(q_i-1)} v^{d(k_0, \tau)} + l.o.t., 
\end{equation} 
where $\tau$ is a minimal state in the sum.
\end{theorem}

We are finally ready to prove Theorem \ref{thm:cancelling}.

\begin{proof}
Every minimal state with parameters $c, k$ may be obtained from the 
leftmost minimal state of the entire set of minimal states $\mathcal{M}$ by transposing to the right. Now we organize the sum \eqref{eq:sum} by 
putting it into equivalence classes of states indexed by the leftmost 
minimal state $\tau_{\sigma}$. We may write

\begin{align*}
\sum_{\sigma \text{ taut} : c(\sigma)=c, k_i(\sigma) = k_i}
v^{\sgn(\sigma)}\langle \Sk(k_0, \sigma) 
\rangle &= \sum_{\text{$\tau$ minimal}} \ \ 
\sum_{\sigma \,: \,  \tau_{\sigma} = \tau} v^{\sgn(\sigma)}\langle \Sk(k_0, \sigma) 
\rangle.
\\
\intertext{By Lemma \ref{l.taut}, this implies}
\sum_{\sigma \text{ taut} : c(\sigma)=c, k_i(\sigma) = k_i}
v^{\sgn(\sigma)}\langle \Sk(k_0, \sigma) 
\rangle &= \sum_{\text{$\tau$ minimal}} 
\sum_{j=0}^{|F_{\tau}|} 
\binom{|F_{\tau}|}{j} v^{\sgn(\tau)-2j}(-v^2-v^{-2})^{o(A_{\tau})+j} \,.
\end{align*}
If $F_{\tau} \not= \emptyset$, then by a direct computation,  
\begin{align*}
\deg_v \left(\sum_{j=0}^{|F_{\tau}|} \binom{|F_{\tau}|}{j} 
v^{\sgn(\tau)-2j}(-v^2-v^{-2})^{o(A_\tau)+j} \right) 
& = \sgn(\tau) + 2o(A_\tau)-4|F_{\tau}| \\ & < 
\deg_v \left( v^{\sgn(\tau)} \langle \Sk(k, \tau) \rangle \right) = \delta(n, k)\,
\end{align*}
by Lemma \ref{l.mindegree}.

Every taut state can be grouped into a nontrivial canceling sum 
except for the rightmost minimal state. Thus it remains and determines the degree of the sum. 
\end{proof} 

\subsection{Proof of Theorem~\ref{thm.Delta}}
\label{sub.Delta}

Recall that $J_{K,n+1} = \sum_{c, k} \mathcal{G}_{c, k}$ and 
\[\mathcal{G}_{c, k} = \sum_{k_0}G_{k_0}(v)\sum_{\sigma: k_i(\sigma)=k_i, c(\sigma)=c} v^{\sgn(\sigma)} \langle N(I_{k_0}\oplus (K^n_+)_\sigma) \rangle\]
By the fusion and untwisting formulas we have 
\[
G_{k_0}(v)= (-1)^{q_0(n-k_0)}
\frac{\triangle_{2k_0}}{\theta(n, n, 2k_0)} 
v^{q_0(2n-2k_0+n^2-k_0^2)}.
\]
We apply the previous lemmas to compute for each $c, k$ the $v$-degree of the sum 
\[
\sum_{\sigma: k_i(\sigma)=k_i, c(\sigma) = c}v^{\sgn(\sigma)}\langle N(I_{k_0}\oplus (K^n_+)_\sigma) \rangle.
\]
When $c, k$ is tight the top degree part of the sum is $\mathcal{G}_{c, k}^{taut}$. By Theorem \ref{thm:cancelling}, we have that the coefficient and the degree of the leading term are given by a minimal state $\tau$ with parameters $c, k$.  The degree is computed to be $\delta(n, k)$ in Lemma \ref{l.mindegree}, which also determines the leading coefficient. 

When $\sigma$ is a state such that $c, k$ is not tight, and $k_0 \geq c(\sigma)/2$ or $k_0 \geq \sum_{i=1}^m k_i(\sigma)$, Lemma \ref{l.zero} says that $\Sk(k_0, \sigma)$ is zero.  Otherwise, Lemma \ref{l.excludeodd} says that there exists a taut state corresponding to a tight $c', k'$ that has strictly higher degree. 
\qed


\section{Quadratic integer programming}
\label{sec.QIP}

In this section we collect some facts regarding real and lattice optimization
of quadratic functions.

\subsection{Quadratic real optimization}
\label{sub.QIPr}

We begin with considering the well-known case of real optimization.

\begin{lemma}
\label{lem.prog1}
Suppose that $A$ is a positive definite $m \times m$ matrix and $b \in \BR^m$.
Then, the minimum
\begin{equation}
\min_{x \in \BR^m} \frac{1}{2} x^t A x + b \cdot x
\end{equation}
is uniquely achieved at $x=-A^{-1}b$ and equals $-\frac{1}{2} b^t A b$.
\end{lemma}

\begin{proof}
The function is proper with the only critical point at $x=-A^{-1}b$ which is
a local minimum since the Hessian of $A$ is positive definite.
\end{proof}

For a vector $v \in \BR^m$, we let $v_i$ for $i=1,\dots, m$ denote its
$i$th coordinate, so that $v=(v_1,\dots,v_m)$. When $v_i$'s are nonzero for all 
$i$, we set $v^{-1}=(v_1^{-1},\dots,v_m^{-1})$.

The next lemma concerns optimization of convex separable 
functions $f(x)$, that is, functions of the form
\be
\label{eq.fabx}
f(x)= \sum_{i=1}^m f_i(x_i), \qquad f_i(x_i)= a_i x_i^2 + b_i x_i
\ee
where $a_i >0$ and $b_i$ are real for all $i$. The terminology follows
Onn~\cite[Sec.3.2]{Onn}. 

\begin{lemma}
\label{lem.prog2}
\rm{(a)} 
Fix a separable convex function $f(x)$ as in~\eqref{eq.fabx} 
and a real number $t \in \BR$. Then the minimum
\begin{equation}
\min \{ f(x) \,\, | \,\, \sum_i x_i=t, \,\, x \in \BR^m \}
\end{equation}
is uniquely achieved at $x^*(t)$ where 
\begin{equation}
\label{eq.x*}
x_i^*(t)= \frac{a_i^{-1} t 
+ \frac{1}{2} \sum_j (b_j-b_i) a_i^{-1} a_j^{-1}}{\sum_j a_j^{-1}}, 
\end{equation}
and
\begin{equation}
\label{eq.fx*}
f(x^*(t)) = \frac{1}{1 \cdot a^{-1}} t^2 
+ \frac{b \cdot a^{-1}}{1 \cdot a^{-1}} t + s_0(a,b) \,
\end{equation}
where $1 \in \BZ^m$ denotes the vector with all coordinates equal to $1$, and $s_0(a, b)$ is a rational function in coordinates $a = (a_1, \ldots, a_m)$ and $b = (b_1, \ldots, b_m)$. 
\newline
\rm{(b)} If $t \gg 0$, then the minimum 
\begin{equation}
\min \{ f(x) \,\, | \,\, \sum_i x_i=t, \,\, x \in \BR^m, \,\, 0 \leq x_i,  \,\,\,
i=1,\dots,m \}
\end{equation}
is uniquely achieved at~\eqref{eq.x*} and given by~\eqref{eq.fx*}.
\end{lemma}
Note that the coordinates of the minimizer $x^*(t)$ are linear functions
of $t$ for $t \gg 0$; we will call such minimizers linear. It is obvious that
the minimal value is then quadratic in $t$ for $t \gg 0$.

\begin{proof}
Let $f(x)=\sum_j a_j x_j^2 + b_j x_j$ and $g(x)=\sum_j x_j$ and
use Lagrange multipliers.
$$
\begin{cases}
\nabla f = \lambda \nabla g \\
g = t \,.
\end{cases}
$$
So, $2 a_j x_j + b_j = \lambda$ for all $j$, hence 
$x_j + b_j/(2 a_j)=\lambda/(2 a_j)$ for all $j$. Summing up, we get 
$t+\sum_j b_j/(2 a_j)=\lambda \sum_j 1/(2 a_j)$. Solving for $\lambda$,
we get $\lambda=\frac{2t+ \sum_j b_j a_j^{-1}}{\sum_j a_j^{-1}}$ and using
$$
x_i=\frac{\lambda-b_i}{2 a_i} 
= \frac{2t+\sum_j (b_j-b_i)a_j^{-1}}{2 a_i \sum_j a_j^{-1}}=
\frac{a_i^{-1} t 
+ \frac{1}{2} \sum_j (b_j-b_i) a_i^{-1} a_j^{-1}}{\sum_j a_j^{-1}}
\,,
$$ 
Equation~\eqref{eq.x*} follows. Observe that $x^*(t)$ is an affine linear
function of $t$. It follows that $f(x^*(t))$ is a quadratic function of $t$.
An elementary calculation by plugging in $x^*$ into $f$ gives~\eqref{eq.fx*} for an explicit rational
function $s_0(a,b)$, which is the portion of $f(x^*(t))$ that does not involve $t$. 

If in addition $t \gg 0$ observe that $x^*(t)=\frac{t}{1 \cdot a^{-1}} a^{-1}
+O(1)$, therefore $x^*(t)$ is in the simplex $x_i \geq 0$ for all $i$
and $\sum_i x_i =t$. The result follows. 
\end{proof}

\subsection{Quadratic lattice optimization}
\label{sub.QIP}

In this section we discuss the lattice optimization problem
\begin{equation}
\label{eq.QIPf}
\min \{ f(x) \,\, | \,\, A x = t, \,\, x \in \BZ^m, \,\, 0 \leq x \leq t \}
\end{equation}
for a nonnegative integer $t$, where $A=(1,1,\dots,1)$ is a $1 \times m$
matrix and $f(x)$ is a convex separable function~\eqref{eq.fabx} with
$a, b \in \BZ^m$ with $a>0$. We will follow the terminology and notation from Onn's book~\cite{Onn}. 
In particular the set $x \in \BZ^m$ satisfying
the above conditions $A x = t$ and $0 \leq x_i \leq t$ is called a feasible set. 
Lemma 3.8 of Onn~\cite{Onn} gives a necessary and sufficient condition 
for a lattice vector $x$ to be optimal. In the next lemma, suppose that a 
feasible $x \in \BZ^m$ is non-degenerate, that is, $x_i <t$ and $x_j>0$ 
for all $i,j$. Note that this is not a serious restriction since otherwise
the problem reduces to a lattice optimization problem of the same shape
in one dimension less.

\begin{lemma}
\cite{Onn}
Fix a feasible $x \in \BZ^m$ which is non-degenerate. Then it is optimal (i.e.,
a lattice optimizer for the problem~\eqref{eq.QIPf}) 
if and only if it satisfies the certificate
\be
\label{eq.cert}
2 (a_i x_i - a_j x_j) \leq (a_i+a_j) -(b_i-b_j) \,.
\ee
\end{lemma}

\begin{proof}
Lemma 3.8 of Onn~\cite{Onn} implies that $x$ is optimal if and only if 
$f(x) \leq f(x+g)$ for all $g \in G(A)$ where $G(A)$ is the Graver basis of
$A$. In our case, the Graver basis is given by the roots of the $A_{m-1}$ 
lattice, i.e., by
$$
G((1,1,\dots,1)) = \{ e_j-e_i \,\,| 1 \leq i,j \leq m, \,\, i \neq j \}.
$$
Let $g=e_j -e_i \in G(A)$ and $f(x)$ as in~\eqref{eq.fabx}. Then 
$f(x) \leq f(x+g)$ is equivalent to~\eqref{eq.cert}. 
\end{proof}

Below, we will call a vector quasi-linear if its coordinates are linear
quasi-polynomials. 

\begin{proposition}
\label{prop.QIP}
\rm{(a)}
Every non-degenerate lattice optimizer $x^*(t)$ of~\eqref{eq.QIPf} is 
quasi-linear of the form
\be
\label{eq.x*lattice}
x^*_i(t) = \frac{a_i^{-1}}{\sum_j a_j^{-1}} t + c_i(t)
\ee
for some $\varpi$-periodic functions $c_i$, where
\be
\label{eq.varpi}
\varpi=\sum_i \prod_{j \neq i} a_j \,.
\ee
\rm{(b)}
When $t \gg 0$ is an integer, the minimum value of~\eqref{eq.QIPf} is 
a quadratic quasi-polynomial
\begin{equation}
\label{eq.fx**}
\frac{1}{1 \cdot a^{-1}} t^2 
+ \frac{b \cdot a^{-1}}{1 \cdot a^{-1}} t + s_0(a,b)(t) 
\end{equation}
where $s_0(a,b)$ is a $\varpi$-periodic function of $t$. \\
\rm{(c)} For all $t>0$ the minimum value of~\eqref{eq.QIPf} is
\begin{equation}
\label{eq.fx***}
\frac{1}{1 \cdot a^{-1}} t^2 
+ \frac{b \cdot a^{-1}}{1 \cdot a^{-1}} t + O(1).
\end{equation}
\end{proposition}

Part (c) of Proposition \ref{prop.QIP} is what we will apply to the degree of the colored Jones polynomial.
Note that in general there are many minimizers of~\eqref{eq.QIPf}. 
Comparing with~\eqref{eq.x*} it follows that any lattice minimizer 
of~\eqref{eq.QIPf} is within $O(1)$ from the real minimizer.

\begin{proof}
Let $A_i=\prod_{j \neq i} a_j = a_1 \dots \hat{a}_i \dots a_m$, then
$\varpi=A_1+\dots + A_m$. 
Suppose $x^*$ satisfies the optimality criterion~\eqref{eq.cert} and
$A x^* = t$ where $A=(1,1,\dots,1)$. Let $x^{**}= x^* + (A_1,\dots,A_m)$.
Since $a_i A_i - a_j A_j =0$ for $i \neq j$, it follows that
$$
2 (a_i x^*_i - a_j x^*_j) = 2 (a_i x^{**}_i - a_j x^{**}_j) \,.
$$
Hence $x^*$ satisfies the optimality criterion~\eqref{eq.cert} if and only if
$x^{**}$ does. Moreover, $A x^{**} = A x^* + \varpi= t+\varpi$.
Since $a_i^{-1}/(\sum_j a_j^{-1})=A_i/\varpi$, it follows that
every minimizer $x^*(t)$ satisfies the property that 
$x^*_i(t) - \frac{a_i^{-1}}{\sum_j a_j^{-1}} t$ is a $\varpi$-periodic function
of $t$. Part (a) follows. For part (b), write $x^*(t) = 
\frac{t}{1 \cdot a^{-1}} a^{-1} + c(t)$ and use the fact that $A c(t) =0$
to deduce that $f(x^*(t))$ is a quadratic quasi-polynomial of $t$ with
constant quadratic and linear term given by \eqref{eq.sqq}.
For part (c), note that by (b) there is a constant $C>0$ such that we get \eqref{eq.fx***} for all $t>C$ by taking the maximum absolute value of the periodic $s_0(a, b)$.
For $0\leq t\leq C$ both the function $f$ and the quadratic are bounded by a constant so the conclusion still holds. 
\end{proof}

\subsection{Application: the degree of the colored Jones polynomial}
\label{sub.application}

Recall that our aim is to compute the maximum of the degree function $\delta(k)= \delta(n, k)$ of the states in the state sum of the colored Jones polynomial with tight parameters $k_0 = \sum_{i=1}^m k_i$, see Theorem \ref{thm.Delta}. Here $k=(k_0, k_1, \ldots, k_m)$ and $q = (q_0, q_1, \ldots, q_m)$ are $(m+1)$-vectors and we make use of the assumption that $q_i$ is odd for all $0\leq i \leq m$. We will compute the maximum in two steps.  
 
{\bf Step 1:} We will apply Proposition~\ref{prop.QIP} to the function 
$\delta(k)$ (divided by $-2$, and ignoring 
the terms that depend on $n$ and $q=(q_0, \ldots, q_m)$ but not on $k$):
\begin{equation}
\label{eq.dk0}
-\frac{1}{2} \delta(k) 
= \sum_{i=1}^m (q_i-1)k_i^2 + (q_0+1) \Bigl( \sum_{i=1}^m k_i \Bigr)^2
+ \sum_{i=1}^m k_i (-2+q_0+q_i) 
\end{equation}
under the usual assumptions that $q_0 <0$, $q_i >0$ for $i=1,\dots,m$.
We assume that $k=(k_1,\dots,k_m)\in \BZ^m$. Restricting $\delta(k)$ 
to the simplex $k_i \geq 0$ and $t = k_1+\dots + k_m$, 
we apply Proposition~\ref{prop.QIP} (c). It follows that for $t> 0$, 
$$
\min_{\substack{k_i \geq 0 \\ \sum_i k_i \leq n}} \delta(k) = 
\min_{0 \leq t \leq n} Q_0(t) + O(1)\,,
$$
where
\begin{equation}
\label{eq.dt}
\qquad Q_0(t) =
s(q) t^2 +  s_1(q) t\,, 
\end{equation}
and $s(q)$, $s_1(q)$ are given by~\eqref{eq.sqq}. 

{\bf Step 2:} Next it follows that 
$Q_0(t)$ is positive definite, degenerate, or negative definite if and 
only if $s(q)>0$, $s(q)=0$, or $s(q)<0$, respectively.  

{\bf Case 1:} $s(q)<0$. Then $Q_0(t)$ is negative definite and the minimum
is achieved at the boundary $t=n$. It follows that 
$$
\min_{\substack{k_i \geq 0 \\ \sum_i k_i \leq n}} \delta(k) 
= s(q) n^2 + s_1(q) n +O(1) \,.
$$

{\bf Case 2a:} $s(q)=0$, $s_1(q)\neq 0$. Then $Q_0(t)$ is a linear function of $t$ and
the minimum is achieved at $t=0$ or $t=n$ depending on whether $s_1(q) > 0$ or
$s_1(q)<0$, so we have
$$
\min_{\substack{k_i \geq 0 \\ \sum_i k_i \leq n}} \delta(k) 
= \begin{cases} O(1) & \text{if} \,\, s_1(q) > 0 \\
s_1(q)n+O(1) & \text{if} \,\, s_1(q) < 0 \,.
\end{cases} 
$$

{\bf Case 2b:} $s(q)=0=s_1(q)$. Now $t=0$ and $t=n$ both contribute equally so cancellation may occur. It does not because
the sign of the leading term is constant due to the parities of the $q_i$'s.

{\bf Case 3:} $s(q)>0$. Then $Q_0(t)$ is positive definite and 
Proposition~\ref{prop.QIP} implies that the lattice minimizers are near
$-s_1(q)/(2s(q))$ or at $0$, when $s_1(q)<0$ or $s_1(q) \geq 0$ and
the minimum value is given by:
$$
\min_{\substack{k_i \geq 0 \\ \sum_i k_i \leq n}} \delta(k) 
= \begin{cases} -\frac{s_1(q)^2}{4s(q)} +O(1) & \text{if} \,\, s_1(q) < 0 \\
O(1) & \text{if} \,\, s_1(q) \geq 0 \,.
\end{cases} 
$$
Note that cancellation of multiple lattice minimizers is ruled out because the signs of the leading terms are always the same due to the assumption on the parities of the $q_i$'s. Note also that the uncomputed $O(1)$ term above does not affect the proof of Theorem~\ref{thm.0}.

\begin{remark} It may be of interest to note that there are very few pretzel knots with $s(q)\geq 0$ and $s_1(q) = 0$.
These are cases 2b and 3 above where cancellations might occur if we had no control on the sign of the leading coefficients.
The case $P(-3,5,5)$ is mentioned in \cite{Manx} for its colored Jones polynomial with growing leading coefficient.
\end{remark} 

\begin{lemma}(Exceptional Pretzel knots)\\
\label{lem.exceptional}
The only pretzel knots with $q_0\leq -2<3\leq q_1,\ldots, q_m$ for which $s(q)\geq 0$ and $s_1(q) = 0$ are
\begin{enumerate}
\item $P(-3,5,5)$, $P(-3,4,7)$, $P(-2,3,5,5)$, with $s(q)=0$.
\item $P(-2,3,7)$, with $s(q)=\frac{1}{2}$.
\end{enumerate}
\end{lemma}
\begin{proof}
Changing variables to $f_i = q_i-1$ turns the two equations $s(q)\geq 0$ and $s_1(q) = 0$ into:
$f_0(f_1^{-1}+\dots+f^{-1}_m)+m = 0$ and
$2+f_0+\frac{1}{f_1^{-1}+\dots+f^{-1}_m} = c$ for some $c\geq 0$.
Solving for $f_0$ yields $f_0 = (c-2)\frac{m}{m-1}$. Since $f_0\leq -3$ we must have
$0\leq c\leq 2-3\frac{m-1}{m}$. This means there can only be such $c$ when $m=2$ or $3$.
Suppose $m=2$ then $c=0$ or $c=\frac{1}{2}$. In the first case we find $f_2=\frac{2f_1}{f_1-2}$ 
so the positive integer solutions are $(f_1,f_2) \in \{(3,6),(4,4),(6,3)\}$. In the case $c=\frac{1}{2}$ we find
$f_2 = \frac{3f_1}{2f_1-3}$ so $(f_1,f_2) \in \{(2,6),(3,3),(6,2)\}$.
Finally the case $m=3,c=0,f_0=-3$ yields $(f_1,f_2,f_3)\in \{(2,4,4),(2,3,6),(3,3,3)\}$ and permutations.
\end{proof}

\section{The colored Jones polynomial of Montesinos knots}
\label{sec.cjmontsinos}

In this section we will extend Theorem \ref{thm.Delta} to the class of Montesinos knots. For a Montesinos knot $K = K(r_0, r_1, \ldots, r_m)$, we will always consider the \emph{standard diagram}, also denoted by $K$,  coming from the unique continued fraction expansion of even length of each rational number as in the case of pretzel knots. Recall that our restriction to Montesinos links with precisely one negative tangle and the existence of reduced diagrams for Montesinos links means that we can assume $r_i[0] = 0$ for all $0\leq i \leq m$, see Section \ref{ss.Mclassify}. 
To build the diagram from simpler diagrams we introduce the 
tangle replacement move (in short, \TR-move), and study its effect
on the state-sum formula for the colored Jones polynomial.

\subsection{The \TR-move}

A \TR-move is a local modification of a link diagram $D$. Suppose $D$ 
contains a twist region $T$. Viewing $T$ as a rational tangle $T= \frac{1}{t}$ 
for some integer $t$ we may consider a new diagram $D_1$ obtained by 
replacing $T$ by the rational tangle $T_1 = r*\frac{1}{t}$ for some non-zero 
integer $r$ with the same sign as $t$. Alternatively, viewing $T$ as an 
integer tangle $t$ we 
replace it with $T_2 = \frac{1}{r} \oplus t$, also with $r$ having the same sign. 
Collectively these two operations are 
referred to as the \emph{\TR-moves}. Recall from Section \ref{sec.RT}, Equations \eqref{e.algt1}, \eqref{e.algt2} that we can construct a diagram of any rational tangle by a combination of \TR-moves, see also Figures \ref{f.R-move} and \ref{f.TRiterate}. We extend this to $n$-cabled tangle diagrams by labeling each arc in the diagram by $n$. 

\begin{figure}[H]
\def \svgwidth{.8\columnwidth}
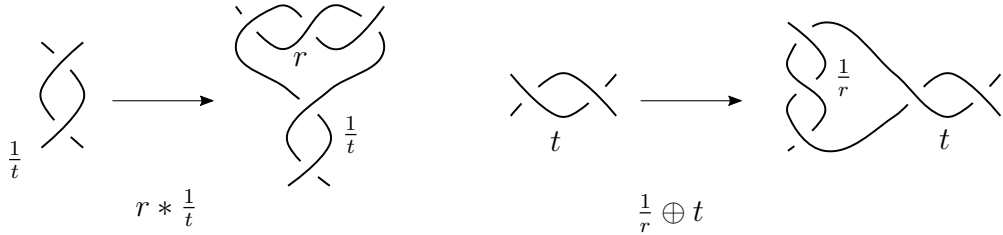
\caption{Two types of \TR-moves.}\label{f.R-move}
\end{figure}

\begin{figure}[H]
\def \svgwidth{.3\columnwidth}
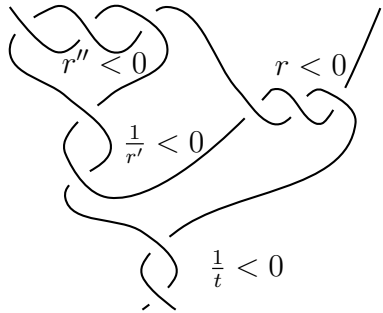
\caption{\label{f.TRiterate}Any rational tangle is produced by a combination of \TR-moves. In the picture shown, we have performed three \TR-moves: first on 
$1/t$, then $r$, then $1/r'$.  }
\end{figure}

We will use the \TR-moves to reduce a Montesinos knot to the associated pretzel knot by first reducing it to a special Montesinos knot. 


\subsection{Special Montesinos knot case }

We start by considering the case of Montesinos links $K(r_0, \ldots, r_m)$ where $\ell_{r_i} =2$ for all $0\leq i \leq m$. This includes the pretzel links by choosing the unique even length continued fraction expansion with $r_0[2] = -1$ and $r_i[2]=1$ for each $1 \leq i \leq m$. We call these links \emph{special Montesinos links}.
We prove the main theorem,  Theorem \ref{thm.2Delta},  for special Montesinos links. 
\begin{figure}[H]
\centering
\def \svgwidth{.7\columnwidth}
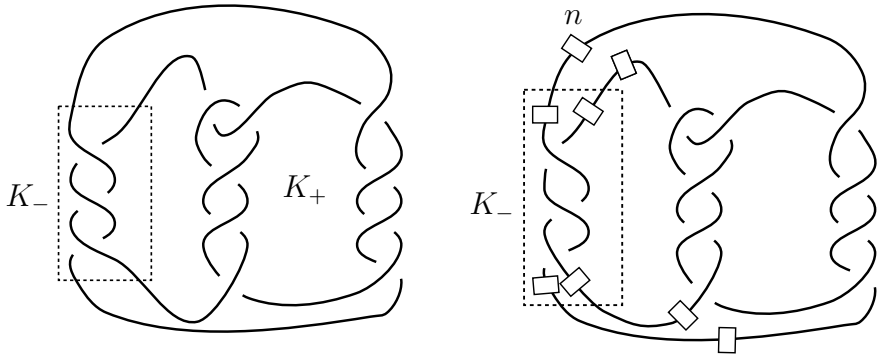
\caption{\label{f.sMontesinos} A special Montesinos knot $K = K(-\frac{1}{3}=\frac{1}{-2+\frac{1}{-1}}, \frac{2}{7} = \frac{1}{3+\frac{1}{2}}, \frac{1}{4} = \frac{1}{3+\frac{1}{1}} ) = N(K_-\oplus K_+)$, and $K^n = N(K_-^n \oplus K_+^n)$.}
\end{figure} 

As in the case of pretzel links we use a customized state sum to compute the colored Jones polynomial, splitting $K = N(K_-\oplus  K_+)$.   In this case $K_-$ is the single vertical 2-tangle $1/(r_0[1]-1)$ and $K_+$ is the 2-tangle that is the rest of the diagram. As before we apply the fusion and untwisting formulas  \eqref{eq.fusion}, \eqref{eq.untwisting}  to $K^n_-$ and the Kauffman state sum to $K^n_+$ after cabling with the $n$th Jones-Wenzl idempotent for the $n$th colored Jones polynomial. See Figure \ref{f.sMontesinos}. 

The methods used previously on the pretzel links also apply to this case with minor modifications.
In particular the notion of global through strands $c(\sigma)$ for a Kauffman state $\sigma$ on $K^n_+$ still makes sense and $k_i(\sigma)$ is still well-defined by restricting $\sigma$ to the $i$th tangle.
In this case $c_i(\sigma)$ means the number of through strands of the $i$th tangle of $K^n_+$ that are also global through strands, and as before $k_i = \lceil \frac{c_i}{2} \rceil$. Let $$\mathcal{G}_{c, k} = \sum_{k_0}\sum_{\sigma: k_i(\sigma)=k_i, c(\sigma) = c}G_{k_0}(v) v^{\sgn(\sigma)} \langle N(I_{k_0}\oplus (K^n_+)_\sigma) \rangle,$$ where as in the case of pretzel links, $I_{k_0}$ is the skein element in $\TL_{2n}$ in the sum obtained by applying the fusion and untwisting formulas to $K^n_-$, and $G_{k_0}(v)$ are the coefficients in rational functions of $v$. 

We have 
\be
\langle K^n \rangle = \sum_{(k_0,\sigma)}G_{k_0}(v)v^{\sgn(\sigma)}\langle N(I_{k_0} \oplus (K^n_+)_\sigma) \rangle = \sum_{c, k} \mathcal{G}_{c, k}. \notag
\ee
We prove the following theorem. 

\begin{theorem}
\label{thm.2Delta}
Consider $K = K(\frac{1}{r_0[1]+\frac{1}{-1}},\frac{1}{r_1[1]+\frac{1}{r_1[2]}},\dots,\frac{1}{r_m[1]+\frac{1}{r_m[2]}})$.
Assume $|r_i[1]|>1$, $|r_i[2]|>0$, $r_0[1] < 0 < r_i[1]$, and define $q_0 = r_0[1]-1$, $q_i = r_i[1]+1$ for $1\leq i \leq m$, $q_i' = r_i[2]$ for $1\leq i \leq m$. Referring to the state sum $\langle K^n\rangle= \sum_{c, k} \mathcal{G}_{c, k}$ we have the following: For a Kauffman state $\sigma$, define the parameters $c = c(\sigma), k=k(\sigma)$ to be tight if $k_0= k_1+\dots + k_m= \frac{c}{2}$.
For tight $c, k$ we have
$\mathcal{G}_{c, k} = (-1)^{q_0(n-k_0)+n+k_0
+\sum_{i=1}^{m}(n-k_i)(q_i-1)} v^{\delta(n, k)}+l.o.t.$\footnote{The abbreviation $l.o.t.$ means lower order terms in $v$.} and  $\frac{-\delta(n, k)}{2} =$
\be
\label{eq.2Deltank} (q_0+1)k_0^2 + \sum_{i=1}^m (q_i-1)k_i^2 + \sum_{i=1}^m(-2+q_0+q_i)k_i- \frac{n(n+2)}{2}\sum_{i=0}^m q_i + (m-1)n - \frac{n^2}{2} \sum_{i=1}^m (q'_i-1).
\ee
If $c,k$ are not tight then there exists a tight pair $c',k'$ (coming from some Kauffman state) such that $\deg_v \mathcal{G}_{c,  k} < \deg_v \mathcal{G}_{c',  k'}$.
\end{theorem}

\begin{proof}
The proof is analogous to that of Theorem \ref{thm.Delta} for pretzel links. As in the pretzel case we identify the minimal states and show that they maximize the degree and do not cancel out. Since these arguments are exactly the same we focus on describing the minimal states, one for each set of tight parameters of through strands $c, k$. 
The minimal states are produced by choosing a minimal state $\tau_P$ for the pretzel link $P = P(q_0,\dots, q_m)$ and extending it to a minimal Kauffman state $\tau$ of $\langle K^n_+\rangle$. The set $A_{\tau}$ on which $\tau$ chooses the $A$-resolution is a union of the restriction of $A_{\tau_P}$  on the twist region $(1/r_i[1])^n$ and a set of $k_i^2$ crossings in pyramidal position with $k_i$ through strands in $(r_i[2])^n$, $1\leq i \leq m$. This set exists whenever $|r_i[2]| > 0$.  We have 
\[d(k, \tau) = d(k, \tau_P) + n^2\sum_{i=1}^m (q_i'-1), \]
to account for the additional crossings from $r_i[2]$ for $1\leq i \leq m$ on which $\tau$ chooses the $B$-resolution. This gives $\delta(n, k)$ in the theorem. 
\begin{figure}[H]
\def \svgwidth{1.1\columnwidth}
\begingroup%
  \makeatletter%
  \providecommand\color[2][]{%
    \errmessage{(Inkscape) Color is used for the text in Inkscape, but the package 'color.sty' is not loaded}%
    \renewcommand\color[2][]{}%
  }%
  \providecommand\transparent[1]{%
    \errmessage{(Inkscape) Transparency is used (non-zero) for the text in Inkscape, but the package 'transparent.sty' is not loaded}%
    \renewcommand\transparent[1]{}%
  }%
  \providecommand\rotatebox[2]{#2}%
  \newcommand*\fsize{\dimexpr\f@size pt\relax}%
  \newcommand*\lineheight[1]{\fontsize{\fsize}{#1\fsize}\selectfont}%
  \ifx\svgwidth\undefined%
    \setlength{\unitlength}{942.9460144bp}%
    \ifx\svgscale\undefined%
      \relax%
    \else%
      \setlength{\unitlength}{\unitlength * \real{\svgscale}}%
    \fi%
  \else%
    \setlength{\unitlength}{\svgwidth}%
  \fi%
  \global\let\svgwidth\undefined%
  \global\let\svgscale\undefined%
  \makeatother%
  \begin{picture}(1,0.4037273)%
    \lineheight{1}%
    \setlength\tabcolsep{0pt}%
    \put(0,0){\includegraphics[width=\unitlength,page=1]{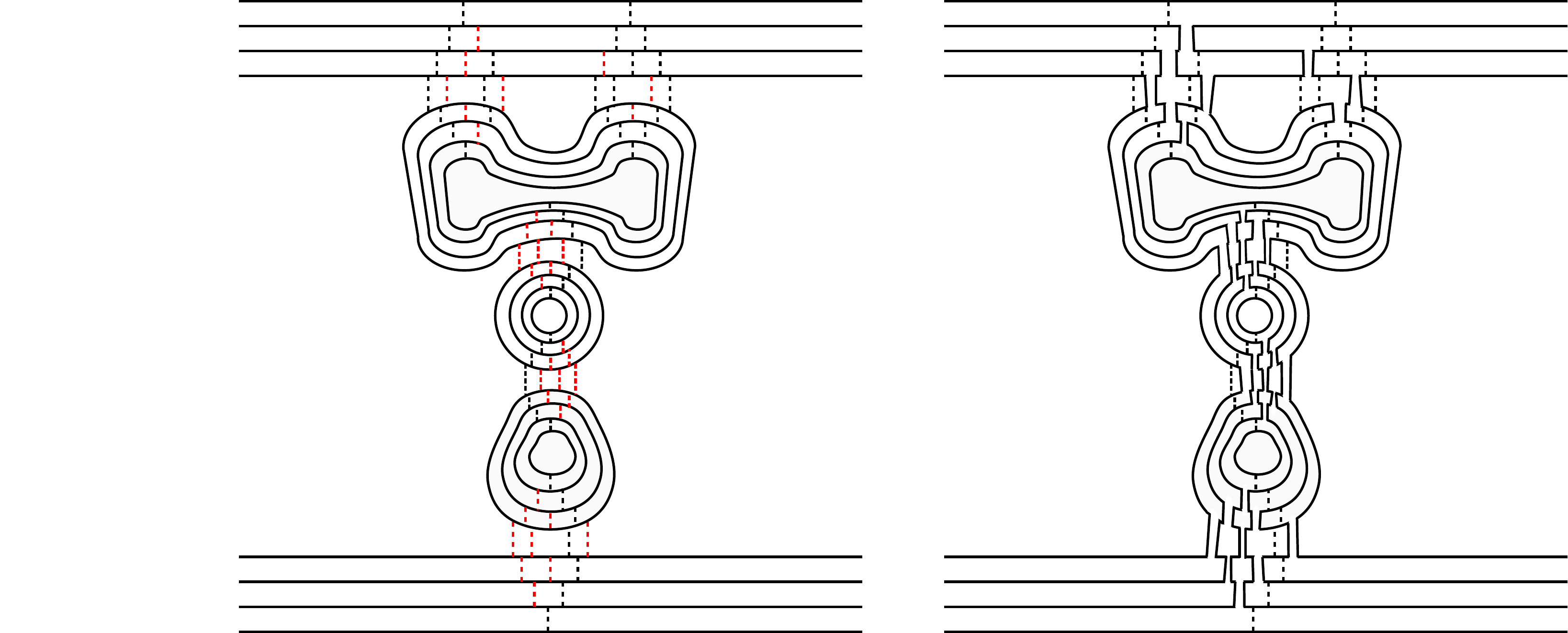}}%
    \put(0.04318163,0.25648787){\color[rgb]{0,0,0}\makebox(0,0)[lt]{\lineheight{1.25}\smash{\begin{tabular}[t]{l}$n$\end{tabular}}}}%
    \put(0.06804743,0.17327325){\color[rgb]{0,0,0}\makebox(0,0)[lt]{\lineheight{1.25}\smash{\begin{tabular}[t]{l}$n$\end{tabular}}}}%
    \put(0,0){\includegraphics[width=\unitlength,page=2]{Montesinosflowex.pdf}}%
    \put(0.48571594,0.10658766){\color[rgb]{0,0,0}\makebox(0,0)[lt]{\lineheight{1.25}\smash{\begin{tabular}[t]{l}Restriction from $\tau_P$\end{tabular}}}}%
    \put(0,0){\includegraphics[width=\unitlength,page=3]{Montesinosflowex.pdf}}%
    \put(0.02968263,0.31908838){\color[rgb]{0,0,0}\makebox(0,0)[lt]{\lineheight{1.25}\smash{\begin{tabular}[t]{l}Extension to $\tau$\end{tabular}}}}%
    \put(0,0){\includegraphics[width=\unitlength,page=4]{Montesinosflowex.pdf}}%
  \end{picture}%
\endgroup%

\caption{\label{f.Montesinosflowex} An example with $n=4$ showing a minimal state with 3 through strands through a rational tangle $r$ where $\ell_r =2$ and $r[2]>1$. One can see the extension of the minimal state on the vertical twist region $1/r[1]$. We choose $3^2 = 9$ crossings in pyramidal position in the  twist region $(r[2])^n$.  }
\end{figure} 
\end{proof}

\subsection{The general case} \label{ss.Mdegree}

Given $K = K(r_0, r_1, \ldots, r_m)$, we decompose the standard diagram $K =  N(K_-\oplus K_+)$, where $K_-$ is the single negative tangle, and $K_+$ is the rest of the diagram. We further decompose $K_- = D_- \cup V_-$ where $D_-$ consists of the negative twist region $1/r_0[1]$ if $r_0[2] \not=-1$, or $1/(r_0[1]-1)$  if $r_0[2] = -1$ and $\ell_{r_0} = 2$, while $V_-$ is the rest of $K_-$. Note $D_-$ and $V_-$ are joined as they are in the diagram $K$. For the 2-tangle diagram corresponding to $r_i$ in $K$, where $i\leq 1\leq m$, let the tangle diagram $T_i$ be the portion corresponding to the first two (with respect to the continued fraction expansion) twist regions $1/r_i[1]$ and $r_i[2]$. If $\ell_{r_i} = 2$, then $T_i = r_i[2]*\frac{1}{r_i[1]}$. Otherwise if $\ell_{r_i}>2$, then $T_i$ is a $(4, 2)$-tangle diagram obtained by joining the upper right strand of $\frac{1}{r_i[1]}$ to the lower right strand of $r_i[2]$. We decompose $K_+$ as $K_+ = D_+ \cup V_+$, where $D_+ = \cup_{i=1}^m T_i$ is the portion of the standard diagram $K$ obtained by arranging $T_i$ side by side in a row in order, and joining each pair $T_i, T_{i+1}$ for $1\leq i \leq m-1$ according to the rules as follows:
\begin{itemize} 
 \item If $\ell_{r_i} = \ell_{r_{i+1}} = 2$, then $T_i$ and $T_{i+1}$ are both $2$-tangles. The lower right strand of $T_i$ is joined to lower left strand of $T_{i+1}$, and the upper right strand of $T_i$ is joined to the upper left strand of $T_{i+1}$. 
 \item If $\ell_{r_i} = 2$ and $\ell_{r_{i+1}} >2$, or $\ell_{r_i} > 2$ and $\ell_{r_{i+1}} >2$  , only the lower right strand of $T_i$ is joined to the lower left strand of $T_{i+1}$. 
 \item If  $\ell_{r_i} > 2$ and $\ell_{r_{i+1}} =2$, then the upper right strand of $T_i$ is joined to the upper left strand of $T_{i+1}$, and the lower right strand of $T_i$ is joined to the lower left strand of $T_{i+1}$. 
 \end{itemize} 
 Define $V_+$ to be the rest of $K_+$.  See Figure \ref{f.T} for examples of $T_i$'s and Figure \ref{f.sMdecomp} for an illustration of the decomposition of a Montesinos knot $K$. The union defined extends to the $n$-cable of the tangle diagrams by decorating each strand with $n$, so $(D_+ \cup V_+)^n = D_+^n \cup V_+^n$ and $(D_- \cup V_-)^n = D_-^n \cup V_-^n$.  
\begin{figure}[H]
\def \svgwidth{.9\columnwidth}
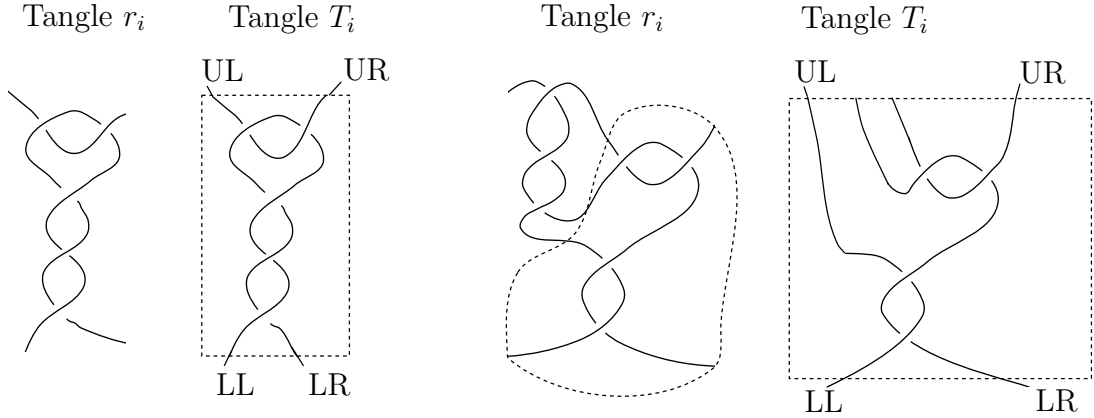
\caption{\label{f.T} Two cases of the $T_i$'s corresponding to the $r_i$'s are shown. $UL$ stands for upper left; $UR$ stands for upper right; $LL$ stands for lower left; $LR$ stands for lower right. }
\end{figure} 
\begin{figure}[H]
\def \svgwidth{\columnwidth}
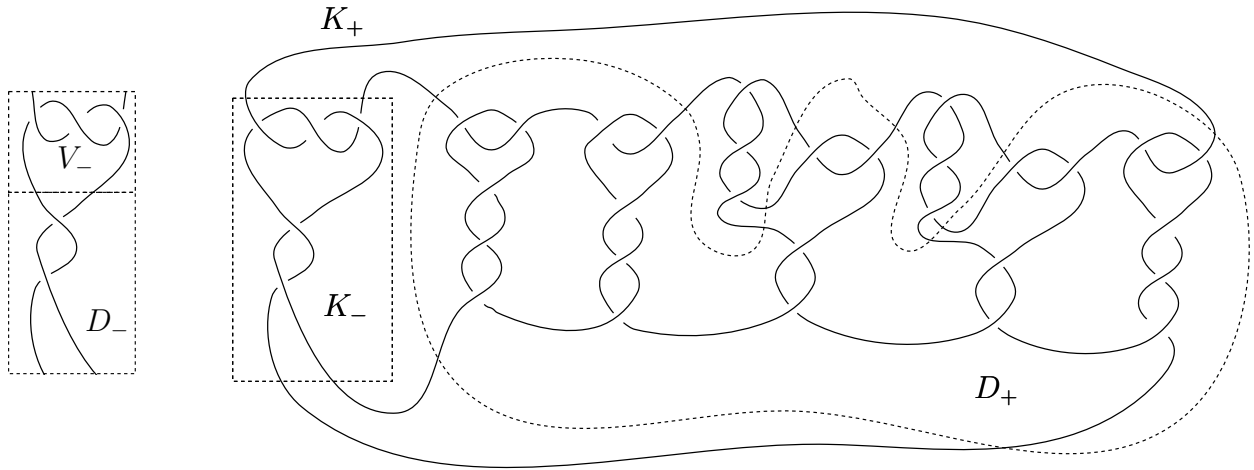
\caption{\label{f.sMdecomp} We show the decomposition $K = N(K_-\oplus K_+)$  of a Montesinos knot $K=K(-\frac{3}{7}=\frac{1}{-2+\frac{1}{-3}}, \frac{2}{7}=\frac{1}{3+\frac{1}{2}}, \frac{2}{7} = \frac{1}{3+\frac{1}{2}},  \frac{7}{17} = \frac{1}{2+\frac{1}{2+\frac{1}{2+\frac{1}{1}}}}, \frac{7}{17} = \frac{1}{2+\frac{1}{2+\frac{1}{2+\frac{1}{1}}}}, \frac{2}{7}=\frac{1}{3+\frac{1}{2}})$. In the figure, $K_-=D_-\cup V_-$ (decomposed on the left) is the $2$-tangle enclosed in the dashed rectangular box on $K$. On the right, $D_+$ is the tangle enclosed in the dashed curve on $K_+$,  and $V_+$ is the rest of the diagram $K_+$.    }
\end{figure} 

Let 
 \[q_0 = \begin{cases} &r_0[1] +\frac{1}{-1} \text{ if $\ell_{r_0} = 2$ and $r_0[2]=-1$.  } \\ 
 	& r_0[1] \text{ otherwise }  \end{cases}.   \] 
The link $L= K(\frac{1}{q_0}, \frac{1}{r_1[1]+\frac{1}{r_1[2]}}, \ldots, \frac{1}{r_m[1]+\frac{1}{r_m[2]}})$ is a special Montesinos link. We approach the general case as insertion of the union of rational tangles $V = V_- \cup V_+$ into this special Montesinos link via $\mathsf{TR}$-moves. The essential feature of $V$ is that its all-$B$ state acts like the identity on $\langle L \rangle$ plus some closed loops, see Figure \ref{f.TR2}.

\begin{lemma} \label{l.addadequate} 
Suppose we have the standard diagram of a Montesinos knot $K= K(r_0, r_1, \ldots, r_m) =N(K_- \oplus K_+) = N((D_-\cup V_-)\oplus(D_+\cup V_+))$, where $L = K(\frac{1}{q_0}, \frac{1}{r_1[1]+\frac{1}{r_1[2]}}, \ldots, \frac{1}{r_m[1]+\frac{1}{r_m[2]}})$ is a special Montesinos knot. Let $V = V_-\cup V_+$, joined as they are in the diagram $K$.
If $q_0<-1$  is odd, and $q_i= r_i[1]+1>1$ is odd for every $i>0$, then
we have
\[\deg_v \langle K^n \rangle = \deg_v  \langle L^n \rangle + c(V)n^2+2n\ o(V_B), \]
where $c(V)$ is the number of crossings in $V$ and  $o(V_B)$ is the number of disjoint circles resulting from applying the all-$B$ state to $V$.
\end{lemma}
\begin{proof}
Decompose the $n$-cable of the standard diagram $L^n = N(L_-^n\oplus L_+^n)$ as in Figure \ref{f.sMontesinos}.  Applying quadratic integer programming to the formula of Theorem \ref{thm.2Delta} for the degree-maximizing states of $\langle L^n \rangle$, discarding any terms that depend only on $q_i$, $q_i'$ and $n$ and not on $k_i$, we see that there are minimal states of the state sum of any special Montesinos knot that attain the maximal degree. Fix one such minimal state $\tau$. 

We decompose $K^n  = N(K_-^n \oplus K_+^n) = N((D_-^n\cup V_-^n)\oplus(D_+^n\cup V_+^n)) $ and write down a state sum for $\langle K^n \rangle$ by applying the fusion and untwisting formulas to the $n$-cable of the single negative twist region in $D_-^n$ and applying Kauffman states on the set of crossings in the rest of the diagram $K^n$.  Let $\sigma \cup \sigma'$ denote a Kauffman state on $K^n$ where $\sigma$ is a Kauffman state on $D^n_+ $ and $\sigma'$ is a Kauffman state on $V^n = (V^n_-\cup V_+^n)$. We have

\[ \langle K^n  \rangle = \sum_{(k_0, \sigma \cup \sigma')} G_{k_0}(v) v^{\sgn(\sigma \cup \sigma')} \langle N((I_{k_0} \cup (V_-^n)_{\sigma'})\oplus ((D_+^n)_\sigma\cup (V_+^n)_{\sigma'})) \rangle.  \]  

Because $V = V_-\cup V_+$ is a (possibly disjoint) union of alternating tangles, applying the all-$B$ state on $V^n$ results in a set $O(V^n_B)$ of $n \ o(V_B)$ disjoint circles, and

\[N((I_{k_0} \cup (V_-^n)_{B})\oplus ((D_+^n)_\sigma\cup (V_+^n)_{B})) = N(I_{k_0} \oplus (L_+^n)_\sigma)   \sqcup O(V_B^n)  \]
by visual inspection of the diagrams involved. Letting $B_V$ denote the all-$B$ state on $V^n$, we get
\begin{align*}
 \langle K^n  \rangle &= \sum_{(k_0, \sigma \cup \sigma'), \sigma' \not=B_V} G_{k_0}(v) v^{\sgn(\sigma \cup \sigma')} \langle N((I_{k_0} \cup (V_-^n)_{\sigma'})\oplus ((D_+^n)_\sigma\cup (V_+^n)_{\sigma'})) \rangle \\ 
 &+ \sum_{(k_0, \sigma \cup B_V)} G_{k_0}(v) v^{\sgn(\sigma \cup B_V)} \langle N((I_{k_0} \cup (V_-^n)_{B_V})\oplus ((D_+^n)_\sigma\cup (V_+^n)_{B_V})) \rangle \\
 & = \sum_{(k_0, \sigma \cup \sigma'), \sigma' \not=B_V} G_{k_0}(v) v^{\sgn(\sigma \cup \sigma')} \langle N((I_{k_0} \cup (V_-^n)_{\sigma'})\oplus ((D_+^n)_\sigma\cup (V_+^n)_{\sigma'})) \rangle \\ 
 &+\sum_{(k_0, \sigma \cup B_V)} G_{k_0}(v) v^{\sgn(\sigma \cup B_V)} \langle N(I_{k_0} \oplus (L_+^n)_\sigma)   \sqcup O(V_B^n) \rangle \\ 
 &=\sum_{(k_0, \sigma \cup \sigma'), \sigma' \not=B_V} G_{k_0}(v) v^{\sgn(\sigma \cup \sigma')} \langle N((I_{k_0} \cup (V_-^n)_{\sigma'})\oplus ((D_+^n)_\sigma\cup (V_+^n)_{\sigma'})) \rangle \\ 
 &+\langle L^n \sqcup O(V_B^n)\rangle. \\ 
\end{align*}  
Let 
\[ d(k_0, \sigma \cup \sigma') = \deg_v \left( G_{k_0}(v) v^{\sgn(\sigma \cup \sigma')} \langle N((I_{k_0} \cup (V_-^n)_{\sigma'})\oplus ((D_+^n)_\sigma\cup (V_+^n)_{\sigma'})) \rangle  \right).  \]
The diagram  $V$ being a union of alternating tangles also implies that a state on $V^n$ that is not the all-$B$ state merges a circle from $O(V_B^n)$. Therefore, by an application of Lemma \ref{l.split},
\[ d(k_0, \sigma \cup B_V) > d(k_0, \sigma \cup \sigma'), \]
for any $\sigma' \not= B_V$.

  Thus for a pair $(k_0, \tau)$ where $d(k_0, \tau)$ maximizes the degree in the state sum of $\langle L^n\rangle$, the term 
  \[ G_{k_0}(v) v^{\sgn(\tau \cup B_V)} N(I_{k_0} \oplus (L_+^n)_\tau)   \sqcup  \ O(V_B^n)  \] 
   also maximizes the degree in the state sum for $\langle K^n \rangle$. The leading terms all have the same sign because of the assumption on the parities of the $q_i$'s and Theorem \ref{thm.2Delta}. Thus there is no cancellation of these maximal terms, and we can determine $\deg_v\langle K^n \rangle$ relative to  $\deg_v\langle L^n \rangle$  by counting the number of disjoint circles in $O(V_B)$, giving the formula in the lemma. 
\end{proof}

It is useful to reformulate Lemma \ref{l.addadequate} in a more relative sense, pinpointing how the degree changes as a result of applying a $\mathsf{TR}$-move. Let $\mathsf{TR}_1^-$ denote the $\mathsf{TR}$-move that sends $\frac{1}{t}$ to $r*\frac{1}{t}$. We define two  composite moves $\mathsf{TR}^-_2(T) = (\frac{1}{r_1}\oplus r_2)*T$, and
$\mathsf{TR}^+(T) = (r_1*\frac{1}{r_2})\oplus T$.

\begin{figure}[H]
\def \svgwidth{.9\columnwidth}
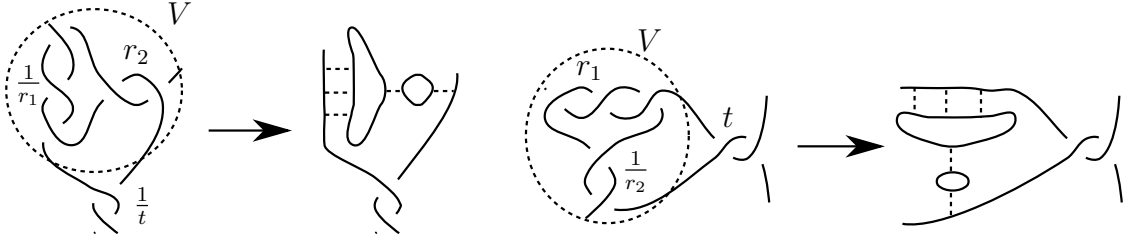 
\caption{\label{f.TR2} Examples of applying the all-$B$ state to $V$ and the 
resulting disjoint circles for moves sending tangles $\frac{1}{t}$ to 
$(\frac{1}{r_1}\oplus r_2)*\frac{1}{t}$ and sending $t$ to $(r_1*\frac{1}{r_2})\oplus t$. The tangles that form part of $V$ are encircled.}
\end{figure}

\begin{lemma} 
\label{l.rmove} 
Suppose two standard diagrams $K,L$ of Montesinos knots satisfy the conditions of Lemma \ref{l.addadequate}, where $K$ is obtained from $L$ by applying one of the moves $\mathsf{TR}^-_{1},\mathsf{TR}^-_{2},\mathsf{TR}^+$,
locally replacing tangle $(T)^n$ by $(T')^n$, then the degree of the colored Jones polynomial changes as follows. See Figure \ref{f.TR2} for examples of the moves $\mathsf{TR}^-_2$, $\mathsf{TR}^+$.

\begin{itemize}
\item[$\mathsf{TR}^-_1$-move:] 
Suppose $r, t <0$,  and  $T = \frac{1}{t}$ is a vertical 
twist region, and $T' = r*\frac{1}{t}$, then 
\[ 
\deg \langle K^{n} \rangle 
= \deg \langle L^{n}\rangle-rn^2 + 2(-r-1)n.  
\] 
\item[$\mathsf{TR}^-_2$-move:] 
Suppose $r_1, r_2, t <0$, $T = \frac{1}{t}$ is a 
vertical twist region,  and $T' = (\frac{1}{r_1}\oplus r_2)*\frac{1}{t}$, then
\[ 
\deg \langle K^{n} \rangle  
= \deg \langle L^{n} \rangle-(r_1+r_2)n^2 - 2r_2n.
\] 
\item[$\mathsf{TR}^+$-move:] 
Suppose $r_1, r_2, t >0$, $T = t$ is a horizontal twist 
region, and $T'=(r_1*\frac{1}{r_2})\oplus t$, then 
\[ 
\deg \langle K^{n} \rangle  
= \deg \langle L^{n}\rangle+(r_1+r_2)n^2 + 2r_2n.
\] 
\end{itemize}
\end{lemma} 

\begin{proof}
Applying Lemma \ref{l.addadequate} we count the number of crossings and the number of state circles from applying the all-$B$ state to the newly added tangle $V$ in each of these cases, and determine the resulting degree. 
\end{proof}

We use Lemma \ref{l.rmove} to prove the part of Theorem \ref{thm.1} concerning the degree of the colored Jones polynomial for the Montesinos knots that we consider. 
\begin{theorem} \label{t.Montesinoscjpd}
With the same definitions for $q_i$'s for $0\leq i \leq m$ as in Lemma \ref{l.addadequate}, 
let $K=K(r_0, r_1, \ldots, r_{m})$ be a Montesinos knot such that 
$r_0 <0$, $r_i >0$ for all $1 \leq i\leq m$, and $|r_i| < 1$ for all $0\leq i \leq m$ with $m\geq 2$ even.
Suppose $q_0 <-1<1<q_1,\ldots, q_m$ are all odd, and $q'_0$ is an integer that is defined to be 0 if $r_0 = 1/q_0$, and defined to be $r_0[2]$ otherwise. Let $P=P(q_0, \dots, q_m)$ be the associated pretzel knot, and let $\omega(D_K)$, $\omega(D_P)$ denote the writhe of standard diagrams $D_K$, $D_P$ with orientations. For all $n>N_K$ we have:
\begin{align*}
\label{eq.jsxM}
\js_K(n) &= \js_P(n) -q'_0 -  [r_0]-\omega(D_P) + \omega(D_K)+ \sum_{i=1}^m (r_i[2]-1) +  \sum_{i=1}^m [r_i] , \\
\jx_K(n) &= \jx_P(n) -2\frac{q'_0}{r_0[2]}
+2 [r_0]_o -2 \sum_{i=1}^m (r_i[2]-1) -2 \sum_{i=1}^m [r_i]_e. 
\end{align*}
\end{theorem} 
\begin{proof}
Suppose $K = K(r_0, r_1, \ldots, r_m)= N(K_- \oplus K_+)$ is a Montesinos knot, then $K$ is obtained from a special Montesinos knot $L= K(\frac{1}{q_0}, \frac{1}{r_1[1]+\frac{1}{r_1[2]}}, \ldots, \frac{1}{r_m[1]+\frac{1}{r_m[2]}}) = N(L_- \oplus L_+)$ by a combination of $\mathsf{TR}$-moves on the tangles in $L$ following the unique even length positive continued fraction expansions of $r_i$ for $0 \leq i \leq m$. Recall each rational tangle diagram corresponding to $r_i$ has an algebraic expression of the form 
\[ (((r_i[\ell_{r_i}] * \frac{1}{r_i[\ell_{r_i}-1]})\oplus r_i[\ell_{r_i}-2])* \cdots * \frac{1}{r_i[1]}).  \] 
The diagram $K_+$  is obtained by applying successive $\mathsf{TR}^+$-moves to $r_i[j]$ in $L_+$, $1\leq i \leq m$, sending $r_i[j]$ to $(r_i[j+2]* 1/r_i[j+1]) \oplus r_i[j]$ for each even $2 \leq j \leq \ell_{r_i}$ starting with $j=2$. 
Similarly, the rational tangle $K_-$ is obtained from $L_-$ by applying the $\mathsf{TR}^-_2$-moves to $\frac{1}{r_0[j]}$, sending $\frac{1}{r_0[j]}$ to $(\frac{1}{r_0[j+2]}\oplus r_0[j+1])*\frac{1}{r_0[j]}$, for each odd $1\leq j < \ell_{r_0}$ starting with $j=1$, with a final $\mathsf{TR}^-_1$-move sending $\frac{1}{r_0[\ell_{r_0-1}]}$ to $(r_0[\ell_{r_0}]*\frac{1}{r_0[\ell_{r_0-1}]})$.  

Recall $$
[ r ]_{e} = \sum_{3\leq j\leq \ell_r, \ j=\text{even}} r[j], \quad
[r ]_{o} = \sum_{3\leq j\leq \ell_r, \ j=\text{odd}} r[j], \quad
 [ r] =  [ r ]_{e} + [ r ]_{o} \,.
$$ We have two cases for the degree of $\langle K^n\rangle$ relative to $\langle L^n \rangle$, where $K^n$ is obtained from applying the combination of $\mathsf{TR}$-moves to $L^n$ as described above: 
\begin{itemize}
\item[(1)] $r_0 = 1/q_0$. 
By Lemma \ref{l.rmove}, each application of the $\mathsf{TR}^+$-move adds $(r_i[j+2] + r_i[j+1])n^2 + 2r_i[j+1]n$ to the degree  for each even $2\leq j \leq \ell_{r_i}$ where $1\leq i \leq m$.  
We have 
\begin{align*}
& \deg_v \langle K^n \rangle  \\ 
&= \deg_v \langle L^n \rangle +  \sum_{i=1}^m \sum_{j \text{ even}, \ 2\leq j \leq \ell_{r_i}} (r_i[j+2] + r_i[j+1])n^2 + 2r_i[j+1]n \\
&= \deg_v \langle L^n \rangle + n^2 \sum_{i=1}^m [r_i]+  2n \sum_{i=1}^m [r_i]_o.    
\intertext{Applying quadratic integer programming to \eqref{eq.2Deltank} for $\deg_v \langle L^n \rangle$ and ignoring the part of the degree function that only depends on $n, q_i$, and $q_i'$'s, we see that as long as the $q_i$'s for  $0\leq i \leq m$ satisfy the hypotheses of the theorem, }
&\deg_v \langle K^n \rangle  \\
&= \underbrace{-2s(q)(n)n^2 -2s_1(q)(n)n + \text{lower order terms}}_{\deg_v \langle L^n \rangle}   + n^2 \sum_{i=1}^m (q'_i-1) + n^2\sum_{i=1}^m [r_i] + 2n\sum_{i=1}^m [r_i ]_o.
\intertext{Gathering the coefficients multiplying $n^2$ and accounting for the writhes of standard diagrams $D_K$, $D_P$ , we get}
&\js_K(n) = \js_P(n) - \omega(D_p) + \omega(D_K)  + \sum_{i=1}^m (q'_i-1) + \sum_{i=1}^m [r_i].
\intertext{Note that $q'_i = r_i[2]$, and $q'_0 = [r_0] = 0$ for this case, and so trivially}
&\js_K(n)  =  \js_P(n)-q'_0 - [r_0]- \omega(D_p) + \omega(D_K) + \sum_{i=1}^m (r_i[2]-1) + \sum_{i=1}^m [r_i].  
\intertext{Now we compute $\jx_K(n)$ by considering $\deg_v \langle K^{n-1} \rangle$ and collecting coefficients of $n$. This gives}
&\jx_K(n) = \jx_P(n) - 2\sum_{i=1}^m (q_i' - 1)  - 2 \sum_{i=1}^m [r_i] +  2\sum_{i=1}^m [r_i]_o.\\ 
&= \jx_P(n) -2\sum_{i=1}^m (q'_i -1) - 2\sum_{i=1}^m [r_i]_e.   
\intertext{Trivially we have} 
&\jx_K(n) = \jx_P(n) -2\frac{q_0'}{r_0[2]} + 2[r_0]_o-2\sum_{i=1}^m (r_i[2] -1) - 2\sum_{i=1}^m [r_i]_e.   
\end{align*} 
\item[(2)] $r_0 \not= 1/q_0$. 
In this case, we account for the degree change for the $\mathsf{TR}^+$-moves applied to $L^n_+$ in the same way as in case (1). It remains to account for the change to the degree based on applying $\mathsf{TR}_2^-$-moves with a final $\mathsf{TR_1^-}$-move to the $n$-cabled negative tangle of the special Montesinos knot $L$. 
Each application of the $\mathsf{TR_2^-}$-move adds $-(r_0[j+2] + r_0[j+1])n^2 - 2(r_0[j+1])n$ to the degree, and the final application of the $\mathsf{TR_1^-}$-move adds $-r_0[\ell_{r_0}] n^2 + 2(-r_0[\ell_{r_0}]-1)n$.
We sum the contribution over $j$ odd from 1 to $\ell_{r_0}$.
\begin{align}
& \sum_{j \text{ odd},  1\leq j <\ell_{r_0}} -(r_0[j+2] + r_0[j+1])n^2 - 2(r_0[j+1])n  \notag \\
& = -(r_0[2] + [r_0]- r_0[\ell_{r_0}])n^2   -2([r_0]_e - r_0[\ell_{r_0}]+ r_0[2]) n.  \label{e.c1} \\ 
\intertext{We compute similarly the quadratic growth rate and the linear growth rate of the final $\mathsf{TR}_1^-$-move: }
& -r_0[\ell_{r_0}] n^2 + 2(-r_0[\ell_{r_0}]-1)n.   \label{e.c2} \\ 
\intertext{Adding \eqref{e.c1}, \eqref{e.c2}, we have} 
& \eqref{e.c1}+ \eqref{e.c2} = -(r_0[2] + [r_0])n^2 - 2([r_0]_e+r_0[2] + 1)n. \label{e.c3}
\intertext{Plugging in $n-1$ for $n$ and expanding, the result is }
& \eqref{e.c1}+ \eqref{e.c2} = -(r_0[2] + [r_0])n^2 - 2([r_0]_e+r_0[2] + 1-(r_0[2] + [r_0]))n.
\intertext{When we add the coefficients multiplying $n^2$ and the coefficients multiplying $n$ from \eqref{e.c3} from the moves on $K^n_+$, we get in this case}
\js_K(n) &=\left( \js_P(n)-\omega(D_p) + \omega(D_K) + \sum_{i=1}^m (r_i[2]-1) + \sum_{i=1}^m [r_i] \right)-(r_0[2] + [r_0] ), 
\intertext{and }
\jx_K(n) &= \left(\jx_P(n) - 2\sum_{i=1}^m(r_i[2] -1) - 2\sum_{i=1}^m [r_i]_e \right)+ 2[r_0]_o -2\frac{q'_0}{r_0[2]}, 
\end{align} 
as in the statement of the theorem. 
\end{itemize}
\end{proof}

\section{Essential surfaces of Montesinos knots}
\label{sec.incompressible}

Let $\Sigma$ be a compact, connected, non boundary-parallel, and properly embedded surface in a compact, orientable 3-manifold $Y$ with torus boundary. We say that $\Sigma$ is \emph{essential} if the map on fundamental groups $\iota^*: \pi_1 (\Sigma) \rightarrow \pi_1 (Y)$ induced by inclusion of $\Sigma$ into $Y$ is injective. The surface $\Sigma$ is \emph{incompressible} if for each disk $D \subset Y$ with $D\cap \Sigma = \partial D$, there is a disk $D'\subset \Sigma$ with $\partial D' = \partial D$. The surface $\Sigma$ is called \emph{$\partial$-incompressible} if for each disk $D\subset Y$ with $D\cap \Sigma = \alpha$, $D\cap \partial Y = \beta$ ($\alpha$ and $\beta$ are arcs), $\alpha \cup \beta = \partial D$, and $\alpha \cap \beta = S^0$, there is a disk $D'\subset \Sigma$ with $\partial D' = \alpha' \cup \beta'$ such that $\alpha'  = \alpha$ and $\beta' \subset \partial \Sigma$.  

Orient the torus boundary $\partial Y$ with the choice of the canonical meridian-longitude basis $\mu$, $\lambda$ from the standard framing (so the linking number of the longitude and the knot is 0) given an orientation on the knot. The boundary curves $\partial \Sigma$ of an essential surface $\Sigma$ with boundary in $\partial Y$ are homologous and thus determines a homology class $[p \mu + q\lambda]$ in $H_1(\partial Y)$. The \emph{boundary slope} of $\Sigma$ is the fraction $p/q\in \mathbb{Q} \cup \{1/0 \} $, reduced to lowest terms. Hatcher showed that the set of boundary slopes of a compact orientable irreducible 3-manifold with torus boundary (in particular a knot exterior) is finite \cite{Ha}. 

An orientable surface is essential if and only if it is incompressible. On the other hand, a non-orientable surface is essential if and only if its orientable double cover in the ambient manifold is incompressible. In an irreducible orientable 3-manifold whose boundary consists of tori (such as a link complement), an orientable incompressible surface is either $\partial$-incompressible or a $\partial$-parallel annulus \cite{Wal}. Therefore, the problem of finding boundary slopes for Montesinos knots may be reduced to the problem of finding orientable incompressible and $\partial$-incompressible surfaces, and we will only consider such surfaces for the rest of the paper. 

In this section, we summarize the Hatcher-Oertel algorithm for finding all boundary slopes of Montesinos knots \cite{HO}, based on the classification of orientable incompressible and $\partial$-incompressible surfaces of rational (also known as 2-bridge) knots in \cite{HT}. For every Jones slope that we find in Sections \ref{sub.application} and \ref{ss.Mdegree}, we will use the algorithm to produce an orientable, incompressible and $\partial$-incompressible surface, whose boundary slope, number of boundary components, and Euler characteristic realize the strong slope conjecture. This completes the proof of Theorem \ref{thm.0} and Theorem \ref{thm.1}. 

We will follow the conventions of \cite{HO} and \cite{HT}. For further exposition of the algorithm, the reader may also consult \cite{IM}.  It will be useful to introduce the negative continued fraction 
expansion~\cite[Ch.13]{BS}
\begin{equation} 
\label{eq:negcfen} 
[[a_0, a_1, \ldots, a_\ell]] = [a_0, -a_1, \ldots, (-1)^\ell a_\ell] =
a_0-\cfrac{1}{a_1-\cfrac{1}{a_2-\cfrac{1}{a_3-\cdots -\cfrac{1}{a_\ell}}}} \,. 
\end{equation} 
with $a_i \in \mathbb{Z}$ and $a_i\not=0$ for $i>0$. 

\subsection{Incompressible and $\partial$-incompressible surfaces for a rational knot} \label{sec.rational}
A notion originally due to Haken~\cite{Haken}, a \emph{branched surface} $B$
in a 3-manifold $Y$ is a subspace locally modeled on the space as shown on the left in 
the Figure \ref{f.branch}. This means every point has a neighborhood diffeomorphic to the neighborhood of a point in the model space. A properly embedded surface $\Sigma$ in $Y$ is \emph{carried by $B$} if $\Sigma$ can be isotoped so that it runs nearly parallel to $B$, i.e., $S$ lies in a fibered regular neighborhood $N(B)$ of $B$, and such that $S$ meets every fiber of $N(B)$ \footnote{In \cite{FO84} and \cite{Ha}, a surface is carried by a branched surface if it lies in a fibered regular neighborhood of the branched surface. A surface is carried by a branched surface with \emph{positive weights} if in addition the surface intersects every fiber of the fibered regular neighborhood of the branched surface. Here we add the condition that the surface meets every fiber of the fibered regular neighborhood of the branched surface to simplify the summary of results from \cite{FO84}, \cite{Ha}.  }.

\begin{figure}[!htpb]
\def \svgwidth{.7\columnwidth}
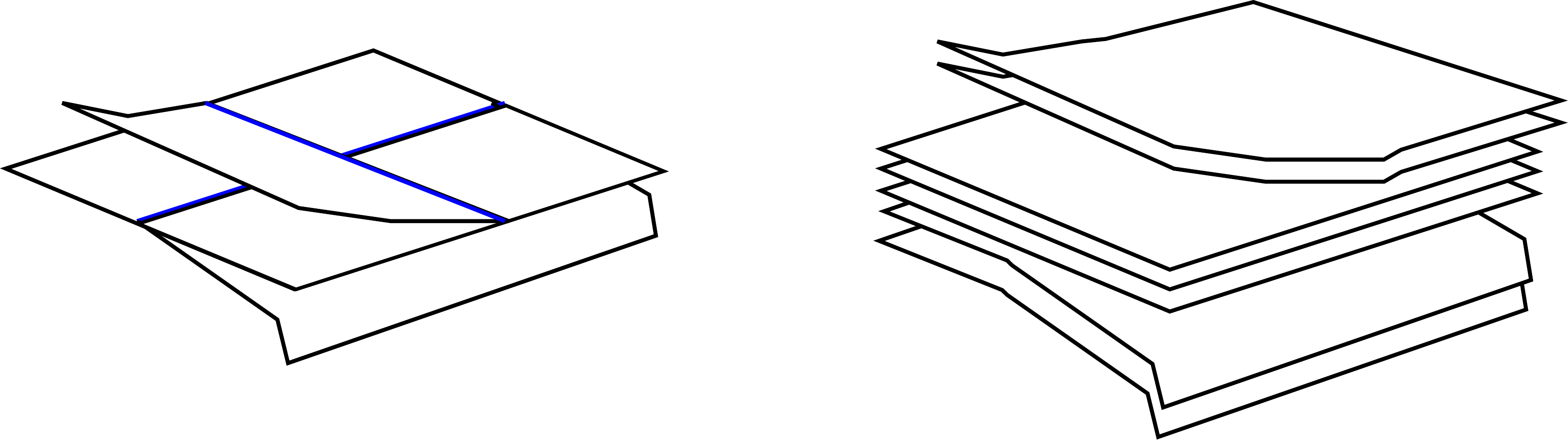
\caption{\label{f.branch} Left: local picture of a branched surface, with the blue lines 
indicating the singularities. Right: a surface carried by the branch surface.}
\end{figure}

Using branched surfaces, Hatcher and Thurston \cite{HT} classify all orientable,  incompressible and $\partial$-incompressible surfaces with nonempty boundary for a 
rational knot $K_r = K(1/r)$ where $r\in \mathbb{Q}\cup \{1/0\}$ in terms of negative continued 
fraction expansions of $r$. For each negative continued fraction expansion $[[b_0, b_1, \ldots, b_k]]$ of $r$ as in \eqref{eq:negcfen} they construct a branched surface $\Sigma(b_1, \ldots, b_k)$ and associated surfaces $S_M(M_1, \ldots, M_k)$ carried by $\Sigma(b_1, \ldots, b_k)$, where $M\geq 1$ and $0 \leq M_j \leq M$.

We will now describe their representation of a surface $S_M(M_1, \ldots, M_k)$ carried by a branched surface $\Sigma(b_1, \ldots, b_k)$ in terms of an \emph{edge-path} on a one-complex $\mathcal{D}$.  Here, $\mathcal{D}$ is the Farey ideal 
triangulation of $\mathbb{H}^2$ on which $\mathrm{PSL}_2(\mathbb{Z})$ is 
the group of orientation-preserving symmetries, see Figure \ref{fig:D}. 
Recall that the vertices (in the natural compactification)
of $\mathcal{D}$ are $\BQ \cup \infty$ and we set $\infty = \frac{1}{0}$ in projective coordinates. A typical vertex of $\mathcal{D}$ will be denoted by $\langle \frac{p}{q} \rangle$ for coprime integers 
$p,q$ with $q \geq 0$. There is an edge between two vertices $\langle \frac{p}{q} \rangle$ and $\langle \frac{r}{s} \rangle$, denoted by $\langle \frac{p}{q} \rangle \pdash \langle \frac{r}{s} \rangle$, whenever $|ps-rq| = 1$. An \emph{edge-path} is a path on the 1-skeleton of $\mathcal{D}$ which may have endpoints on an edge
rather than on a vertex.

Given a negative continued fraction expansion $[[b_0, \ldots, b_k]]$ of $r$, the vertices of the 
corresponding edge-path are the sequence of partial sums 
\[ 
[[b_0, b_1, \ldots, b_k]], [[b_0, b_1, \ldots, b_{k-1}]],  \ldots, [[b_0, b_1]], [[b_0]], \infty. 
\] 
Given a choice of integers $M\geq 1$ and $0\leq M_j \leq M$, we construct a surface $S_M(M_1, \ldots, M_k)$   in the exterior of $K_r$ from this edge-path as 
follows. We isotope the 2-bridge knot presentation of $K_r$ so that it 
lies in $S^2 \times [0, 1]$, with the two bridges intersecting $S^2\times \{1\}$ 
in two arcs of slope $\infty$, and the arcs of slope $r$ lying in 
$S^2 \times \{0\}$. See \cite[p.1 Fig. 1(b)]{HT}. The slope here is 
determined by the lift of those arcs to $\R^2$, where 
$S^2 \times \{i\} \setminus K$ is identified with the orbit space of $\Gamma$, 
the isometry group of $\R^2$ generated by $180^{\circ}$-degree rotation about 
the integer lattice points. 

Given an edge-path with vertices $\{ \la v \ra \}$, choose heights $\{i_v\}$, $i_v \in [0, 1]$ respecting the ordering of the vertices in the path. At $S^2\times\{0\}$, we have $2M$ arcs of slope $r$, and at $S^2\times \{1\}$ we have $2M$ arcs of slope $\infty$. 
For a fixed $M$, each vertex $\langle v \rangle$ of an edge-path determines a curve system on $S^2\times \{i_v\}$, consisting of $2M$ arcs of slope $v$ with ends on the four punctures representing the intersection with the knot. 
The surface $S_M(M_1, \ldots, M_k)$ is constructed by having its intersections with $S^2 \times \{ i_v \}$ coincide 
with the curve system at $\langle v \rangle$. Between one vertex $\langle v \rangle$ to another, say $\langle v' \rangle$ connected by an edge, $M$ saddles are added to change all $2M$ arcs of slope $v$ to $2M$ arcs of slope $v'$, with $M_j$ indicating one of the two possible choices of such saddles. At the end of the edge-path, $2M$ disks are added to the slope $\infty$ curve system, which corresponds to closing the knot by the two bridges.

\begin{figure}[!htpb] 
\def \svgwidth{.4\columnwidth}
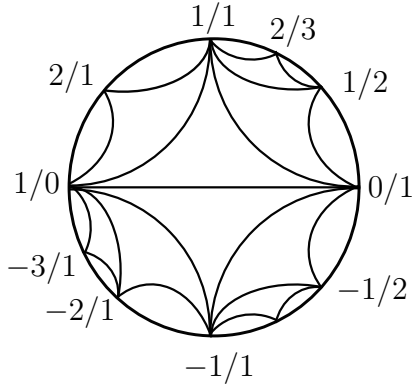
\caption{Some edges of the 1-complex $\mathcal{D}$. }
\label{fig:D} 
\end{figure} 

Hatcher and Thurston have  shown that every non-closed incompressible, $\partial$-incompressible surface in $S^3\setminus K_r$ is isotopic to $S_M(M_1, \ldots, M_k)$ for some $M$ and $M_j$'s . Furthermore, a surface $S_M(M_1, \ldots, M_k)$ carried by $\Sigma(b_1, \ldots, b_k)$ is incompressible and $\partial$-incompressible if and only if $|b_j| \geq 2$ for each $1\leq j\leq k$ \cite[Theorem 1.(b) and (c)]{HT}.  For more details on the construction of the branched surface $\Sigma (b_1, \ldots, b_k)$, and how it is used in the proof, see \cite{HT}. Floyd and Oertel have shown that there is a finite, constructible set of branched surfaces for every Haken 3-manifold with incompressible boundary, which carries all the two-sided, incompressible and $\partial$-incompressible surfaces \cite{FO84}. For the general theory of branched surfaces applied to the question of finding boundary slopes in a 3-manifold, interested readers may consult these references. We will continue to specialize to the case of  knots. 

Let $B$ be a branched surface in a 3-manifold $Y$ with torus boundary, and let $S$ be a properly embedded surface in $Y$ carried by $B$. 
There is an orientation on $\partial B$ such that all the boundary circles of $S$, oriented with the induced orientations from the orientation on $\partial B$, are homologous in the torus boundary of $Y$.  See \cite[p.375, Lemma]{Ha} for the full statement and a proof of this result that generalizes to the case where an orientable, compact, and irreducible 3-manifold has boundary the union of multiple tori.  Thus to compute a boundary slope it suffices to specify a branched surface, and hence the edge-path representing the surface as described above in the case of rational knots. This is how we will describe the surfaces we consider for computing the boundary slopes of a Montesinos knot for the rest of this paper. 

\subsection{Edge-paths and candidate surfaces for Montesinos knots} 
\label{subsec.hoalg}

Hatcher and Oertel \cite{HO} give an algorithm that provides a complete 
classification of boundary slopes of Montesinos knots by decomposing 
$K(r_0, r_1, \ldots, r_{m})$ via a system of Conway spheres 
$\{S^2_i\}_{i=1}^m$, each of which contains a rational tangle $T_{r_i}$. 
Their algorithm determines the conditions under which the incompressible and the $\partial$-incompressible  
surfaces in the complement of each rational tangle, as classified by 
\cite{HT} and put in the form in terms of edge-paths as discussed in Section \ref{sec.rational}, may be glued together across the system of Conway spheres to 
form an incompressible surface in $S^3 \setminus K(r_0, r_1, \ldots, r_{m})$.

To describe the algorithm, it is now necessary to give coordinates to 
curve systems on a Conway sphere. The 
curve system $S \cap S^2_i$ for a connected surface 
$S\subset S^3 \setminus K(r_0, r_1, \ldots, r_m)$ may be described by 
homological coordinates $A_i$, $B_i$, and $C_i$ as shown in 
Figure \ref{f.curvesystem} \cite{Hat88}. 

\begin{figure}[ht]
\def \svgwidth{.2\columnwidth}
\begingroup%
  \makeatletter%
  \providecommand\color[2][]{%
    \errmessage{(Inkscape) Color is used for the text in Inkscape, but the package 'color.sty' is not loaded}%
    \renewcommand\color[2][]{}%
  }%
  \providecommand\transparent[1]{%
    \errmessage{(Inkscape) Transparency is used (non-zero) for the text in Inkscape, but the package 'transparent.sty' is not loaded}%
    \renewcommand\transparent[1]{}%
  }%
  \providecommand\rotatebox[2]{#2}%
  \newcommand*\fsize{\dimexpr\f@size pt\relax}%
  \newcommand*\lineheight[1]{\fontsize{\fsize}{#1\fsize}\selectfont}%
  \ifx\svgwidth\undefined%
    \setlength{\unitlength}{269.59997222bp}%
    \ifx\svgscale\undefined%
      \relax%
    \else%
      \setlength{\unitlength}{\unitlength * \real{\svgscale}}%
    \fi%
  \else%
    \setlength{\unitlength}{\svgwidth}%
  \fi%
  \global\let\svgwidth\undefined%
  \global\let\svgscale\undefined%
  \makeatother%
  \begin{picture}(1,0.925816)%
    \lineheight{1}%
    \setlength\tabcolsep{0pt}%
    \put(0,0){\includegraphics[width=\unitlength,page=1]{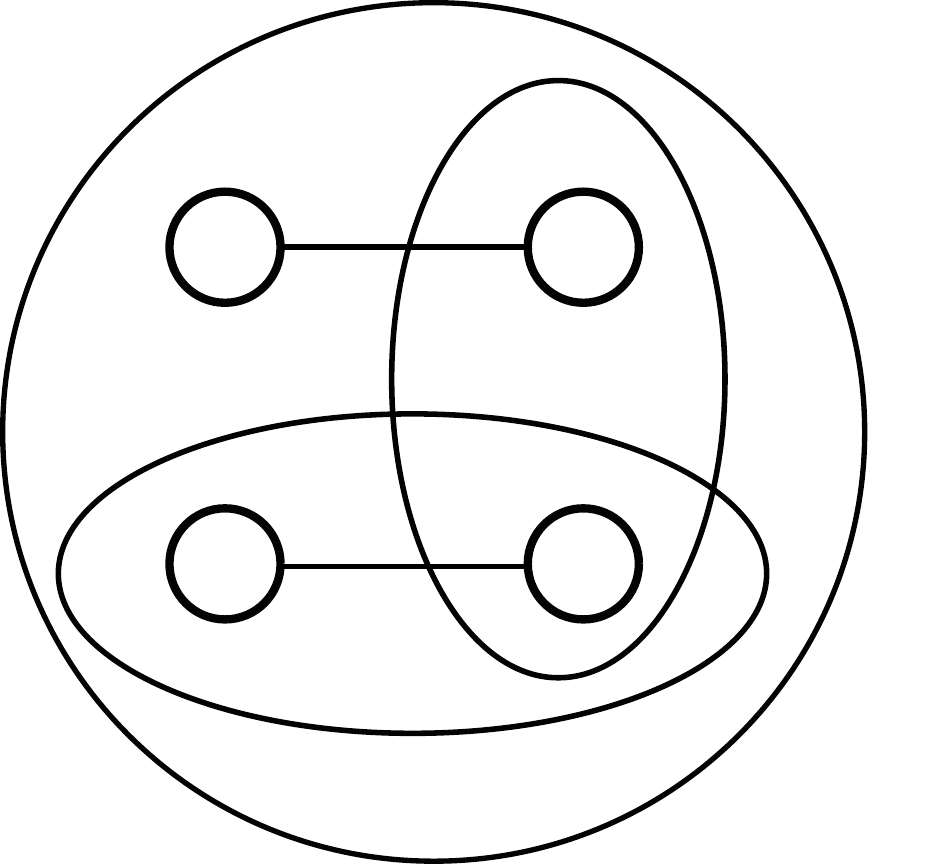}}%
    \put(0.27410983,0.55489615){\color[rgb]{0,0,0}\makebox(0,0)[lt]{\lineheight{0}\smash{\begin{tabular}[t]{l}$A_i$\end{tabular}}}}%
    \put(0.28411901,0.36798261){\color[rgb]{0,0,0}\makebox(0,0)[lt]{\lineheight{0}\smash{\begin{tabular}[t]{l}$A_i$\end{tabular}}}}%
    \put(0.3738873,0.04154292){\color[rgb]{0,0,0}\makebox(0,0)[lt]{\lineheight{0}\smash{\begin{tabular}[t]{l}$B_i$\end{tabular}}}}%
    \put(0.79123815,0.44814646){\color[rgb]{0,0,0}\makebox(0,0)[lt]{\lineheight{0}\smash{\begin{tabular}[t]{l}$C_i$\end{tabular}}}}%
  \end{picture}%
\endgroup%

\caption{\label{f.curvesystem}The Conway sphere containing the tangle 
corresponding to $r_i$ and the curve system on it.}
\end{figure}
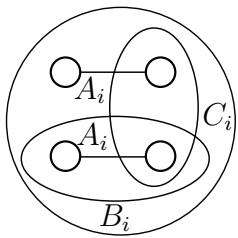 

Since an incompressible surface $S$ must also be incompressible when restricting to a tangle inside a Conway sphere, the classification of \cite{HT} applies, and the representation by Hatcher-Thurston of such a surface in terms of an edge-path also carries over. However, the edge-paths lie instead in an augmented 1-complex $\hat{\mathcal{D}}$ in the plane
obtained by splitting open $\mathcal{D}$ along the slope $\infty$ edge and adjoining constant edge-paths $\langle \frac{p}{q} \rangle \pdash \langle \frac{p}{q}\rangle$. See \cite[Fig. 1.3]{HO}. The additional edges in $\hat{\mathcal{D}}$ incorporate the new possibilities of curve systems that arise when gluing the surfaces following the tangle sum.

Again, an edge-path in $\hat{\mathcal{D}}$ is a path in the 1-skeleton of $\hat{\mathcal{D}}$ which may or may not end on a vertex. It describes a surface in the complement of a rational tangle in $K(r_0, r_1, \ldots, r_m)$ consisting of saddles joining curve systems corresponding to vertices, as described in the last paragraph of Section \ref{sec.rational}. The main adjustment is that the endpoint of an edge-path may not end at $\langle \infty \rangle$. In order for the endpoint representing the curve system to come from the intersection with an incompressible and $\partial$-incompressible surface, it must be on an edge
$\langle \frac{p}{q} \rangle \pdash \langle \frac{r}{s} \rangle$ and has the form 
\[
\frac{K}{M}\langle \frac{p}{q} \rangle 
+  \frac{M-K}{M} \langle\frac{r}{s}\rangle, 
\] 
for integers $K \in \mathbb{Z}$, $M > 0$. If $\frac{p}{q} \not= \frac{r}{s}$, this describes a curve system on a Conway sphere consisting of $K$ arcs of slope $p/q$, of $(A, B, C)$-coordinates $K(1, q-1, p)$, and $M-K$ arcs of slope $r/s$, of $(A, B, C)$-coordinates $(M-K)(1, s-1, r)$.  The coordinate of the point is the sum: $(M, K(q-1)+(M -K)(s-1), Kp+(M-K)r)$. If $\frac{p}{q} = \frac{r}{s}$, this describes a curve system on a Conway sphere consisting of $(M-K)$ arcs of slope $p/q$, of $(A, B, C)$-coordinates $(M-K)(1, q-1, p) = (M-K, (M-K)(q-1), (M-K)p)$,  and $K$ circles of slope $p/q$, of $(A, B, C)$-coordinates $K(0, q, p) = (0, Kq, Kp)$. The coordinate of the point is again the sum: $(M-K, (M-K)(q-1)+Kq, Mp)$.

The algorithm is as follows. 

\begin{enumerate}
\item 
For each fraction $r_i$, pick an edge-path $\gamma_i$ in the 1-complex
$\hat{\mathcal{D}}$ corresponding to a continued fraction expansion 
\[
r_i = [[b_0, b_1, \ldots, b_k]], b_j \in \mathbb{Z}, |b_j| \geq 2 \text{ for } 1\leq j \leq k. 
\]
As discussed in Section \ref{sec.rational}, these continued fraction expansions correspond to essential surfaces in the complement of the rational knot $K_{r_i}$. For example, for $1/3$ the choices are either $[[0, -3]]$ or $[[1, 2, 2]]$.  Or, choose the constant edge-path $\langle r_i \rangle \pdash \langle r_i \rangle$. 
\item 
For each edge $\langle \frac{p}{q}\rangle \pdash \langle \frac{r}{s} \rangle$ 
in $\gamma_i$, determine the integer parameters $\{K_i \}_{i=0}^m$, $\{M_i\}_{i=0}^m$ 
satisfying the following constraints.  
\begin{enumerate}
\item 
$A_i = A_j$ and $B_i = B_j$ for all the $A$-coordinates 
$A_i$ and the $B$-coordinates $B_i$ of the point 
\[ 
\frac{K_i}{M_i} \langle \frac{p}{q} \rangle + \frac{M_i-K_i}{M_i} 
\langle  \frac{r}{s} \rangle. 
\] 
\item 
$\sum_{i=0}^{m} C_i = 0$ where $C_i$ is the $C$-coordinate of the point 
\[  
\frac{K_i}{M_i} \langle \frac{p}{q} \rangle + \frac{M_i-K_i}{M_i} 
\langle  \frac{r}{s} \rangle.
\] 
\end{enumerate}
The edge-paths chosen in (1) with endpoints specified by the solutions to 
(a) and (b) of (2) determine a candidate edge-path system $\{\gamma_i\}_{i=0}^m$, 
corresponding to a connected and properly embedded surface $S$ in $S^3 \setminus 
K(r_0, r_1, \ldots, r_m)$. We call this the \emph{candidate surface} associated 
to a candidate edge-path system. 
\item 
Apply incompressibility criteria \cite[Prop. 2.1, Cor. 2.4, and Prop. 2.5-2.9]{HO} 
to determine if a candidate surface is an 
incompressible surface and thus actually gives a boundary slope.
\end{enumerate}

\begin{remark}
We would like to remark that Dunfield~\cite{Dunfield:table} has written a computer program implementing the Hatcher-Oertel algorithm, which will output the set of boundary slopes given a Montesinos knot and give other information like the set of edge-paths representing an incompressible, $\partial$-incompressible surface, Euler characterstic, number of sheets, etc. The program has provided most of the data we use in our examples in this paper. Interested readers may download the program at his website \url{https://faculty.math.illinois.edu/~nmd/montesinos/index.html}. 
\end{remark}

We will write $S(\{ \gamma_i\}_{i=0}^{m})$ to indicate a candidate 
surface associated to a candidate edge-path system $\{ \gamma_i\}_{i=0}^{m}$. 
Note that for a candidate edge-path system without constant edge-paths, $M_i$ is identical for 
$i=0, \ldots, m$ by condition (2a) in the algorithm. We will only consider this type of edge-path systems for the rest of this paper, and simply write $M$ for $M_i$ for a candidate surface $S$. 

We will mainly be applying \cite[Corollary 2.4]{HO}, which we restate here. 
Note that for an edge 
$\langle \frac{p}{q}\rangle \pdash \frac{K}{M}\langle \frac{p}{q} \rangle  + \frac{M-K}{M}\langle \frac{r}{s} \rangle$,
the $\rv$-value (called the ``$r$-value" in \cite{HO}) is 0 if $\frac{p}{q} = \frac{r}{s}$ or if the edge is 
vertical, and the $\rv$-value is $|q-s|$ when $\frac{p}{q} \not= \frac{r}{s}$. 

\begin{theorem}{\cite[Corollary 2.4]{HO}} \label{thm:incompressible}
A candidate surface $S(\{ \gamma_i\}_{i=0}^{m})$
is incompressible unless the cycle of $\rv$-values for 
the final edges of the $\gamma_i$'s is of one of the following types: 
$\{0, \rv_1, \ldots, \rv_{m} \}$, 
$\{1, 1, \ldots, 1, \rv_{m} \}$, or 
$\{1, \ldots, 1, 2, \rv_{m} \}$. 
\end{theorem}

\subsection{The boundary slope of a candidate surface} \label{subsec.bdslopecs}

The \emph{twist number} $\tw(S)$ for a candidate surface 
$S(\{ \gamma_i\}_{i=0}^{m})$ is defined as 
\be
\tw(S) :=  
\frac{2}{M} \sum_{i=0}^{m} (s^-_i - s^+_i) = 2\sum_{i=0}^m (e^-_i-e^+_i), 
\ee
where $s^-_i$ is the number of slope-decreasing saddles of $\gamma_i$,
$s^+_i$ is the number of slope-increasing saddles of $\gamma_i$, and $M$ 
is the number of sheets of $S$ \cite[p.460]{HO}. Let an edge be given by $\langle \frac{p}{q} \rangle \pdash \langle \frac{r}{s} \rangle$, we say that the edge decreases slope if $\frac{r}{s} < \frac{p}{q}$, and that the edge increases slope if $\frac{r}{s} > \frac{p}{q}$. In terms of edge-paths, $\tw(S)$ can be 
written in terms of the number 
$e^-_i$ of edges of $\gamma_i$ that decreases slope and $e^+_i$, the number 
of edges of $\gamma_i$ that increases slope as shown. 
If $\gamma_i$ has a final edge
$$  
\langle \frac{p}{q}\rangle \pdash \frac{K_i}{M}  \langle \frac{p}{q}  \rangle 
+ \frac{M-K_i}{M}   \langle \frac{r}{s}  \rangle, 
$$ 
then the final edge of $\gamma_i$ is called a fractional edge and counted as a fraction $\frac{M-K_i}{M}$. 
Finally, the boundary slope {$\bs(S)$ of a candidate surface $S$ is given by
\begin{equation} \label{eq.twist}
\bs(S) = \tw(S) - \tw(S_0) \,  
\end{equation} 
where $S_0$ is a Seifert surface that is a candidate surface from the 
Hatcher-Oertel algorithm. For the relation
of the twist numbers to boundary slopes, see~\cite[p.460]{HO}.  

\subsection{The Euler characteristic of a candidate surface} \label{subsec:echarjs}

We compute the Euler characteristic of a candidate surface $S = S(\{\gamma_i\}_{i=0}^m)$, where none of the $\gamma_i$'s are 
constant or end in $\la \infty \ra$ as follows. $M$ is again the number of sheets of the surface $S$. We begin with $2M$ disks which intersect $S^2_i\times \{0\}$ in slope 
$r_i$ arcs in each 3-ball $B_i^3$ containing the rational tangle corresponding to $r_i$. 

\begin{itemize}
\item 
From left to right in an edge-path $\gamma_i$, each non-fractional edge $\langle \frac{p}{q} \rangle \text{\pdash} 
\langle \frac{r}{s} \rangle$ is constructed by gluing $M$ 
saddles that change $2M$ arcs of slope $\frac{p}{q}$ (representing the intersections with $S_i^2\times \{i_{\frac{p}{q}}\}$) to slope 
$\frac{r}{s}$ (representing the intersections with $S_i^2\times \{i_{\frac{r}{s}}\}$), therefore decreasing the Euler characteristic  by $M$. 
\item 
A fractional final edge of $\gamma_i$ of the form $\langle \frac{p}{q} \rangle$\pdash 
$\frac{K}{M}\langle \frac{p}{q} \rangle + \frac{M-K}{M}\langle 
\frac{r}{s} \rangle$ changes $2(M-K)$ out of $2M$ arcs of slope 
$\frac{p}{q}$ to $2(M-K)$ arcs of slope $\frac{r}{s}$ via $M-K$ saddles, 
thereby decreasing the Euler characteristic by $M-K$. 
\end{itemize} 

This takes care of the individual contribution to the Euler characteristic of an edge-path 
$\{\gamma_i\}$. 
Now the identification of the surfaces on each of the 4-punctured spheres will also affect the Euler characteristic of the resulting surface. In terms of the common $(A, B, C)$-coordinates of each edge-path, there are two cases: 

\begin{itemize}
\item 
The identification of hemispheres between neighboring balls $B^3_i$  and 
$B^3_{i+1}$ identifies $2M$ arcs and $B_i$ half circles.  Thus it subtracts 
$2M+B_i$ from the Euler characteristic for each identification. 
\item 
The final step of identifying hemispheres from $B^3_0$ and $B^3_m$ on a single 
sphere adds $B_i$ to the Euler characteristic. 
\end{itemize}

\subsection{Matching the growth rate to topology for pretzel knots} 
\label{sec.jslopem}
We consider two candidate surfaces from the Hatcher-Oertel algorithm whose boundary slopes and the ratios of  Euler characteristic to the number of sheets will be shown to match the growth rate of the degree of the colored Jones polynomial from the previous sections as predicted by the strong slope conjecture.

\subsubsection{The surface $S(M, x^*)$} 
For $1\leq i \leq m$ write
\begin{equation} \label{e.x}
x^*_{i} = \frac{(q_i-1)^{-1}}{\sum_{j=1}^m (q_j-1)^{-1}}. 
\end{equation}   

The $x^*_{i}$'s come from the coefficients of $t$ in \eqref{eq.x*} in the real maximizers $x^*_i(t)$ of the degree function $\delta(n, k)$ from the state sum of the $n$th colored Jones polynomial.  Let $M$ be the least common multiple of the denominators of $\{x^*_{i}\}_{i=1}^m$, reduced to lowest terms.  For example, suppose we have the pretzel knot $P(-11, 7, 9)$, then 
\[ x^*_{1} = \frac{\frac{1}{7-1}}{\frac{1}{7-1} + \frac{1}{9-1}} = \frac{4}{7} \text{, }x^*_{2} = \frac{\frac{1}{8}}{\frac{1}{7-1} + \frac{1}{9-1}} = \frac{3}{7},  \]\
and $M$ is 7.  We show that $\{x_i^*\}$ and $M$ determine a candidate surface from the Hatcher-Oertel algorithm. 

Recall for $q = (q_0, q_1, \ldots, q_m)$, 
\[ s(q) =1 +q_0 +\frac{1}{\sum_{i=1}^m (q_i-1)^{-1}} .\] 
\begin{lemma} 
\label{l.HOm} Suppose $q = (q_0, q_1, \ldots, q_m)$ is such that $s(q) \leq 0$. 
There is a candidate surface 
$S(M, x^*)$ from the Hatcher-Oertel algorithm with $M>0$ sheets and 
$C$-coordinates \[ \{-M, Mx^*_{1}, Mx^*_{2}, \ldots, Mx^*_{m}\}.\]  
\end{lemma} 

\begin{proof}
 Directly from the proof of Lemma \ref{lem.prog1}, the elements of the set $\{x^*_i\}_{i=1}^m$ satisfy the following equations.
\begin{align}
x^*_{i}(q_i-1) &= x^*_{j}(q_j-1), \text{ for } i\not=j, \text{ and }\notag \\
\sum_{i=1}^{m} x^*_{i} &= 1. \label{eq.lins}
\end{align}

 Consider the edge-path systems determined by the following choice of continued fraction expansions for $\{1/q_i\}_{i=0}^m.$
\begin{align*}
1/q_0 &= [[-1, \underbrace{ -2, -2, \ldots, -2}_{-q_0-1}]], \text{and} \\
1/q_i &= [[0, -q_i]] \text{, \ for $1\leq i \leq m$.}\\ 
\end{align*} Note that they represent locally incompressible surfaces since $|-2|\geq 2$ and $q_i \geq 2$ for $1\leq i \leq m$ as discussed in Section \ref{sec.rational}. 
Let $K_i = Mx^*_{i}$ for $1\leq i \leq m$, and  suppose we have $0\leq K_0\leq M$,  $2\leq q \leq -q_0$ such that
\begin{equation} 
K_0 + M(q-2) = K_1(q_1-1).  \label{eq.HO1}
\end{equation}
This condition is the same as requiring $B_0 = B_1$ from $(a)$ of Step (2) of the Hatcher-Oertel algorithm. 
We specify a candidate surface $S(M, x^*)$ in terms of edge-paths $\{\gamma_i\}_{i=0}^m$: 

The edge-path $\gamma_0$ for $q_0$ is
\begin{align}
&\langle \frac{-1}{-q_0} \rangle \pdash \langle \frac{-1}{-q_0-1} \rangle \pdash \cdots \pdash \frac{K_0}{M} \langle\frac{-1}{q}\rangle + \frac{M-K_0}{M}\langle \frac{-1}{q-1} \rangle. \notag
\intertext{For $1\leq i \leq m$, we have the edge-path $\gamma_i$: }
&\langle \frac{1}{q_i} \rangle \pdash \frac{K_i}{M}\langle \frac{1}{q_i} \rangle  + \frac{M-K_i}{M} \langle \frac{0}{1}\rangle. \label{e.HOgluing}
\end{align}

Provided that $K_0, q$ satisfying \eqref{eq.HO1} exist, together with \eqref{eq.lins} this edge-path system satisfies the equations coming from (a) and (b) of Step (2) of the algorithm. Thus, there is a candidate surface with $\{-M, Mx^*_{1}, Mx^*_{2}, \ldots, Mx^*_{m}\}$ as the $C$-coordinates in the tangles corresponding to $r_i$'s. 

It remains to show that the assumption  $s(q)\leq 0$ implies the existence of $K_0, q$ satisfying \eqref{eq.HO1}. 

Write 
\[x^*_{i} = \frac{(q_i-1)^{-1}}{\sum_{j=1}^m (q_j-1)^{-1}} = \frac{\sum_{j=1}^m (q_j-1)}{(q_i-1)\Pi_{j=1}^m (q_j-1)}. \]
 Recall $s(q) \leq 0$  means 
\begin{align*} 
&1+q_0 + \frac{1}{\sum_{i=1}^m (q_i-1)^{-1}}  = 1+ q_0 + \frac{\sum_{i=1}^m (q_i-1)}{\Pi_{i=1}^m (q_i-1)} \leq 0. \\ 
\intertext{Multiply both sides by $M$, we get }
&M(1+q_0) +  M\frac{\sum_{i=1}^m (q_i-1)}{\Pi_{i=1}^m (q_i-1)} \leq 0.\\
\intertext{ This implies that a pair of integers $K_0$, $q$ such that $0\leq K_0 \leq M$, $q_0 \leq q \leq -2$ exist such that \eqref{eq.HO1} is satisfied, since by definition }
& M\frac{\sum_{i=1}^m (q_i-1)}{\Pi_{i=1}^m (q_i-1)} = Mx^*_1(q_1-1) =  K_1(q_1-1), 
\intertext{which is the same as saying}
& M(1+q_0) +  K_1(q_1-1) \leq 0.
\intertext{So if $M >  K_1(q_1-1)$, we can choose $q = 2$ and $K_0 = K_1(q_1 - 1)$. Otherwise, we choose some $q_0 \leq -q \leq -2$ such that}
& 0 \leq K_1(q_1-1) - M(q-2) \leq M. 
\end{align*} 
Let $K_0$ be the difference $ K_1(q_1-1) - M(q-2)$. 
\end{proof}

\paragraph{\bf{The twist number of $S(M, x^*)$}}
With the given edge-path system \eqref{e.HOgluing} in the proof of Lemma \ref{l.HOm} and applying the formula for computing the boundary slope in Section \ref{subsec.bdslopecs}, we compute the twist number of $S(M, x^*)$. 
For the edge-path $\gamma_0$ of $q_0$, since $q_0 < 0$, each edge of the edge-path is slope-decreasing. Similarly, each edge in $\gamma_i$ for $q_i$ is slope-decreasing (since $q_i > 0$, the edge $\left \langle 1/q_i \right \rangle \pdash  \left \langle 0/1 \right \rangle$ is decreasing in slope). Each non-fractional path contributes 1, and then the single fractional edge at the end contributes $(M-K_i)/M$ for $0\leq i \leq m$.  Thus
\begin{align} \label{tw:S} 
\frac{\tw(S(M, x^*))}{2} &=  \underbrace{(-q-q_0)}_{\text{contribution of the non-fractional edges of $\gamma_0$}} + \underbrace{\frac{M-K_0}{M}}_{\text{contribution of the single fractional edge at the end of $\gamma_0$}}  \notag \\ 
&+ \underbrace{\sum_{i=1}^{m} \frac{M-K_i}{M} }_{\text{contribution of the single fractional edge for each of the $\gamma_i$'s for $1\leq i \leq m$.}} \notag 
\intertext{By construction, $\sum_{i=1}^m \frac{K_i}{M} = 1$ and $-q  -\frac{K_0}{M} = -\frac{K_1}{M}(q_1-1)-2$ from \eqref{eq.HO1}, so} 
\tw(S(M, x^*)) &= 2(-q_0-x^*_1(q_1-1)+m-2). 
\end{align}

\paragraph{\bf{The Euler characteristic of $S(M, x^*)$}}
With the given edge-path system and applying the formula for computing the Euler characteristic in Section \ref{subsec:echarjs}, we compute the Euler characteristic over the number of sheets for $S(M, x^*)$. We start with $2M$ disks for each tangle. Each non-fractional edge of an edgepath in $\{\gamma_i\}_{i=0}^m$ subtracts $M$ from the Euler characteristic, while the final fractional edges subtract $\sum_{i=0}^m M-K_i$ from the Euler characteristic. 
At the final step of gluing surfaces across Conway spheres, we subtract $2M+B_i$ for each identification out of $m$ identifications, then add a single $B_i$ back.  We have, since $B_i = K_i(q_i-1)  = B_j$, 
\be \label{e.bi}
\sum_{i=1}^m B_i = m K_i(q_i - 1).  \ee
All together, the Euler characteristic over the number of sheets of $S(M, x^*)$ is given by
\begin{align} \label{echar:S}
\frac{2\chi(S(M, x^*))}{\#S(M, x^*)} &=2( \underbrace{\frac{2M(m+1)}{M}}_{\text{contribution from $2M$ disks in each tangle}} - \underbrace{\frac{(-q-q_0)M+ ( \sum_{i=0}^m M- K_i )}{M}}_{\text{contribution from the edges of the edge-path system}}  \\ \notag
&-\underbrace{\left( \sum_{i=1}^m \frac{(2M + B_i)}{M} \right)}_{\text{contribution from the $m$ identifications}}+ \underbrace{\frac{B_i}{M}}_{\text{contribution from the final identification}}) \\
\intertext{Using \eqref{eq.lins}, \eqref{tw:S},  and \eqref{e.bi} to simplify, we get}
\frac{2\chi(S(M, x^*))}{\#S(M, x^*)}  &= 4-\tw(S(M, x^*))-2(m-1)x^*_i(q_1-1) \notag \\ 
&= 4-2(-q_0-x^*_i(q_1-1)+m-2)  - 2(m-1)x^*_i(q_1 -1) \notag  \\ 
& = 8-2m +2q_0 -2(m-2) x^*_i(q_1-1). 
\end{align} 

\paragraph{\bf{The cycle of $\rv$-values of  $S(M, x^*)$}}
For $i = 0$, the last edge of the edge-path $\gamma_0$ is  
\[\langle \frac{-1}{q} \rangle \pdash \frac{K_0}{M} \langle\frac{-1}{q}\rangle + \frac{M-K_0}{M}\langle \frac{-1}{q-1} \rangle, \]
so the $\rv$-value for this edge-path is $|q-(q-1)| = 1$.   
For $1\leq i \leq m$, the final edge of the edge-path $\gamma_i$ is of the form 
\[\langle \frac{1}{q_i} \rangle \pdash \frac{K_i}{M}\langle \frac{1}{q_i} \rangle  + \frac{M-K_i}{M} \langle \frac{0}{1}\rangle.  \] 
So the value of each $1\leq i\leq m$ is $q_i - 1$ following the discussion preceding Theorem \ref{thm:incompressible}. 

The cycle of $\rv$-values for the edge-path system is $(1, q_1-1, q_2-1, \ldots, q_m-1)$. 
\subsubsection{The reference surface $R$}

Note that the sequence of parameters $(0)_{i=0}^m$ also trivially satisfy the equations from Step 2(a) and 2(b) of the Hatcher-Oertel algorithm with the choice of continued fraction expansion $1/q_i = [[0, -q_i]]$ for $0\leq i \leq m$, and therefore defines a connected candidate surface in the complement of $K(1/q_0, \ldots, 1/q_{m})$. We will call this surface the \emph{reference} surface $R$. 

In the framework of the Hatcher-Oertel algorithm, the edge-path 
corresponding to the reference surface has the following form for each $q_i$, $0 \leq i \leq m$:
$$
\langle \frac{1}{q_i} \rangle \ \pdash \ \langle 0 \rangle \,. 
$$ 

\paragraph{\bf{The twist number of $R$}}

With the exception of $\gamma_0$, which has a single slope-increasing edge (reading from left to right, the edge increases in slope from $1/q_0<0$ to $0$), each edge-path $\gamma_i$ is slope-decreasing (the edge decreases in slope from $1/q_i > 0$ to 0) of 
length 1, thus the twist number of the reference surface $R$ is
\begin{equation}
\label{eq.twR}
\tw(R) = 2(m-1) \,. 
\end{equation}
\paragraph{\bf{The Euler characteristic of $R$}}

The surface $R$ has $1$ sheet. The Euler characteristic is computed similarly as for $S(M, x^*)$, \eqref{echar:S}, except that there are only non-fractional edges in the edge-path system. 

\begin{align*} 
\frac{2\chi(R)}{\#R} &= 2(\frac{2M(m+1)}{M}-\sum_{i=0}^m \frac{M}{M} - \left(\sum_{i=1}^m \frac{2M+0}{M}\right)  + \frac{0}{M}) \\ 
&= 2(1-m).  
\end{align*}

\paragraph{\bf{The cycle of $\rv$-values of  $R$}}
The cycle of $\rv$-values of $R$ is $(-q_0 -1 , q_1-1, \ldots, q_m - 1)$.

\subsubsection{Matching the Jones slope}
The results of Section \ref{sub.application} applied to the class of pretzel knots we consider gives the degree of the $n$th colored Jones polynomial. We show that the quadratic growth rate with respect to $n$ matches the boundary slope of an incompressible surface. The claim is that the Jones slope is either realized by the surface $S(M, x^*)$ or the reference surface $R$ in Section \ref{sec.jslopem} depending on $s(q)$ and $s_1(q)$. Note that both $S(M, x^*)$ (if $s(q) \leq 0$) and $R$ are incompressible by an immediate application of Theorem \ref{thm:incompressible}, since $m \geq 2$ and $|q_i| >2$ for all $i$. By the Hatcher-Oertel algorithm, the reference surface is incompressible for a Montesinos knot except  $K(-\frac{1}{2}, \frac{1}{3}, \frac{1}{3}), K(-\frac{1}{2}, \frac{1}{3}, \frac{1}{4})$, and $K(-\frac{1}{2}, \frac{1}{3}, \frac{1}{5})$. 

\begin{lemma} 
\label{l.match} Suppose $s(q) \leq 0$. 
Let $R$ be the reference surface and 
$S(M, x^*)$ the surface by Lemma \ref{l.HOm} with boundary slopes $\bs(R)$ and $\bs(S(M, x^*))$, respectively. If 
$$ 
-2s(q)= \tw(S(M, x^*)) - \tw(R), 
$$
then $-2s(q)$ equals the boundary slope of the surface $S(M, x^*)$. 
\end{lemma}  

\begin{proof} 
Note that $R$ is a Seifert surface from the Hatcher-Oertel algorithm, so $\bs(R) = 0$,  and 
$$ \bs(S(M, x^*)) = \tw(S(M, x^*)) - \tw(R) $$
by \eqref{eq.twist}. 
\end{proof} 
\begin{theorem} \label{t.jslopem} Suppose $s(q) \leq 0$, we have: 
\[-2s(q) = \tw(S(M, x^*)) - \tw(R). \] 
\end{theorem} 

\begin{proof}
From Equations \eqref{tw:S} and \eqref{eq.twR} we have 
\begin{align*}
 \tw(S(M, x^*)) - \tw(R) &= 2(-q_0-x^*_1(q_1-1)+m-2)  - 2(m-1).  \\
 \intertext{By the definition of $x_i^*$, }
 \tw(S(M, x^*)) - \tw(R) &= -2(q_0 + \frac{(q_i-1)^{-1}}{\sum_{j=1}^m  (q_j -1)^{-1}} (q_1-1) +1) \\
 &= -2s(q). 
\end{align*} 
\end{proof} 

\subsubsection{Matching the Euler characteristic} \label{subsec:mechar}
Recall that for $q = (q_0, q_1, \ldots, q_m)$,  
\[ s_1(q) =  \frac{\sum_{i=1}^m (q_i + q_0 -2) (q_i -1)^{-1}}{\sum_{i=1}^m (q_i-1)^{-1}}.\] 
\begin{lemma} \label{l.jxm}
We have  
\[
-2s_1(q) + 4s(q) - 2(m-1)  = 2\frac{\chi(S(M, x^*))}{\#S(M, x^*)},
\]
where $\chi(S(M, x^*))$ is the Euler characteristic and $\#S(M, x^*)$ is the 
number of sheets $M$ of the surface $S(M, x^*)$. 
\end{lemma} 

\begin{proof}
We have by \eqref{echar:S} and substituting for $x_i^*$ by definition, 
\begin{align*} 
\frac{2\chi(S(M, x^*))}{\# S(M, x^*)} &= 8-2m +2q_0 -2(m-2) x^*_i(q_1-1) \\
&=  8-2m +2q_0 -2(m-2) \frac{(q_i-1)^{-1}}{\sum_{j=1}^m  (q_j -1)^{-1}} (q_1-1)\\
&= 8-2m +2q_0 -2(m-2) \frac{1}{\sum_{j=1}^m  (q_j -1)^{-1}}. \tag{*}
\intertext{On the other hand, also substituting $s(q)$, $s_1(q)$ by definition,}
-2s_1(q) + 4s(q) - 2(m-1)  &= -2(\frac{\sum_{i=1}^m (q_i + q_0 -2)(q_i-1)^{-1}}{\sum_{i=1}^m (q_i -1)^{-1}}) + 4(1+q_0+ \frac{1}{\sum_{j=1}^m  (q_j -1)^{-1}}) \\
&-2(m-1) \\ 
& = -2(\frac{m}{\sum_{j=1}^m  (q_j -1)^{-1}} + q_0 -1 ) + 4(1+q_0+ \frac{1}{\sum_{j=1}^m  (q_j -1)^{-1}}) \\ 
& -2(m-1). 
\end{align*} 
The last line is easily seen to be equal to (*) by expanding and gathering like terms. 
\end{proof}

\subsection{Proof of Theorem \ref{thm.0}} \label{subsec.jslopemp}
Now we prove Theorem \ref{thm.0}. Fix odd integers $q_0, \ldots, q_m$ with $q_0 < -1 < 1 < q_1, \ldots, q_m$. Let $P = P(q_0, \ldots, q_m)$ denote the pretzel knot.
By Theorem \ref{thm:incompressible},  both of the surfaces $S(M, x^*)$ (if $s(q) \leq 0$) and $R$  are incompressible by examining their edge-paths and computing their $\nabla$-values. Lemma \ref{l.match}, Theorem \ref{t.jslopem}, and Lemma \ref{l.jxm} show that $-2s(q) =  \bs(S(M, x^*))$ and $-2s_1(q) + 4s(q) - 2(m-1) = 2\frac{\chi(S(M, x^*))}{\#S(M, x^*)}$. For the reference surface $R$, it is immediate also from the Hatcher-Oertel algorithm that its boundary slope $\bs(R) = 0$ and  $2\frac{\chi(R)}{\#R} = -2(m-1)$. 

From Section \ref{sub.application},  we have the following cases for the degree of the colored Jones polynomial $J_{P, n}(v)$. The choice of the surface detected by the Jones slope swings between the surface $S(M, x^*)$ and the reference surface $R$.

{\bf Case 1:} $s(q) < 0$. We have that the maximum of $\delta(n, k)$ is given by \[  -2s(q)n^2 -2 s_1(q)n -2(m-1)n +(n^2+2n) \sum_{i=0}^m q_i+O(1), \] where recall that $s(q)$ and $s_1(q)$ are explicitly defined by \eqref{eq.sqq}. We see that $s(q)$ and $s_1(q)$ for any $n \gg 0$ are actually constant in $n$. The fact that $\js_P = -2s(q)= \bs(S(M, x^*))$ and $\jx_P  =-2s_1(q) + 4s(q) - 2(m-1) =2\frac{\chi(S(M, x^*))}{\#S(M, x^*)}$ (by considering $J_{K, n} = (-1)^{n-1}((-1)^{n-1}v)^{\omega(K)(n^2-1)}\langle K^{n-1} \rangle $) verifies the strong slope conjecture in this case.

{\bf Case 2:} $s(q) = 0$, $s_1(q) \not=0$. If $s_1(q) \geq 0$, the maximum of $\delta(n, k)$  has no quadratic term, but its  linear term is $-2(m-1)n$,  so the reference surface $R$ verifies the conjecture. If $s_1(q) < 0$, then the maximum 
\[ -2 s_1(q)n -2(m-1)n +(n^2+2n) \sum_{i=0}^m q_i+O(1)\] of $\delta(n, k)$ is found 
 at maximizers $\tau^*$ with parameters $n, k^*$, again all satisfying $n = k_0^* = k_1^*+\cdots + k_m^*$. Thus the surface $S(M, x^*)$ verifies the conjecture.

{\bf Case 3:} $s(q) > 0$. In this case the maximum of $\delta(n, k)$ also does not have quadratic term but has a linear term $-2(m-1)n$, and the reference surface $R$ verifies the conjecture.


\subsection{Matching the growth rate to topology for Montesinos knots}
\label{sec.jslopemm}

Let $K(r_0, \ldots, r_m)$ be a Montesinos knot satisfying the assumptions of Theorem \ref{thm.1}, and let $P(q_0, \ldots, q_m)$ be the associated pretzel knot. Similar to the case of pretzel knots, we define a surface  $S(M, x^*)$ where 

\begin{equation} \label{e.xm}
x^*_{i} = \frac{(q_i-1)^{-1}}{\sum_{j=1}^m (q_j-1)^{-1}}. 
\end{equation}   
 We give the explicit description of the surface in terms of an edge-path system from the Hatcher-Oertel algorithm below. We will see that these surfaces are built from extending the surfaces of the associated pretzel knots. 


\subsubsection{The surface $S(M,  x^*)$}

The edge-path system of $S(M, x^*)$ is described as follows. 
\begin{align}
\intertext{For $i=0$, say $r_0 =  [0, a_1, a_2, \ldots, a_{\ell_{r_0}}]$ the unique even length continued fraction expansion for 
$a_j < 0, 1\leq j \leq \ell_{r_0}$, we take the following continued fraction expansion}
r_0&= 
[[-1, \underbrace{-2, \ldots, -2}_{-a_1-1 \text{ times }}, a_2-1-1, 
\underbrace{-2, \ldots, -2}_{-a_3-1 \text{ times }}, a_{2j}-1-1, 
\underbrace{-2, \ldots, -2}_{-a_{2j+1}-1 \text{ times }}, \ldots, a_{\ell_{r_0}}-1]],  
\label{eq.negcfe} 
\intertext{with corresponding edge-path (reading backwards from the continued fraction expansion)}
&\left \langle [[-1, -2, \ldots, a_{\ell_{r_0}}-1]] \right \rangle 
\text{\pdash} \cdots \text{\pdash} \left \langle [[-1, -2, -2] ]
\right \rangle \text{\pdash} \left \langle [[-1, -2]]\right \rangle  
\text{\pdash} \left \langle -1 \right \rangle. \notag 
\intertext{For $1\leq i \leq m$, say $r_i 
= [0, a_1, a_2, \ldots, a_{\ell_{r_i}}]$ for $a_j>0$, $1\leq j \leq \ell_{r_0}$, we take the following 
continued fraction expansion }
r_i&= 
[[0, -a_1-1, \underbrace{-2, \ldots, -2}_{a_2-1 \text{ times }}, -a_3-1-1, 
\underbrace{-2, \ldots, -2}_{a_4-1 \text{ times }}, -a_{2j+1}-1-1,  
\underbrace{-2, \ldots, -2}_{a_{2j+2}-1 \text{ times }}, \ldots, \underbrace{-2, \ldots, -2}_{a_{\ell_{r_i}}-1 \text{ times }}]],  
\label{eq.poscfe} \\
\intertext{with corresponding edge-path (reading backwards from the continued fraction expansion)}
&\left \langle [[0, -a_1-1, \ldots, \underbrace{-2, \ldots, -2}_{a_{\ell_{r_i}}-1 \text{ times }}]] \right \rangle 
\text{\pdash} \cdots \text{\pdash} \left \langle 
[[0,-a_1-1, -2]]  \right \rangle
\text{\pdash} \left \langle [[0,-a_1-1]] \right  \rangle \text{\pdash}  
\left \langle 0 \right \rangle.  \notag
\end{align}  
We let $M$ be the least common multiple of the denominators of $\{x^*_{i}\}$.
We similarly have 
\begin{lemma} 
\label{l.HOm2} Let $q_0, q_1, \ldots, q_m$ be defined as they are for Theorem \ref{thm.1} for a Montesinos knot $K = K(r_0, \ldots, r_m)$. Suppose $q = (q_0, q_1, \ldots, q_m)$ is such that $s(q) \leq 0$. 
There is a candidate surface 
$S(M, x^*)$ for $K$ from the Hatcher-Oertel algorithm with $M$ sheets and 
$C$-coordinates \[ \{-M, Mx^*_{1}, Mx^*_{2}, \ldots, Mx^*_{m}\}.\]  
\end{lemma} 

\begin{proof}
Let $K_i = Mx^*_{i}$ for $1\leq i \leq m$, and  $0\leq K_0\leq M$, $2 \leq q \leq -q_0$ such that
\begin{equation} \label{eq.HO1m}
K_0 + M(q-2) = K_1(q_1-1), 
\end{equation}
We specify a candidate surface $S(M, x^*)$ in terms of edge-paths $\{\gamma_i\}_{i=0}^m$, by tacking onto the existing edge-path system for the associated pretzel knot $P(q_0, q_1, \ldots, q_m)$:  

The edge-path $\gamma_0$ for $r_0$ from \eqref{eq.negcfe} is
\begin{align*}
&\left \langle [[-1, -2, \ldots, a_{\ell_{r_0}}-1]] \right \rangle 
\text{\pdash} \cdots \pdash\langle \frac{-1}{-q_0} \rangle \pdash \langle \frac{-1}{-q_0-1} \rangle \pdash \cdots \pdash \frac{K_0}{M} \langle\frac{-1}{q}\rangle + \frac{M-K_0}{M}\langle \frac{-1}{q-1} \rangle.
\intertext{For $i\not=0$, we have the edge-path $\gamma_i$ from \eqref{eq.poscfe}: }
&\left \langle [[0, -a_1-1, \ldots, \underbrace{-2, \ldots, -2}_{a_{\ell_{r_i}}-1 \text{ times }}]]  \right \rangle 
\text{\pdash} \cdots \text{\pdash}\langle \frac{1}{q_i} \rangle \pdash \frac{K_i}{M}\langle \frac{1}{q_i} \rangle  + \frac{M-K_i}{M} \langle \frac{0}{1}\rangle. 
\end{align*}

Provided that $K_0, q$ satisfying \eqref{eq.HO1m} exist, this edge-path system satisfies the equations coming from (a) and (b) of Step (2) of the algorithm. We have already verified that $K_0, q$ exist when $s(q) \leq 0$ in Lemma \ref{l.HOm}. Thus, there is a candidate surface with $\{-M, Mx^*_{1}, Mx^*_{2}, \ldots, Mx^*_{m}\}$ as the $C$-coordinates in the tangle corresponding to $r_i$. 
\end{proof}

We also define a reference surface $R$ 
for $K(r_0, r_1, \ldots, r_m)$. 

\subsubsection{The reference surface $R$}

For the reference surface $R$, we have for each $r_i$, the edge-path system 
corresponding to the following  continued fraction expansion \\
For $r_0=[0, a_1, a_2, \ldots, a_{\ell_{r_0}} ]$ for $a_j < 0$, $1\leq j \leq \ell_{r_0}$, we take the following continued fraction expansion. 
\begin{equation} 
\label{eq.MRcfe}
r_0 = [[0, -a_1, a_2-1,\underbrace{-2, \ldots, -2}_{-a_3-1 \text{ times}}, a_4-1-1, \underbrace{-2, 
\ldots, -2}_{-a_5-1 \text{ times}}, a_{2j}-1-1, \underbrace{-2, \ldots, -2}_{-a_{2j+1}-1 \text{ times}}, \ldots, a_{\ell_{r_0}} -1]], 
\end{equation} 
with corresponding edge-path
\begin{equation*}
\left \langle [[0, -a_1, \ldots, a_{\ell_{r_0}} -1]] \right \rangle 
\text{\pdash} \cdots \text{\pdash} \left \langle [[0, -a_1]] \right \rangle \pdash \left \langle 0 \right \rangle.
\end{equation*}

For $1\leq i \leq m$, say $r_i 
= [0, a_1, a_2, \ldots, a_{\ell_{r_i}}]$ for $a_j>0$, $1\leq j \leq \ell_{r_i}$, we take the following 
continued fraction expansion.
\begin{equation} \label{eq.MRpcfe}
r_i = [[0, -a_1-1, \underbrace{-2, \ldots, -2}_{a_2-1 \text{ times }}, -a_3-1-1, 
\underbrace{-2, \ldots, -2}_{a_4-1 \text{ times }}, -a_{2j+1}-1-1,  
\underbrace{-2, \ldots, -2}_{a_{2j+2}-1 \text{ times }}, \ldots, \underbrace{-2, \ldots, -2}_{a_{\ell_{r_i}}-1 \text{ times }}]],  
\end{equation}
with corresponding edge-path
\begin{equation*} 
\left \langle [[0, -a_1-1, \ldots, \underbrace{-2, \ldots, -2}_{a_{\ell_{r_i}}-1 \text{ times }}]] \right \rangle 
\text{\pdash} \cdots \text{\pdash} \left \langle 
[[0,-a_1-1, -2]]  \right \rangle
\text{\pdash} \left \langle [[0,-a_1-1]] \right  \rangle \text{\pdash}  
\left \langle 0 \right \rangle. 
\end{equation*}

Again, both $R$ and $S(M, x^*)$ are incompressible by a direct application of Proposition \ref{thm:incompressible}. 

\subsection{Proof of Theorem \ref{thm.1}}
\label{subsec.jslopemmp}

Putting everything together we prove Theorem \ref{thm.1}.
\begin{proof} Let $K = K(r_0, \ldots, r_m)$. Recall $q = (q_0,\dots,q_m) \in \BZ^{m+1}$ denotes the
associated tuple of integers to $(r_0, \ldots, r_m)$ where $q_i=r_i[1]+1$ for $1\leq i \leq m$ and \[q_0 = \begin{cases} &r_0[1] -1 \text{ if $\ell_{r_0} = 2$ and $r_0[2]=-1$, } \\ 
 	& r_0[1] \text{ otherwise }  \end{cases} .\]  from the unique even length positive continued fraction expansions of $r_i$'s, and $q'_0$ is an integer that is defined to be 0 if $r_0 = 1/q_0$, and defined to be $r_0[2]$ otherwise.
Theorem \ref{t.Montesinoscjpd} gives $\js_K$ and $\jx_K$ in terms of the Jones slope $\js_P$ and the normalized Euler characteristic $\jx_P$ of the associated pretzel knot $P=P(q_0, q_1, \ldots, q_m)$.  Depending on the signs of $s(q)$ and $s_1(q)$ we have three cases by Theorem \ref{thm.0}. 
\begin{itemize}
\item[(1)] If $s(q)<0$, then 
\be
\js_P(n) = -2s(q), \qquad 
\jx_P(n) = -2s_1(q)+4s(q)-2(m-1). \notag
\ee 
\item[(2)] If $s(q)=0$, then 
\be 
\js_P(n) = 0, \qquad
\jx_P(n) =\begin{cases} -2(m-1) & \text{if} \,\, s_1(q) \geq 0 \\
-2s_1(q)-2(m-1) & \text{if} \,\, s_1(q) < 0
\end{cases} \,. \notag
\ee 
\item[(3)] If $s(q)>0$, then 
\be 
\js_P(n) = 0, \qquad
\jx_P(n) =-2(m-1).  \notag
\ee
\end{itemize} 

When $s(q) = 0$ and $s_1(q) \geq 0$, or $s(q) > 0$, applying Theorem \ref{t.Montesinoscjpd} we get
\[ \js_K(n) = -q'_0- [r_0] -\omega(D_P) + \omega(D_K)+ \sum_{i=1}^m (r_i[2]-1) +  \sum_{i=1}^m [ r_i ]   \]  
 and 
 \[ \jx_K(n) = -2(m-1)-2\frac{q'_0}{r_0[2]}+2 [r_0]_o -2 \sum_{i=1}^m (r_i[2]-1) -2 \sum_{i=1}^m [r_i ]_e. \] The reference surface $R$ is easily seen to verify the strong slope conjecture using \cite{FKP}, by viewing it as a \emph{state surface}. Since this material is well-known, we will briefly describe what a state surface is and indicate the state surface corresponding to the reference surface $R$.
 
A state surface from a Kauffman state $\sigma$ on a link diagram $D$ is a surface that comes from filling in the state circles of the $\sigma$-state graph $D_{\sigma}$ by disks and replacing the segments recording the original locations of the crossings by twisted bands. 

With the standard diagram that we are using for a Montesinos knot $K(r_0, \ldots, r_m)$ with $r_0 < 0 < r_1, \ldots, r_m$, the reference surface $R$ is the state surface that comes from the Kauffman state which chooses the $A$-resolution on the negative twist region $1/r_0[1]$ (or $1/(r_0[1]-1)$ if $r_0 = 1/q_0$) in the negative tangle corresponding to $r_0$, and the $B$-resolution everywhere else. Using \cite{FKP}, \cite{L} shows that

\[\bs(R) = -q'_0- [r_0] -\omega(D_P) + \omega(D_K)+ \sum_{i=1}^m (r_i[2]-1) +  \sum_{i=1}^m [ r_i ], \] and  
\[ 2\frac{\chi(R)}{\#R} = -2(m-1) -2\frac{q'_0}{r_0[2]}+2 [r_0]_o -2 \sum_{i=1}^m (r_i[2]-1) -2 \sum_{i=1}^m [r_i ]_e. \]
We use this fact to prove that $S(M, x^*)$ realizes the strong slope conjecture when the reference surface $R$ does not realize the Jones slope.

When $s(q) < 0$ or $s(q) = 0$ and $s_1(q) < 0$, the candidate surface $S(M, x^*)$ exists by Lemma \ref{l.HOm}. It suffices to 
verify that 
\[ 
\js_K - \bs(R) = \tw(S(M, x^*)) - \tw(R)
\] 
for the part of the strong slope conjecture concerning relationship of $\js_K$ to boundary slopes. This is because if the equation is true, then 
\begin{align*}
&\js_K - (\tw(R) - \tw(S_0)) = \tw(S(M, x^*)) - \tw(R), \\
\intertext{where $S_0$ is a Seifert surface from the Hatcher-Oertel algorithm, by \eqref{eq.twist}. Rearranging terms in the equation gives }
&\js_K = \tw(S(M, x^*)) - \tw(S_0) = \bs(S(M, x^*)). 
\end{align*} 

By Theorem \ref{t.Montesinoscjpd}, 
\[ 
\js_K - \underbrace{\left( -q'_0- [r_0] -\omega(D_P) + \omega(D_K)+ \sum_{i=1}^m (r_i[2]-1) +  \sum_{i=1}^m [ r_i ]\right)}_{\bs(R)} = \js_P.
\] 

 Notice that the edge-path systems of the two surfaces $S(M, x^*)$ (from \eqref{eq.negcfe}, \eqref{eq.poscfe}) and $R$ (from \eqref{eq.MRcfe}, \eqref{eq.MRpcfe}) coincide beyond the first segments of their edge-path systems, which define candidate surfaces $S_P(M, x^*)$ and the reference surface $R_P$ for the associated pretzel knot $P$. Now by Theorem \ref{t.jslopem},  we have 
\[\js_P  = \tw(S_P(M, x^*)) - \tw(R_P). \]
Since $S(M, x^*)$ and $R$ are identical beyond the first edges of their edge-path systems, we get
\[ \tw(S_P(M, x^*)) - \tw(R_P) = \tw(S(M, x^*)) - \tw(R), \]
and we are done.  
 
The proof that $\jx_K = \frac{2\chi(S(M, x^*))}{\#S(M, x^*)}$ is similar.
\end{proof}





\section*{Acknowledgments}

S.G. wishes to thank Shmuel Onn for enlightening conversations on quadratic
integer programming and for the structure of the lattice optimizers in 
Proposition~\ref{prop.QIP}. C.L. would like to thank the Max-Planck Institute for Mathematics in Bonn for excellent working conditions where the bulk of this work was conceived. 


\bibliographystyle{hamsalpha}
\bibliography{biblio}
\end{document}